%% file: GenFun.tex
\documentclass[reqno]{amsart}
\pdfoutput=1
\usepackage{amsrefs}
\usepackage{amsmath}
\usepackage{amssymb}
\usepackage{amsaddr}
\usepackage{bm}
\usepackage{stmaryrd}
\usepackage{calc}

\usepackage[all]{xy}
\usepackage{mathtools}
\usepackage{euscript}

\renewcommand{\epsilon}{\varepsilon} 

\newtheorem{theorem}{Theorem}[section]
\newtheorem{corollary}[theorem]{Corollary}
\newtheorem{lemma}[theorem]{Lemma}
\newtheorem{proposition}[theorem]{Proposition}

\theoremstyle{definition}
\newtheorem{defn}[theorem]{Definition}

\newtheorem{remark}[theorem]{Remark}

\usepackage{tikz}
\usetikzlibrary{arrows.meta}
\usetikzlibrary{patterns}
\usetikzlibrary{patterns.meta}
\usetikzlibrary{intersections}
\usepackage{pgfplots}
\pgfplotsset{compat=1.10}
\usepgflibrary{patterns} 
\usetikzlibrary{decorations.markings}

\tikzset{
  partial ellipse/.style args={#1:#2:#3}{
    insert path={+ (#1:#3) arc (#1:#2:#3)}
  }
}
\tikzset{
  partial ellipsecake/.style args={#1:#2:#3:#4}{
    insert path={+ (#1:#3) arc (#1:#2:#3 and #4) -- (0,0)  -- (#3,0)}
  }
}
\makeatletter
\tikzset{
  use path for main/.code={%
    \tikz@addmode{%
      \expandafter\pgfsyssoftpath@setcurrentpath\csname tikz@intersect@path@name@#1\endcsname
    }%
  },
  use path for actions/.code={%
    \expandafter\def\expandafter\tikz@preactions\expandafter{\tikz@preactions\expandafter\let\expandafter\tikz@actions@path\csname tikz@intersect@path@name@#1\endcsname}%
  },
  use path/.style={%
    use path for main=#1,
    use path for actions=#1,
  }
}
\makeatother

\DeclareMathOperator{\MP}{MP}

\DeclareMathOperator{\Wh}{Wh}

\DeclareMathOperator{\Gr}{Gr}
\DeclareMathOperator{\Diff}{Diff}
\DeclareMathOperator{\Map}{Map}

\DeclareMathOperator{\im}{im}
\DeclareMathOperator{\rt}{rt}

\newcommand{\Zp}[1]{\Z\!\!\>/\!\!\>#1}

\newcommand{\se}{\mbox{\kern 0.5pt \textrm{$\varoast$}\kern 0.5pt}}
\newcommand{\Ps}{\mathcal{P}}

\newcommand{\opsi}{\overline{\psi}}
\newcommand{\oeta}{\overline{\eta}}

\newcommand{\M}{\mathcal{M}}
\newcommand{\U}{\mathcal{U}}
\newcommand{\UN}{\mathcal{UN}}
\newcommand{\UNF}{\mathcal{UNF}}
\renewcommand{\H}{\mathcal{H}}

\newcommand{\cC}{\mathcal{C}}
\newcommand{\cR}{\mathcal{R}}
\newcommand{\hF}{\mathcal{HF}}
\newcommand{\nH}{(\mathcal{H}_{\R^k})}
\newcommand{\nT}{(\mathcal{T}_{\R^k})}
\newcommand{\nHF}{(\hF_{\R^k})}

\newcommand{\oQ}{\overline{Q} \mkern -12.7mu \phantom{\mathcal{U}}}
\newcommand{\tQ}{\tilde Q}

\newcommand{\oq}{\overline{q}}

\newcommand{\wH}{\widetilde{\mathcal{H}}}

\newcommand{\oW}{\overline{\mathcal{W}} \mkern -12.7mu \phantom{\mathcal{U}}}
\newcommand{\oU}{\overline{\mathcal{U}} \mkern -12.7mu \phantom{\mathcal{U}}}
\newcommand{\oC}{\mathcal{C}(X)}

\newcommand{\hU}{\tilde{\mathcal{U}}}
\newcommand{\tU}{\tilde{\mathcal{U}}}

\newcommand{\tW}{\tilde{\mathcal{W}}}

\renewcommand{\th}{\tilde h}
\DeclareMathOperator{\un}{un}

\newcommand{\osig}{\overline{\sigma}}
\newcommand{\usig}{\underline{\sigma}}

\DeclareMathOperator*{\colim}{colim}
\newcommand{\R}{\mathbb{R}}
\newcommand{\F}{\mathcal{F}}

\newcommand{\Ra}{\Rightarrow}
\newcommand{\C}{\mathbb{C}}
\newcommand{\Ss}{\mathbb{S}}
\newcommand{\oN}{\tilde{N}}
\newcommand{\Z}{\mathbb{Z}}

\newcommand{\T}{\mathcal{T}}

\newcommand{\pd}[2]{\frac{\partial #1}{\partial #2}}

\newcommand{\eh}{\tfrac{1}{2}}

\DeclareMathOperator{\inte}{Int}

\DeclareMathOperator{\hw}{hw}
\DeclareMathOperator{\Adj}{Adj}

\DeclareMathOperator{\id}{id}

\DeclarePairedDelimiter\absv{\lvert}{\rvert}
\DeclarePairedDelimiter\norm{\lVert}{\rVert}
\DeclarePairedDelimiter\inner{\langle}{\rangle}

\DeclarePairedDelimiter\pare{(}{)}

\newcommand{\tdd}[2]{\tfrac{\partial #1}{\partial #2}}

\numberwithin{equation}{section}

\begin{document}

\title{Generating Functions in $\R^{2n}$ and the Hatcher--Waldhausen Map}

\author[T. Kragh]{Thomas Kragh}
\address{
  Department of Mathematics, Uppsala University, Sweden \\
  ORCID: 0000-0003-2618-5712
}
\email{thomas.kragh@math.uu.se}

\begin{abstract}
  In this paper, we construct a generating function quadratic at infinity for any exact Lagrangian in $\R^{2n}$ that equals $\R^n$ outside a compact set. Such a Lagrangian may be viewed as a Lagrangian filling of the standard Legendrian unknot $S^{n-1}$ in $D^{2n}$. Generating functions of the type we construct are related to the space $\M_\infty$ considered by Eliashberg and Gromov. We also show that $\M_\infty$ is the homotopy fiber of the so-called Hatcher--Waldhausen map. This further relates the study of exact Lagrangians (and Legendrians) to algebraic K-theory of spaces. Using this and B\"okstedt's result that the Hatcher--Waldhausen map is a rational homotopy equivalence, we prove that the stable Lagrangian Gauss map (relative to the boundary) of the Lagrangian is null-homotopic.
\end{abstract}

\date{\today}


\maketitle

\tableofcontents

\section{Introduction}

In this paper, we describe how to construct a generating function (sometimes called a generating family) quadratic at infinity for any exact Lagrangian $L_\infty \subset \C^n=T^*\R^n$ that equals the zero section $\R^n$ outside a compact set. By scaling, we can assume that the nontrivial part of $L_\infty$ is completely contained in $T^*D^n$ and thus we may consider a compact $L\subset T^*D^n$ with boundary $S^{n-1}$. Alternatively, one can view this as an exact Lagrangian filling of the standard Legendrian unknot $S^{n-1} \subset D^{2n}$.

By adding a point at infinity to $L_\infty$, we also consider the exact Lagrangian submanifold $L^+$ in $T^*S^n$, which is essentially the same as adding a standard handle to the Lagrangian filling. It follows from the papers~\cite{Abou1} and~\cite{MySympfib} that $L^+$ is a homotopy sphere. We conclude that $L$ is in fact contractible, and this will be our starting point.

The main idea is to construct a Hamiltonian isotopy which takes the zero section $\R^n \subset \R^{2n}$ to a Lagrangian $K$ containing $L$. We then appeal to Chaperon's broken geodesic approach from~\cite{MR765426} to conclude that $K$ has a generating function quadratic at infinity (extended in~\cite{ref1} to cotangent bundles). We then use a cut-and-paste argument to construct one for $L$. This last step requires some control on the primitive of the Liouville form on $K$. Indeed, the cutting part will use fiberwise regular values of the generating function. The construction of the isotopy is carried out in Section~\ref{sec:an-infin-dimens}, and the cut-and-paste argument is carried out in Section~\ref{sec:cutting-paste-gener}.

Before formulating the precise theorem, we need to define which type of generating function we are considering. Let $N$ be a compact smooth manifold of dimension $n$, possibly with boundary. We will say that a Lagrangian $L \subset T^*N$ has \emph{compact support} if $L$ agrees with the zero section outside a compact set in $T^*N$ which projects to the interior of $N$. Similarly, we will say that a Hamiltonian is compactly supported if it has support contained in a compact set in $T^*N$ which projects to the interior of $N$.

Let $\pi_N: N\times \R^{2k} \to N$ denote the projection. Let $F:N \times \R^{2k} \to \R$ be a smooth function, and for each $x\in N$, let $F_x$ denote $F$ restricted to the fiber $\pi_N^{-1}(x)$. Let $\Sigma_F \subset N\times \R^{2k}$ be the zero set of the fiberwise differential $d_fF=\bigcup_{x\in N} d(F_x)$ of $F$. We have a canonical map
\begin{align*}
  d_h : \Sigma_F \to T^*N.
\end{align*}
Indeed, the projection to the base $N$ is given by $\pi_N$ above, and taking any tangent vector $v\in T_{\pi_N(z)}N$ we can lift it to a vector $w \in T_z(N\times \R^{2k})$ and define $d_h(z)(v) = dF(w)$. This is independent of the lift since the fiberwise differential $d_fF$ is $0$ at $z\in \Sigma_F$. By a \emph{generating function} for an immersed Lagrangian $L \to T^*N$ we mean a function as above such that
\begin{itemize}
\item[(G1)] the set $\Sigma_F$ is cut out transversely by the equation $d_fF=0$, and
\item[(G2)] the map $d_h : \Sigma_F \to T^*N$ factors through a diffeomorphism to $L$.
\end{itemize}
Note that the image of $d_h$ is always an exact Lagrangian immersion. Indeed, $F$ restricted to ${\Sigma_F}$ is a primitive for the pullback to $\Sigma_F$ of the canonical form on $T^*N$. Let $Q_k : \R^{2k} \to \R$ be the quadratic form
\begin{align} \label{eq:1Qk}
  Q_k(x,y) = -\norm{x}^2+\norm{y}^2
\end{align}
for $x,y \in \R^k$. We will refer to this as the \emph{standard quadratic} form. By abuse of notation we will also let $Q_k$ denote the composition
\begin{align*}
  N\times \R^{2k} \to \R^{2k} \xrightarrow{Q_k} \R.
\end{align*}
In this paper we will say that a generating function is \emph{quadratic at infinity} if
\begin{itemize}
\item[(G3)] the support of $F-Q_k$ is compact.
\end{itemize}
Thus, we allow the support to be nontrivial in fibers over boundary points, even when generating the zero section. The first result of this paper is the following.
\begin{theorem}\label{thm:1}
  Any exact Lagrangian $L\subset T^*D^n$ that agrees with the zero section over a neighborhood of the boundary $\partial D^n=S^{n-1}$ has a generating function quadratic at infinity.
\end{theorem}
This type of generating function quadratic at infinity is one of the most restrictive types of generating functions. However, in the case where $N$ has boundary, it is not the most restrictive type. In fact, we will use the terminology that a function $F: N\times \R^{2k} \to \R$ is quadratic at \emph{both} infinities if the support of $F-Q_k$ is compact and projects to the interior of $N$. Given a compactly supported exact Lagrangian $L\subset T^*N$, the difference between these two definitions leads us to consider the space $\M_\infty=\colim_{k\to \infty} \M_k$, where $\M_k$ is essentially the space of functions on $\R^{2k}$ equal to $Q_k$ at infinity and with only one critical point which has value 0 and is non-degenerate (this was defined by Eliashberg and Gromov in \cite{MR1732407}). This space essentially consists of the possible fiber functions generating the zero section, and any generating function quadratic at infinity generating the zero section close to $\partial N$ defines a map
\begin{align} \label{eq:GenFun:1}
  \partial N \to \M_\infty.
\end{align}
This map is given by taking the adjoint of the generating function (see Section~\ref{sec:lagrangian-gauss-map}).

\begin{remark}
  We will use the following only as motivation. If the map in Equation~(\ref{eq:GenFun:1}) extends to a map $N\to \M_\infty$ then one may modify the original generating function to be quadratic at both infinities by adding, fiberwise, the loop inverse\footnote{Theorem~\ref{thm:3} and \cite{MR1282230} show that $\M_\infty$ is an infinite loop space, but we will not need this.} of that extension over $N$. More interestingly, if the map does \emph{not} extend then the Lagrangian generated cannot be compactly Hamiltonian isotopic to the zero section. Indeed, any isotopy would, by \cite{ref1}, lead to a generating function of the zero section which is an extension of this map.
\end{remark}

In general, we would therefore like to understand the homotopy type of $\M_\infty$. In particular, in the case $N=D^n$, its homotopy groups are relevant. We will show that this space is closely related to the stable pseudo-isotopy group of a point $\Ps_\infty$ (see for example \cite{MR689392}) and its delooping $\H_\infty$, which is the space of stable $h$-cobordisms of a point. For this we will need to consider the $J$-homomorphism:
\begin{align*}
  J : O=\colim_{n\to \infty} O(n) \to F=\colim_{n\to \infty} \Map^{\simeq}_*(S^n)
\end{align*}
where $\Map_*^{\simeq}$ denotes the space of based self-homotopy equivalences with stabilization maps given by taking reduced suspension. The map $J$ is defined by the based action of $O(n)$ on the sphere $S^n\cong\R^n\cup \{\infty\}=(\R^n)^+$. As $J$ is a map of monoids (in fact, a map of infinite loop spaces), it has a delooping $BJ:BO\to BF$ which sends a vector space to its associated sphere bundle. The homotopy quotient of $J$, which is also the homotopy fiber of $BJ$, is denoted $F/O$. In this paper, we show that the Hatcher--Waldhausen map $\hw: F/O \to \H_\infty$ can be defined as follows. Define
\begin{align*}
  Q_k^l(x,y) = -\norm{x}^2+\norm{y}^2
\end{align*}
for $(x,y) \in \R^k \times \R^l$. Consider the set $\{Q_k^l \leq -1\}$ (red in Figure~\ref{Fig:hwpic}). For a vector space $V\subset \R^{k+l}$ we let $DV$ denote its disk and $SV$ its sphere.
\begin{figure}[ht]
  \begin{center}
    \begin{tikzpicture}
      \clip (3,1.41) -- (-3,1.41) -- (-3,-1.41) -- (3,-1.41) -- cycle;
      \draw[line width=0.25cm,pink] (-2,0) to[out=45,in=200] (-1,0.6) to[out=10,in=130] (0.5,0.6) to[out=-50,in=0] (0,0) to[out=180,in=145] (-0.5,-0.6) to[out=-35,in=180] (2,0);
      \draw                         (-2,0) to[out=45,in=200] (-1,0.6) to[out=10,in=130] (0.5,0.6) to[out=-50,in=0] (0,0) to[out=180,in=145] (-0.5,-0.6) to[out=-35,in=180] (2,0);
      \fill[red] (3,1.41) -- plot[domain=1.4:-1.4] ({sqrt((\x)^2+1)},\x) -- (3,-1.41) -- cycle;
      \fill[red] (-3,1.41) -- plot[domain=1.4:-1.4] ({-sqrt((\x)^2+1)},\x) -- (-3,-1.41) -- cycle;
      \fill[blue] plot[domain=1.4:-1.4] (\x,{-sqrt((\x)^2+1)});
      \fill[blue] plot[domain=1.4:-1.4] (\x,{sqrt((\x)^2+1)});
      \draw[dashed, very thick] (3,1.41) -- (-3,1.41) -- (-3,-1.41) -- (3,-1.41) -- cycle;
    \end{tikzpicture}
  \end{center}
  \caption{The image of $i$ (black) and the smoothing defining $\hw$ (white).} \label{Fig:hwpic}
\end{figure}
Consider the space of pairs $(V,i)$ where $V \in \Gr_k(\R^{k+l})$ and $i : DV \to \{-1\leq Q_k^l \leq 1\}$ is a smooth embedding such that
\begin{itemize}
\item $i$ is the standard inclusion of $V \subset \R^{k+l}$ in a neighborhood of $0$.
\item The interior of $DV$ is mapped to $\{-1<Q_k^l<1\}$.
\item $i$ maps $SV$ transversely to $\{Q_k^l=-1\} \simeq S^{k-1}$ and this map is a homotopy equivalence.
\end{itemize}
The image of such an $i$ is illustrated by the black curve between the red and blue regions in Figure~\ref{Fig:hwpic}. We prove in Section~\ref{sec:Identify-FO} that the colimit of such spaces of disks for $k,l\to \infty$ is a model for $F/O$. Heuristically, the idea is that increasing $l$ essentially makes the embedding and transversality conditions redundant, and the map from $SV$ to $\{Q_k^l=-1\}$ is a homotopy trivialization of the sphere bundle $SV$. Given such a pair $(V,i)$, we pick a tubular neighborhood $T$ around $\im i$ (pink in the figure) and realize $\hw$ at level $(k,l)$ as a smoothing of the boundary of $\{-1\leq Q_k^l \leq 1\}\setminus T$ (white in the figure). Indeed, this is essentially a compactly supported $h$-cobordism from $\{Q_k^l=1\}$ to the other part of the boundary. Stably this defines an element in $\H_\infty$. We give details of this in Section~\ref{sec:Hatcher-Wald-map}.

\begin{remark}
  As pointed out by a referee, the definition of the Hatcher--Waldhausen map may be somewhat ambiguous in the literature. We will use the definition of the Hatcher--Waldhausen map from Corollary 3.4 in \cite{MR686115}. In Section~\ref{sec:Hatcher-Wald-map} we relate the above description to this definition. We also note that the result used in Appendix~\ref{sec:homot-groups-m_infty} from \cite{MR1923990} uses this definition, and that the result in \cite{MR764574} does not depend on the specific definition, since potential variations of the map only concern torsion.
\end{remark}

In Sections~\ref{sec:flabby-map},~\ref{sec:param-handle-attachm},~\ref{sec:Identify-FO} and~\ref{sec:Hatcher-Wald-map} we prove the following theorem. In Section~\ref{sec:flabby-map} we even prove an interesting unstable version of the theorem.
\begin{theorem} \label{thm:3}
  There is a fibration sequence
  \begin{align}
    \M_\infty \xrightarrow{\oN_\infty} F/O \xrightarrow{\hw} \H_\infty.
  \end{align}
\end{theorem}
Here the map $\oN_\infty$ is given by taking the tangent space of the unstable manifold at the critical point together with a canonical lift through $F/O \to BO$ that we describe in Section~\ref{sec:Identify-FO}. The geometric idea for why the composition in the theorem is null-homotopic is as follows. We will give a description of $\oN_\infty$ as recording the unstable manifold disk of the unique critical point and its tangent space at the critical point, essentially defining a point in the space of pairs $(V,i)$ considered above. The description above of the Hatcher--Waldhausen map now implies that we are attaching a tubular neighborhood of this unstable disk to $\{Q_k^l\leq -1\}$. However, since a function in $\M_\infty$ has only one critical point, which is inside the tubular neighborhood, the resulting $h$-cobordism (the white part in the figure) has a function defined on it with no critical points (and is standard at infinity). Such a ``collar'' defines a path from this element to the basepoint in $\H_\infty$.

Any Lagrangian $L\subset T^*D^n$ has a Lagrangian Gauss map $L \to U(n)/O(n)$ given by sending each point to its tangent space, viewed as a linear Lagrangian in $\R^{2n}$. The stable version of this map has target $U/O=\colim_{n\to \infty} U(n)/O(n)$. Restricting this map to $S^{n-1}=\partial L$, it is trivial and thus defines an element in $[L/S^{n-1},U/O] \cong \pi_n(U/O)$. In Section~\ref{sec:lagrangian-gauss-map}, we relate this to the map $\M_\infty \xrightarrow{\oN_\infty} F/O\to BO$ in the theorem above. In \cite{MR764574} B\"okstedt proved that this composition is rationally trivial. We will use this to prove the following result.

\begin{theorem}\label{thm:4}
  With $L$ as in Theorem~\ref{thm:1} above, the stable Lagrangian Gauss map (relative to the boundary) $L/S^{n-1} \to U/O$ is null-homotopic.
\end{theorem}
It was already known from \cite{MR3465849} that this class had to lie in the kernel of $\pi_n(U/O) \cong \pi_{n-1}(\Z \times BO) \xrightarrow{\Z\times BJ_*} \pi_{n-1}(\Z\times BF)$. However, even though the image for $n=4k+1$ grows large with $k$, the kernel is never $0$ (except for $k=0$). Indeed, the homotopy groups of the image are finite except for $n=1$.

It is well-known that Gromov's $h$-principle for Lagrangian immersions tells us that the set of Lagrangian immersion classes of a given homotopy sphere into $T^*S^n$ is canonically identified with $\pi_n(U)$. This is a lift of the Gauss map under the canonical map $U\to U/O$, and since $\pi_{4k+1}(U) \to \pi_{4k+1}(U/O)$ is injective, the above theorem implies the following corollary.
\begin{corollary}\label{thm:4cor}
  For $n=4k+1$ the Lagrangian embedding $L^+ \subset T^*S^n$ represents the trivial immersion class $0\in \pi_n(U)$.
\end{corollary}
Note that this is not the same as saying that $L^+$ is isotopic through immersions to the zero section. Indeed, there is still the obstruction given by the possible exotic smooth structure on $L^+$.

It follows from Theorem~\ref{thm:3} and B\"okstedt's result that $\M_\infty$ has finite homotopy groups, but more refined information about the homotopy groups of $\M_\infty$ in low degrees can be inferred from \cite{MR1988283}, \cite{MR1923990}, \cite{MR3921317} and the recent preprint \cite{K8is0}. We carry out the details of this in Appendix~\ref{sec:homot-groups-m_infty} where we argue that the homotopy groups
\begin{table}[ht]
  \begin{center}
    \begin{tabular}{|l|cccccccccc|}
      \hline
      $n$ &0&1&2&3&4&5&6&7&8&9 \\
      \hline
      $\pi_n(\M_\infty)$ &0&0&0&$\Z/m_1$&0&$0$&$0$&$\Z/m_2$& $\Z/8$ or $\Z/2\oplus \Z/4$ &$\Z/3$ or 0 \\
      \hline
    \end{tabular}
  \end{center}
  \caption{Here the $m_i$ are unknown odd integers.}\label{tab:ind}
  \vspace{-0.55cm}
\end{table}
are as listed in Table~\ref{tab:ind}. This table implies that the generating functions constructed in this paper can be extended over the point at infinity to generate $L^+ \subset T^*S^n$ in the cases $n \in \{1,2,3,5,6,7\}$ (and possibly $n=10$).

\begin{remark}
  Work in progress by the present author and Eliashberg involves extending the injectivity theorem from \cite{MR1732407} to get nontrivial results about homotopy groups of spaces of Legendrians. This construction also relies heavily on algebraic K-theory of spaces, but it is not clear how this relates to the current construction (nor the newer, more general construction in \cite{MR4905540}). In fact, it seems to be of a different nature.
\end{remark}

\subsection{Acknowledgments}

I would like to thank Mohammed Abouzaid, Marcel B\"okstedt, Sylvain Courte, Tobias Ekholm, Yasha Eliashberg, S\o ren Galatius, and John Rognes for conversations about the material in this paper. I would also like to thank the anonymous referees for their many corrections and suggestions. In particular, I am grateful for the suggested simplification of the proof of Lemma~\ref{cor:Theorem1:1}.

\input{Action.tex}
\input{Theorem1.tex}
\input{Gauss.tex}
\input{Tubes.tex}
\input{Parhandles.tex}
\input{FmodO.tex}
\input{Hatcher.tex}

\appendix

\section{Proof of standard results} \label{sec:proof-stand-results}

In this appendix we include proofs of standard lemmas for which we could not find suitable references. Such references probably exist, but we include the proofs here for completeness.

\begin{proof}[Proof of Lemma~\ref{lem:FmodO:4}]
  Picking the pseudo-gradient is a contractible choice, so it suffices to show that the inclusion of the subspace $(\U^q)_k^l \subset \U_k^l$ defined by those $f$ which are equal to a quadratic form with eigenvalues $\pm1$ close to $0$ is a homotopy equivalence.
    
  For each $z\in \R^{k+l}$ we may consider the path moving the point $z$ in a straight line to $0$. Using isotopy extension we may extend the family of these paths to a family of compactly supported ambient isotopies $\psi_{z}^t : \R^{k+l} \to \R^{k+l}$ such that $\psi_0^t=\id$. Letting $z_f$ denote the unique critical point with critical value $0$ for $f\in \U_k^l$ we may consider the homotopy given by $(f,t) \mapsto f\circ \psi_{z_f}^t$. This defines a deformation retraction of $\U_k^l$ onto the subspace where the unique critical point is $0$.

  Then, similarly using a local isotopy fixing $0$ that stretches the eigenspaces of the Hessian $H_f$ one may adjust the function and deformation retract onto the subspace where the Hessian has eigenvalues $\pm1$.

  Now let $f$ be in this subspace. Denote the second order approximation of $f$ at $0$ by $p_2f$ (a quadratic form with eigenvalues $\pm1$). Let $X=\nabla f$ which we know to be transverse to $0$ at $0$. Let $\varphi : \R^{k+l} \to [0,1]$ be a bump function with compact support and equal to 1 close to 0. For each $a>0$ let $\varphi_a(z)=\varphi(az)$ and consider for large $a$ the function
  \begin{align*}
    f' = \varphi_a \cdot (p_2f) + (1-\varphi_a)\cdot f.
  \end{align*}
  Applying $X$ locally yields
  \begin{align*}
    X(f') = X(f) + X(\varphi_a)(p_2f-f) + \varphi_aX(p_2f-f) \geq \norm{X}^2 - ac_1c_2\norm{X}^4-c_3\norm{X}^3
  \end{align*}
  where $\norm{\nabla \varphi} \leq c_1$ and locally we have $\norm{p_2f-f} \leq c_2\norm{X}^3$ and $\absv{X(p_2f-f)} \leq c_3\norm{X}^3$. The function $f'$ equals $f$ outside a set $\{a\norm{X} \leq c_4\}$ for some $c_4>0$. It follows that for $a$ sufficiently large, $X$ is a strict pseudo-gradient for $f'$.
  
  The bounds on $a$ only depend on the first few derivatives of $f$ in a neighborhood of $0$, so $a$ can be chosen continuously depending on $f$ in the $C^\infty$ topology. Interpolating from $f$ to $f'$ thus defines a homotopy from the identity to a map landing in $(\U^q)_k^l$. As this homotopy preserves $(\U^q)_k^l$ as a set it follows that the inclusion is a homotopy equivalence.
\end{proof}

\begin{proof}[Proof of Lemma~\ref{lem:GenFun:2}]
  Consider the homotopy-commutative diagram
  \begin{align*}
    \xymatrix{
      \H_{\textrm{sm}}(M) \ar[d] \ar[r] & \H_{\textrm{sm}}(N) \ar[d] \\
      \H_{\textrm{sm}}(M\times I) \ar[r] & \H_{\textrm{sm}}(N\times I)
    }
  \end{align*}
  where the vertical arrows are stabilizations (it does not matter which of the two stabilizations is used). By Igusa's theorem (see \cite{MR972368}) these vertical arrows are at least as connected as the lemma requires. Taking the colimit over stabilizations, we get, by Waldhausen's theorem, a map that is $(n-2)$-connected. Indeed, the induced map on algebraic $K$-theory $K(\Omega M) \to K(\Omega N)$ is at least $(n-1)$-connected.
\end{proof}

\begin{proof}[Proof of Lemma~\ref{lem:GenFun:4}]
  Given $P \subset M$, the $h$-cobordism $W$ from $\partial M$ to $P$ need not be an $s$-cobordism. Indeed, it may have nontrivial simple homotopy type relative to $\partial M$. If it is not, we can add some 1 and 2 handles to $W$ (inside $M$) and still conclude that there is at least one collar $c: \partial M \times [0,1) \to M$ which contains all of $W$.

  We now claim that, fixing a $P \subset M$, the space of collars containing $P$ in its image is contractible.

  Proof of claim: Let $(c_k)_{k \in K}$ be a compact family of such collars. Using the existence construction above we can pick another such collar $c$ such that the image of $c$ is contained inside all the images $\im c_k$ for $k\in K$. Indeed, we may carry out the attaching in the construction above inside any small neighborhood of $P$. Now, pick $t\in(0,1)$ large enough for $P \subset c_{\mid \partial M \times [0,t)}$. Pick a family of embeddings $\varphi_{(k,u)} : \partial M \times [0,t] \to \partial M \times [0,1)$ for $u\in [0,1]$ which starts at $(m,v) \mapsto c_k^{-1}(c(m,v))$ and ends at the standard inclusion (this can be done because the space of collars is contractible). By isotopy extension we can extend this family of embeddings to a family of diffeomorphisms $\Phi_{(k,u)} : \partial M \times [0,1) \to \partial M\times [0,1)$. Then reparametrizing $c_k$ by this isotopy creates a homotopy of collars from the family $(c_k)$ to a family $(b_k)$ which agrees with $c$ on $\partial M \times [0,t]$. Note that, as these are reparameterizations, they contain $P$ throughout. Now the family of collars $u \in [t,1]$, given by
  \begin{align*}
    (k,m,v) \mapsto b_k(m,uv)
  \end{align*}
  shows that the family is null-homotopic, while still containing $P$. Thus the claim is proven.

  We thus define a larger space $S$ consisting of pairs $(P,c)$ where $P \in  H_\partial(M)$ and $c$ is a collar containing $P$. The forgetful map $S \to H_\partial(M)$ is by the above claim a homotopy equivalence. The inclusion $H_{\textrm{sm}}(\partial M) \to H_\partial(M)$ was defined to lie in one such given collar, so we have a canonical lift to $S$. Now as the space of collars is contractible we get a deformation retraction of $S$ onto the image of this inclusion.
\end{proof}

\section{Homotopy groups of $\M_\infty$} \label{sec:homot-groups-m_infty}

In \cite{MR1988283} and \cite{MR1923990} Rognes computed many of the homotopy groups of $\H_\infty$. In \cite{MR3921317} Blumberg and Mandel improved these calculations. Using their Table 1 together with the preprint \cite{K8is0}, which shows that $K_8(\Z)=0$, we can extract the groups listed in Table~\ref{tab:1}.
\begin{table}[h!]
  \begin{center}
    \begin{tabular}{|l|cccccccccc|}
      \hline
      $n$ &$\leq 1$&2&3&4&5&6&7&8&9&10\\
      \hline
      $\pi_n(F/O)$ &0& \makebox[0.1cm]{$\Zp{2}$} &0&$\Z$& 0 &\makebox[0.2cm]{$\Zp{2}$}& 0 &$\Z\oplus \Zp{2}$ & $(\Zp{2})^2 $ & $\Zp{6}$   \\
      $\pi_n(\H_\infty)$ &0&\makebox[0.2cm]{$\Zp{2}$} &0&$\Z$& 0 &\makebox[0.2cm]{$\Zp{2}$}& 0 &$\Z\oplus \Zp{2}$ & $(\Zp{2})^2 \oplus \Zp{8}$ & $\Zp{6}$   \\
      $\pi_{n-1}(\M_\infty)$ &0&0&0&\makebox[0.5cm]{$\Zp{m_1}$}&0&$0$&$0$&\makebox[0.5cm]{$\Zp{m_2}$}& $\Zp{8}$ or $\Zp{2}\oplus \Zp{4}$ &$\Zp{3}$ or 0 \\
      \hline
    \end{tabular}
  \end{center}
  \caption{Homotopy groups of $\H_\infty$, $F/O$ and $\M_\infty$.}\label{tab:1}
\end{table}
Indeed, the second column in their table is the sphere-spectrum factor in the identification $\Ss \vee \Wh^{\Diff}(*) \simeq K(\Ss)$, and the table consists of the remaining columns combined and shifted by 1 since $\H_\infty \simeq \Omega \Wh^{\Diff}(*)$.

These are very similar to the well-known homotopy groups of $F/O$ also listed in Table~\ref{tab:1}. Indeed, the only difference is that $\pi_9(F/O) \cong (\Z/2)^2$ is missing the $\Z/8$ summand. The homotopy groups of $\M_\infty$ listed above and in Table~\ref{tab:ind} now follow from the long exact sequence of homotopy groups associated to the fibration in Theorem~\ref{thm:3} and Theorem 7.5 in \cite{MR1923990}, which states that the Hatcher--Waldhausen map is a 2-primary equivalence in degrees less than $8$ and injective in degrees less than 13.

\bibliographystyle{plain}
\bibliography{/home/thomas/Dropbox/Mybib}

\end{document}

%% file: Action.tex
\section{A Hamiltonian moving a part of the zero section to $L$} \label{sec:an-infin-dimens}

Let $L \subset D^{2n} \subset T^*D^n$ be an exact Lagrangian submanifold equal to the zero section near the boundary of the disk (so $\partial L = S^{n-1}$). In this section we construct a compactly supported Hamiltonian flow on $\R^{2n}=T^*\R^n$ which moves the zero section $\R^n$ to a Lagrangian $K$ that contains $L$, and such that the other parts of the image that land in $T^*D^n$ have primitive values much lower than those on $L$. Parts of this section only work for $n \geq 5$. In Section~\ref{sec:extending-nleq-4} we will describe modifications to the arguments for $n\leq 4$.

Let $L_\infty=L\cup (\R^n-D^n)$ be the extension by the zero section in $\R^{2n}$. In this section we will denote coordinates in (this copy of) $\R^{2n}$ by $z=(x,y)$ with $x,y\in \R^n$. Let $x_0=(-1,0,\dots,0)\in \R^n$ and let $f_{L_\infty}:L_\infty \to \R$ be the unique primitive of the standard Liouville form $ydx$ that is compactly supported (equivalently $f_{L_\infty}(x_0)=0$).

\begin{lemma}\label{lem:Action:4}
  For $n\geq 5$ there is a diffeomorphism $L_\infty \cong \R^n$ which outside $D_2^n$ is given by $x \mapsto \norm{x} \varphi(\hat x)$ where $\varphi : S^{n-1} \cong S^{n-1}$ is a diffeomorphism and $\hat x$ is the normalization of $x$. Furthermore, we can assume that in a neighborhood of $x_0$ it is given by $x \mapsto x-x_0$. In particular, it maps $x_0$ to $0$.
\end{lemma}

We will call a diffeomorphism $\R^n\to\R^n$ \emph{radial} at $x' \in \R^n$ if it is of the form $x \mapsto \norm{x} \varphi(\hat x)$ (as in the lemma) in a neighborhood of $x'$. 

\begin{proof}
  Recall from the introduction that $L_\infty$ is contractible and hence $L$ is a smooth homotopy ball with standard boundary $S^{n-1}$. Using the $h$-cobordism theorem (see Milnor \cite{MR0190942}), we can find an orientation-preserving diffeomorphism $\varphi: L \cong D^n$ which, restricted to the boundary, is a possibly exotic diffeomorphism $\varphi_{\mid S^{n-1}}:S^{n-1} \cong S^{n-1}$. We may adjust $\varphi$ to be radial in a neighborhood of the boundary, so that we may extend it radially to identify all of $L_\infty$ with $\R^n$. It is thus of the desired form outside $D_2^n$. We can now post-compose with an ambient isotopy of $\R^n$ with support in $D_2^n$ to make it of the form $x\mapsto x-x_0$ in a neighborhood of $x_0$.
\end{proof}

We use such a diffeomorphism to induce a flat metric on $L_\infty$ (when $n\geq 5$) and to induce symplectic coordinates on $T^*L_\infty$, which we denote by $(q,p) \in \R^n\times \R^n$. By the translation property in the lemma, we have $(x,y)=(q+x_0,p)$ in a neighborhood of the fiber $T^*_{x_0}\R^n$, and the Riemannian structures agree.

Fix a $\delta<1$ such that for some $\delta'<\delta$ the Lagrangian $L_\infty$ equals the zero section outside $D_{\delta'}^{2n}$. We define $L_\delta = L \cap D_\delta^{2n}$, which is in the interior of $L$. We will denote the \emph{closed} unit disk and sphere bundles of a cotangent bundle $T^*M$ by $D^*M$ and $S^*M$ respectively, and we will denote the disks and sphere bundles of radius $r$ by $D_r^*M$ and $S_r^*M$ respectively.

\begin{lemma} \label{lem:Action:1}
  For any Riemannian structure on $L_\infty$ induced by a radial diffeomorphism $\varphi : L_\infty \cong \R^n$ there exists an $\epsilon>0$ such that the embedding $L_\infty \subset \R^{2n}$ extends to a symplectic embedding
  \begin{align} \label{eq:2}
    D_{\epsilon}^*L_\infty \subset D^*\R^n
  \end{align}
  which over $L_\infty-L_\delta$ is induced by the identification $L_\infty-L_\delta = \R^n-D_\delta^n$ (equal as subsets inside the zero section $\iota:\R^n \subset \R^{2n}$), and hence respects fibers. Furthermore, we may assume that this is the only part of the image landing in $T^*(\R^n-D_\delta^n)$.
\end{lemma}

Note that even though this respects fibers outside $D_\delta^n$, the Riemannian structures may not agree---hence the image may not be a disk bundle in the target structure. However, we made sure that in a neighborhood of $x_0$ these do in fact agree.

\begin{proof}
  With $\delta'<\delta$ as above we first define $L_{\beta} = L\cap D_{\beta}^{2n}$ for some $\beta \in (\delta',\delta)$ and $L_\beta^c=\overline{L_\infty-L_\beta} \subset \R^n$. We thus have the canonical symplectic identification $T^*L_\beta^c = T^*(\overline{\R^n-D_\beta^n})$ which extends the embedding over $L_\beta^c$ to $\psi : D_\epsilon^*L_\beta^c \to \R^{2n}$. As the diffeomorphism $L_\infty \to \R^n$ inducing the Riemannian structure is radial outside $D_2^n$, its derivative has a global bound, and it follows that for small enough $\epsilon$ this lands in $D^*\R^n$.
  
  We may extend the domain to include a small symplectic neighborhood of $L_\beta \subset T^*L_\beta$ by using the Darboux-Weinstein symplectic neighborhood theorem on the map
  \begin{align*}
    \psi \cup \iota : (D_\epsilon^* L_\beta^c)\cup L_\beta \to \R^{2n}.
  \end{align*}
  By compactness of $L_\beta$ there is a possibly smaller $\epsilon>0$ such that this defines an embedding $D_\epsilon^*L_\infty \subset D^*\R^n$. By the choice of $\delta>\beta$ we may shrink $\epsilon$ a bit more to make sure that the image of $D_\epsilon^*L_\beta$ avoids the parts over $\R^n-L_\delta$.
\end{proof}

Consider a non-compactly supported Hamiltonian of the form
\begin{align*}
  H(x,y) = f(\norm{(x,y)})
\end{align*}
defined on all of $\R^{2n}$.

\begin{lemma} \label{lem:Action:2}
  For any smooth function $\theta: [0,\infty) \to \R$ with $\theta(r)$ constant for $r$ close to 0, there exists an $f: [0,\infty) \to \R$ such that $H(x,y)=f(\norm{(x,y)})$ is smooth and its Hamiltonian time-1 flow is given by $z \mapsto e^{i\theta(\norm{z})}z$.
\end{lemma}

We will refer to such a Hamiltonian isotopy as a \emph{generalized rotation}.

\begin{proof}
  A direct computation shows that the Hamiltonian vector field of such an $H$ is given by:
  \begin{align*}
    (X_H)_z = \frac{f'(\norm{z})}{\norm{z}}(J_0 z)
  \end{align*}
  whose time-1 flow is given by $z\mapsto e^{i\frac{f'(\norm{z})}{\norm{z}}}z$. Given $\theta$ as in the lemma, we can solve for $f$ so that this factor is equal to $\theta(r)$ by setting
  \begin{align*}
    f(r) = \int_0^r \theta(s)sds.
  \end{align*}
  The assumptions on $\theta$ imply that $f(r)=c r^2$ close to $0$ so that $H$ is globally smooth.
\end{proof}

Now fix a smooth decreasing function $\theta:[0,\infty) \to [0,\tfrac{\pi}{2}]$ so that
\begin{itemize}
\item $\theta(r)=\tfrac{\pi}{2}$ for $0\leq r\leq 1$,
\item $1<r\sin(\theta(r))\leq 2$ for $r>1$ and
\item $\theta(r) = \arcsin(\tfrac{2}{r})$ for $r\geq 3$.
\end{itemize}
Note that the second condition is easily solved compatibly with the others by rewriting it as $\theta(r) \in (\arcsin(\tfrac{1}{r}),\arcsin(\tfrac{2}{r})]$. If we apply the generalized rotation associated with this $\theta$ to the zero section $\R^n \subset \C^n$, we get a Lagrangian $F_\infty' \subset \R^{2n}=\C^n$ (illustrated as the red line in Figure~\ref{fig:rotate}) which
\begin{figure}[ht]
  \centering
  \definecolor{lgr}{rgb}{0.95,0.95,0.95}
  \begin{tikzpicture}[scale=0.4]
    \draw[thin,->] (-5,0) -- (5,0) node[below] {$x$};
    \draw[thin,->] (0,-2.9) -- (0,2.9) node[left] {$y$};
    \draw[red,thick] (0,0) -- (0,1) to[out=90,in=180] (2,2) -- (7,2);
    \draw[red,thick] (0,0) -- (0,-1) to[out=270,in=0] (-2,-2) -- node[above] {$F_\infty'$} (-7,-2);
    \draw (0,1) node[left] {$1$};
    \draw (0,-1) node[left] {$-1$};
    \draw (0,2) node[left] {$2$};
    \draw (0,-2) node[left] {$-2$};
  \end{tikzpicture}
  \caption{Image of rotated zero section for $n=1$}
  \label{fig:rotate}
\end{figure}
satisfies $F_\infty' \cap D^*\R^n = iD^n$. Indeed, the second bullet point above implies that for $x\in \R^n$ with $\norm{x}> 1$ the norm of the imaginary part of $e^{i\theta(\norm{x})}x$ must lie in $(1,2]$. We also conclude that $F_\infty' \subset D^*_2\R^n$.

\begin{lemma} \label{lem:Action:3}
  For any $R>1$ there is a compactly supported generalized rotation of the zero section to a Lagrangian $F_R'$ such that
  \begin{align*}
    F_R' \cap T^*D_{R+1}^n = F_\infty' \cap T^*D_{R+1}^n.
  \end{align*}
\end{lemma}

The Lagrangian $F_R'$ is illustrated in Figure~\ref{fig:rotate2}.
\begin{figure}[ht]
  \centering
  \definecolor{lgr}{rgb}{0.95,0.95,0.95}
  \begin{tikzpicture}[scale=0.4]
    \draw[thin,->] (-13,0) -- (13,0) node[below] {$x$};
    \draw[thin,->] (0,-2.9) -- (0,2.9) node[left] {$y$};
    \draw[red,thick] (0,0) -- (0,1) to[out=90,in=180] (2,2) -- (8,2) to[out=0,in=180] (12,0)  -- (14,0);
    \draw[red,thick] (0,0) -- (0,-1) to[out=270,in=0] (-2,-2) -- node[above] {$F_R'$} (-8,-2) to[out=180,in=0] (-12,0)  -- (-14,0);
    \draw (0,1) node[left] {$1$};
    \draw (0,-1) node[left] {$-1$};
    \draw (0,2) node[left] {$2$};
    \draw (0,-2) node[left] {$-2$};
  \end{tikzpicture}
  \caption{Image of rotated zero section with compact support}
  \label{fig:rotate2}
\end{figure}

\begin{proof}
  As the conditions get stronger when we make $R$ larger, we may assume that $R>2$. Define a new function $\theta_R$ by $\theta_R(r) = \theta(r)$ for $r\leq R+2$, and then bump it off for $r>R+2$, so that it equals 0 outside a compact set. We let $\phi_R : \C^n \to \C^n$ denote the associated compactly supported generalized rotation. We define $F_R'=\phi_R(\R^n)$ and prove below that this satisfies the conditions. We may of course assume that $0 \leq \theta_R \leq \theta$ which again implies that $F_R' \subset D_2^*\R^n$. 

  As the norm is preserved by generalized rotations, it follows that if $z$ is any point where $F_\infty'$ and $F_R'$ disagree then $\norm{z}\geq R+2$. Since the imaginary part is bounded in norm from above by $2$ and $R>2$, its real part must satisfy
  \begin{align*}
    \norm{\operatorname{Re} z}^2 = \norm{z}^2-\norm{\operatorname{Im} z}^2 \geq ((R+1)+1)^2-2^2 > (R+1)^2
  \end{align*}
  and therefore $z$ is \emph{not} contained in $T^*D_{R+1}^n$.
\end{proof}

Let $F_R$, for $R\in[2,\infty]$, be given by the shift $F_R'+(x_0,0)$ where $F_R'$ is as in the lemma above for $R<\infty$. This is illustrated in Figure~\ref{fig:rotate3}.
\begin{figure}[ht]
  \centering
  \definecolor{lgr}{rgb}{0.95,0.95,0.95}
  \begin{tikzpicture}[scale=0.4]
    \draw[thin,->] (-13,0) -- node[below left] {$x_0$} (13,0) node[below] {$x$};
    \draw[thin,->] (1,-2.9) -- (1,2.9) node[left] {$y$};
    \draw[red,thick] (0,0) -- (0,1) to[out=90,in=180] (2,2) -- (8,2) to[out=0,in=180] (12,0)  -- (14,0);
    \draw[red,thick] (0,0) -- (0,-1) to[out=270,in=0] (-2,-2) -- node[below] {$F_R$} (-8,-2) to[out=180,in=0] (-12,0)  -- (-14,0);
    \draw (1,1) node[right] {$1$};
    \draw (1,-1) node[right] {$-1$};
    \draw (1,2) node[right] {$2$};
    \draw (1,-2) node[right] {$-2$};
  \end{tikzpicture}
  \caption{Image of bumped off rotated zero section for $n=1$}
  \label{fig:rotate3}
\end{figure}
The above lemma implies the following (which explains why we used $R+1$ in the lemma above).

\begin{corollary} \label{cor:Action:1}
  For any $R>1$ there is a compactly supported Hamiltonian isotopy of the zero section to a Lagrangian $F_R$ such that
  \begin{align*}
    F_R \cap T^*D_R^n = F_\infty \cap T^*D_R^n.
  \end{align*}
  In particular $F_R\cap D^*D_R^n = D^*_{x_0}\R^n = iD^n + (x_0,0)$.
\end{corollary}

As $(x_0,0)$ corresponds to $(0,0)$ in $(q,p)\in T^*L_\infty$ coordinates, the plan is roughly to undo this rotation inside the neighborhood $D_\epsilon^*L_\infty\cong D^*_\epsilon \R^n$ (the latter using $(q,p)$ coordinates) to rotate the Lagrangian back to the zero section $L_\infty$---at least over the compact subset $L$. The end result is illustrated in Figure~\ref{fig:Rotationqp2}. Concretely, we apply a generalized rotation in symplectic coordinates of the form $(q',p')= (c^{-1}q,cp) \in \R^{2n}$ depending on some large $c>0$. We want the rotation to happen within the set where $\norm{p} < \epsilon$ which corresponds to $\norm{p'}<c\epsilon$. So we let $f_c(r)$ be a solution to the rotation from Lemma~\ref{lem:Action:2} where $\theta=\theta_c$ is a smooth increasing extension of:
\begin{align*}
  \overline{\theta_c}(r) =
  \begin{cases}
    -\tfrac{\pi}{2} & r < \tfrac{c\epsilon}{2} \\
    0 & r > c\epsilon
  \end{cases}.
\end{align*}
This defines a Hamiltonian which in $(q,p)$ coordinates is given by
\begin{align} \label{eq:3}
  H(q,p) = f_c(\sqrt{c^{-2}\norm{q}^2+c^2\norm{p}^2})
\end{align}
In $(q,p)$ coordinates the resulting flow applied to the fiber over $q=0$ is illustrated by the red part in Figure~\ref{fig:Rotationqp2}.
\begin{figure}[ht]
  \centering
  \definecolor{mygreen}{rgb}{0.9,1,0.75}
  \begin{tikzpicture}
    \fill[mygreen] (-0.0,-1.85) -- (0.8,-1.85) -- (0.8,1.85) -- (-0.0,1.85) -- cycle;
    \draw[->] (-4.4,0) -- (5,0) node[below] {$q$};
    \draw[->] (0,-1.7) -- (0,1.7) node[left] {$p$};
    \draw (0,0) ellipse (4 and 1);
    \draw[dashed] (0,0) ellipse (2 and 0.5);
    \draw[red, ultra thick] (2,0) .. controls (3,0.3) and (0,0.6) .. (0,1);
    \draw[red, ultra thick] (-2,0) .. controls (-3,-0.3) and (0,-0.6) .. (0,-1);
    \draw[red, ultra thick] (4.2,1.7) -- (4.7,1.7);
    \draw (4.7,1.7) node[right] {$S_c$};
    \draw[red, ultra thick] (4.2,1.3) -- (4.7,1.3);
    \draw[green, ultra thick, dashed] (4.2,1.3) -- (4.7,1.3);
    \draw (4.7,1.3) node[right] {$L\subset S_c$};
    \fill[mygreen] (4.2,0.9) -- (4.7,0.9) -- (4.7,0.5) -- (4.2,0.5) -- cycle;
    \draw (4.7,0.7) node[right] {$T^*L$};
    \draw[red, ultra thick,dotted] (0,-1.85) -- (0,-1.5);
    \draw[red, ultra thick,dotted] (0,1.85) -- (0,1.5);
    \draw[red, ultra thick] (0,-1.5) -- (0,-1);
    \draw[red, ultra thick] (0,1.5) -- (0,1);
    \draw[red, ultra thick] (-2,0) -- (2,0);
    \draw[ultra thick, green, dashed] (-0.0,0) -- (0.8,0);
    \draw (2,-0.1) -- node[below] {$\tfrac{c^2\epsilon}{2}$} (2,0.1);
    \draw (4,-0.1) -- node[below] {$c^2\epsilon$} (4,0.1);
    \draw (0,1) node[above left] {$\epsilon$};
    \draw[->] (0.7,1.2) -- (0.5,0.77);
    \draw[->] plot[smooth] coordinates {(0.9,1.2) (1,0) (0.1,-0.8)};
    \draw (0.6,1.1) node[above right] {$K_1$};
  \end{tikzpicture}  
  \caption{Rotation in $(q,p)$ coordinates.}
  \label{fig:Rotationqp2}
\end{figure}

\begin{lemma} \label{lem:Action:5}
  Let $S_c$ be the result of the Hamiltonian flow of the Hamiltonian from Equation~\eqref{eq:3} applied to the fiber $T^*_{x_0}L_\infty \subset T^*L_\infty$ and let $p_c:S_c \to \R$ be the unique primitive of the Liouville form $\lambda_{L_\infty}$ in $T^*L_\infty$ whose value is $0$ at $(q,p)=(0,0)$. For any $C>0$ there is $c>0$ large enough such that
  \begin{align*}
    S_c\cap T^*L = L \sqcup K_1
  \end{align*}
  with primitive $p_c=0$ on $L$ and $\sup_{K_1}p_c < -C$.
\end{lemma}

Note that $L \subset L_\infty$ is compact with $x_0$ on its boundary.

\begin{proof}
  This follows by construction. In particular, the statement about the primitive follows from the fact that the part of the red curve away from the zero section stays out of the dashed ellipse in Figure~\ref{fig:Rotationqp2}, and thus the signed symplectic area below the curve, as it returns to $K_1$, goes to $-\infty$ as $c\to\infty$. Note in particular that on the non-compact part, which stays inside the fiber, the primitive is a constant that depends on $c$.
\end{proof}

We now have all the pieces for proving the following proposition.

\begin{proposition} \label{prop:Action:1}
  For $n\geq 5$ there exists a compactly supported Hamiltonian isotopy of the zero section $\R^n \subset \C^n$ to a Lagrangian $K$ such that
  \begin{align*}
    K\cap T^*D^n = L \sqcup K_1 \qquad \textrm{and} \qquad K\cap T^*_{x_0}\R^n = \{(x_0,0)\}
  \end{align*}
  and such that for any primitive $p_K: K \to \R$ of the Liouville form there is a $C\in \R$ such that
  \begin{align*}
    \min_{z \in L} p_K(z) > C > \max_{z \in K_1} p_K(z).
  \end{align*}
\end{proposition}

\newcommand{\marg}{0.05}
\begin{figure}[ht]
  \centering
  \begin{tikzpicture}[scale=0.5]
    \draw[red,very thick] (0,1) to[out=90,in=180] (2,2) -- (8,2) to[out=0,in=180] (12,0)  -- (12.5,0);
    \draw[red,very thick] (0,-1) to[out=270,in=0] (-2,-2) -- (-8,-2) to[out=180,in=0] (-12,0)  -- (-12.5,0);
    \draw[red,very thick] (0,-1) -- (0,-0.3);
    \draw[red,very thick] (0,1) -- (0,0.3);
    \draw[red,very thick] (0,-0.3) .. controls (0,-0.2) and (-8.5,-0.3) .. (-7.5,-0);
    \draw[red,very thick] (0,0.3) .. controls (0,0.2) and (8.5,0.3) .. (7.5,0);
    \draw[red,very thick] (-7.5,0) -- (7.5,0);
    \fill[fill=pink] (\marg,-1+\marg) -- (2-\marg,-1+\marg) -- (2-\marg,1-\marg) -- (\marg,1-\marg) -- cycle;
    \draw[green] (0.4,-3) -- (-0.4,3);
    \draw (0,0.1) node[below] {$x_0$}  -- (0,-0.1);
    \draw[->] (-8,0) -- (8,0) node[below] {$x$};
    \draw[->] (1,-3) -- (1,3) node[left] {$y$};
    \draw[red] (0.98,0) node {\Huge ?}; 
  \end{tikzpicture}
  \caption{The red line is an illustration of $K'$ from the proof of Proposition~\ref{prop:Action:1}. Inside the pink box, $K'$ consists of two copies of $L$ close to each other. However, since the flow in the box is far from being given by complex rotation in $(x,y)$ coordinates, we choose not to illustrate it at all.}
  \label{fig:flowresult}
\end{figure}

\begin{proof}
  All Lagrangians considered in this proof are connected and contain $x_0$, and we make primitives unique by always choosing the one that is $0$ at $x_0$. This only changes the desired inequality up to a constant shift, which we may ignore. The Lagrangian $F_R$ from Corollary~\ref{cor:Action:1} satisfies $F_R\cap T^*D^n = F_\infty \cap T^*D^n$ and this piece of it is compact and connected. It follows that the primitive, say $f_{F_\infty} : F_R\cap T^*D^n \to \R$, is independent of $R$ and that there is a $C_1>0$ such that $|f_{F_\infty}(z)-f_{F_\infty}(z')| < C_1$ for $z,z' \in F_\infty\cap T^*D^n$.

  The difference between the two Liouville forms $\lambda_{L_\infty} - ydx$ on $D_\epsilon^*L_\infty$ is compactly supported. Indeed, in Lemma~\ref{lem:Action:1} we made sure that outside a compact set the embedding was given by the canonical identification of cotangent bundles---hence the Liouville forms agree. By exactness this difference is the differential of a function $P:D_\epsilon^*L_\infty \to \R$ which extends $f_{L_\infty} : L_\infty\to \R$. Since $f_{L_\infty}$ is compactly supported we get that $P$ is compactly supported. Hence there is a $C_2>0$ bounding the absolute value of $P$. Since $P(x_0)=0$ and we are fixing primitives to be $0$ at $x_0$ we conclude that the two primitives (one for each Liouville form) on $S_c$ (from Lemma~\ref{lem:Action:5}) inside $D_\epsilon^*L_\infty$ differ exactly by $P$ restricted to $S_c$. Hence they differ by at most $C_2$.

  Now fix $c>0$ in Lemma~\ref{lem:Action:5} such that, on the part of $S_c$ inside $D_\epsilon^*L$ where the primitive $p_c$ is negative, it is in fact strictly less than $-2C_2-C_1$. The Hamiltonian in that Lemma is compactly supported inside $D_\epsilon^*L_\infty$ and thus extends to a Hamiltonian on $\R^{2n}$ (using $(x,y)$ coordinates) with support inside $D^*D_R^n$ for some $R>1$, which we now fix.  
  
  Now, we define $K'$ (which will be modified slightly later to become $K$) as the Lagrangian given by first using this $R$ in Lemma~\ref{lem:Action:3} to get a Lagrangian $F_R$ that coincides with the fiber over $x_0$ in $D_\epsilon^* L_\infty\cap D^*D_R^{n}$. Then we use Lemma~\ref{lem:Action:5} with our chosen $c>0$ inside this subspace to flow $F_R$ to $K'$. By construction we have $K'\cap T^*D^n=L \sqcup K_1'$. See Figure~\ref{fig:flowresult} for a heuristic picture of $K'$.

  By construction, the primitive (with respect to $ydx$) which equals 0 at $x_0$ for $K'\subset \R^{2n}$ is given by $P=f_{L_\infty}$ on the part of $K'$ in $T^*D^n$ which coincides with $L$. This is bounded from \emph{below} by $-|P|$ hence bounded from below by $-C_2$.

  We claim that the primitive on $K_1'$ is strictly bounded from \emph{above} by $-C_2$. Indeed, by our choice of $c$ the primitive inside $D_\epsilon^* L_\infty$ on $K_1'$ with respect to $\lambda_{L_\infty}$ is bounded from above by $-2C_2-C_1$. Hence by the bound on $P$ the primitive inside $D_\epsilon^*L_\infty$ with respect to $ydx$ is bounded strictly from above by $-C_2-C_1$. This in particular implies the same bound on the boundary $S_\epsilon T^*_{x_0}\R^n \subset F_R \cap T^*D^n$. As $F_R-D_\epsilon^*L_\infty$ is connected and the isotopy from Lemma~\ref{lem:Action:5} keeps $K'$ constantly equal to $F_R$ outside $D_\epsilon^*L_\infty$ the primitive can only change by adding a constant. Therefore, the bound $C_1$ on the variation of the primitive $f_{F_\infty}$ on $F_R\cap T^*D^n$ is still valid for the primitive on $K'$; hence we get that the primitive on the part of $K'$ outside $D_\epsilon^*L_\infty$ is bounded from above by $-C_2-C_1 + C_1 = -C_2$.

  The only thing in the proposition that $K'$ does not satisfy is that $K'\cap T^*_{x_0}\R^n$ is the union of $\{(x_0,0)\}$ and an annulus. However, as the green line in Figure~\ref{fig:flowresult} indicates, one may shear near this fiber and make the fiber completely miss this part of $K'$. To be explicit, we consider a neighborhood around $x_0$ in $L$ where the Riemannian structures agree. In Figure~\ref{fig:flowresult2},
\begin{figure}[ht]
  \centering
  \begin{tikzpicture}[scale=0.7]
    \draw[red, thick] (0,2) to[out=180,in=90] (-1,1) -- (-1,0.4) to[out=270,in=180] (0,0.1);
    \draw[red, thick] (-2,-2) to[out=0,in=270] (-1,-1) -- (-1,-0.4) to[out=90,in=0] (-2,-0.1);
    \draw[->] (-2,0) -- (0,0) node[below] {$x$};
    \draw[dashed, thin] (-1,-3) -- (-1,3);
    \draw (-1,-0.2) node {$x_0$};
  \end{tikzpicture}
  \caption{$K'$ close to $T_{x_0}^*\R^n$.}
  \label{fig:flowresult2}
\end{figure}
we see $K'$ close to the fiber $T_{x_0}^*\R^n$ for $n=1$. For $n>1$, the picture still describes all points in $K'$. Indeed, intersecting $K'$ with an affine complex line through $(x_0,0)$ parallel to a real unit vector $v\in S^{n-1} \subset \R^n$ yields the same local picture, independent of $v$. It follows that an arbitrarily small negative generalized rotation (around $(x_0,0)$) would clear the fiber for all points except $x_0$. However, we do not want to change the piece that equals $L$. So, we pick a Hamiltonian $H$ which equals $H(x,y)=-\norm{y}^2$ in a small neighborhood of the annulus, but has support in a slightly larger neighborhood (away from $L$). We define the final Lagrangian $K$ as the Hamiltonian time-1 flow of $\delta H$ for small $\delta>0$. For $\delta$ small enough, this does not change the strict bounds on the primitive and $K$ does not intersect the fiber except at $x_0\in L$.
\end{proof}


%% file: Theorem1.tex
\section{Cutting and pasting generating functions and proof of Theorem~\ref{thm:1}} \label{sec:cutting-paste-gener}

In this section we perform the cut-and-paste argument on a generating function quadratic at infinity for the Lagrangian $K$ from Proposition~\ref{prop:Action:1} to obtain one for $L\subset T^*D^n$, thereby proving Theorem~\ref{thm:1} (in the case $n\geq 5$).

Recall that $F:\R^n \times \R^{2k} \to \R$ is called \emph{quadratic at both infinities} if $F=Q_k$ outside a compact set. Here
\begin{align*}
  Q_k(q,x,y)= Q_k(x,y) = - \norm{x}^2 + \norm{y}^2
\end{align*}
for $x,y\in \R^k$, which we call the standard quadratic form. In \cite{MR765426} Chaperon proved that the image of $\R^n$ under any compactly supported Hamiltonian isotopy has such a generating function. We use this and the previous constructions to prove the following lemma.

\begin{lemma} \label{lem:Theorem1:2}
  There is a generating function $F:D^n \times \R^{2k} \to \R$ quadratic at infinity and two fiberwise regular values $C'<C$ of $F$ such that
  \begin{itemize}
  \item $F$ generates the zero section in a neighborhood of $\{x_0\} \in D^n$.
  \item The restriction of $F$ to the set $\{F\leq C'\} \cup \{F \geq C\}$ generates $L\subset T^*D^n$. 
  \end{itemize}  
\end{lemma}

Note that the set in the second bullet point has complement given by
\begin{align*}
  \{C'<F<C\}.
\end{align*}
In the next section we will prove the case $n\leq 4$.

\begin{proof}[Proof of Lemma~\ref{lem:Theorem1:2} for $n\geq 5$]
  Proposition~\ref{prop:Action:1} and Chaperon's result imply that there is a generating function $F':\R^n\times \R^{2k} \to \R$ quadratic at both infinities for the Lagrangian $K$ described in the proposition.
  
  Now, the restriction
  \begin{align*}
    F = F'_{\mid D^n\times \R^{2k}} : D^n \times \R^{2k} \to \R
  \end{align*}
  generates $K\cap T^*D^n$, which by Proposition~\ref{prop:Action:1} has two pieces:
  \begin{align*}
    K\cap T^*D^n = L \sqcup K_1
  \end{align*}
  and in particular has $K \cap T^*_{x_0}\R^n = \{(x_0,0)\}$. 
  
  Let $C$ be a value as in Proposition~\ref{prop:Action:1} that separates the primitive values on these two pieces for the specific primitive induced by $F$. This primitive gives the fiberwise critical value of $F$ at each fiberwise critical point generating a point of $K$. Using that $F=Q_k$ outside a compact set, we may pick another constant $C' < C$ such that the fiberwise critical values of $F$ are all bounded from below by $C'$. This implies that the restriction
  \begin{align*}
    F_{\mid} : \{F\leq C'\} \cup \{F \geq C\} \to \R
  \end{align*}
  is a function that generates only $L$ and not $K_1$.  
\end{proof}

Assuming Lemma~\ref{lem:Theorem1:2} for all $n$, we now prove Theorem~\ref{thm:1}.

\begin{proof}[Proof of Theorem~\ref{thm:1}]
  Let $F_{\mid}$ be the restriction of $F$ provided by Lemma~\ref{lem:Theorem1:2} to the subset $\{F\leq C'\} \cup \{F \geq C\}$. We wish to extend this $F_\mid$ to all of $D^n\times \R^{2k}$ without generating anything beyond $L$.

  Consider the closure of the \emph{missing piece}:
  \begin{align*}
    M = \{C' \leq F \leq C\}
  \end{align*}
  Since $C$ and $C'$ are fiberwise regular values, this defines a fiber bundle (whose fibers have boundary). Let $K \subset D^n \times \R^{2k}$ be compact such that $F=Q_k$ outside $K$. Picking a large ball $D_R^{2k} \subset \R^{2k}$ containing the projection of $K$ to $\R^{2k}$ in its interior, we see that $M\cap (D^n\times (\R^{2k}-D_R^{2k}))$ is a product. Since $D^n$ is contractible, we can trivialize $M$ so that it respects the product structure outside $D^n\times D_R^{2k}$. In fact, we will be slightly more specific about this. Pick a horizontal distribution (a fiber bundle connection) for $D^n \times \R^{2k} \to D^n$ which satisfies:
  \begin{itemize}
  \item It is tangent to all regular level sets $\{F=a\}$ for $a$ close to $C$ or $C'$.
  \item Outside $D^n \times D_R^{2k}$ it is the trivial horizontal distribution.
  \end{itemize}
  This restricts to $M$ by construction, and using parallel transport (lifting smooth curves) we can trivialize the bundle.
  
  The chosen horizontal distribution makes it possible to fill in the missing piece without creating fiberwise critical points if we can do so in a single fiber. The function $F_{\mid \{x_0\} \times \R^{2k}}$ provides such an extension in the fiber over $x_0$. As the horizontal distribution is tangent to the level sets of $F$ close to the ``seam'' $\partial M$, the pasted function is smooth. Furthermore, the fact that the trivialization is induced by the standard product structure outside a compact set makes the resulting function equal to $Q_k$ outside this compact set (since this was already true in the fiber over $x_0$). This yields the desired generating function, proving Theorem~\ref{thm:1}.  
\end{proof}

\section{Extending to $n\leq 4$}\label{sec:extending-nleq-4}

\newcommand{\nn}{9}
\newcommand{\nnm}{8}

In this section, we describe how to modify Sections~\ref{sec:an-infin-dimens} and~\ref{sec:cutting-paste-gener} for the case $n\leq 4$. The difficulty is that the standard Morse-theoretic arguments do not work in low dimensions.

Instead of $L_\infty$, we ``stabilize'' and consider the Lagrangian $L'_\infty=L_\infty \times \R^{\nn-n} \subset \C^{\nn}$ which equals $\R^{\nn}$ outside $T^*(D^n\times \R^{\nn-n})$. Similarly, we let $L'=L\times \R^{\nn-n}$, and as before, we let $L_\delta=L\cap T^*D_\delta^n$ be a slightly smaller version of $L$ and also define its stabilization $L_\delta'=L_\delta\times \R^{\nn-n}$. The spaces $L'$ and $L'_\delta$ are not compact, but as we only care about generating the compact $L$ over $D^{n}\times \{0\}$, this will not pose any serious problems. 
 
The stabilization $L'_\infty$ is still contractible, and the first main point is to find a flat Riemannian structure on $L'_\infty$ and generalize Lemmas~\ref{lem:Action:4} and~\ref{lem:Action:1} to obtain the following lemma. We still denote $x_0=(-1,0,\dots,0)\in L \times D^{\nn-n}$.

\begin{lemma} \label{lem:Action:1b}
  There is a diffeomorphism $\varphi : L'_\infty \cong \R^{\nn}$ given by $x \mapsto x-x_0$ in a small neighborhood of $x_0$, and there is a symplectic embedding
  \begin{align*}
    D^*_\epsilon L_\infty \subset D^*_{\tfrac12}\R^{n}  
  \end{align*}
  satisfying the properties listed in Lemma~\ref{lem:Action:1}. Here, by abuse of notation, we identify $L_\infty$ with $L_\infty \times \{0\} \subset L_\infty'$, and we equip it with the restricted Riemannian structure induced by $\varphi$. Furthermore, there is an $\epsilon'<\epsilon$ small enough so that
  \begin{align*}
    D^*_{\epsilon'} L_\infty' \subset D^*_\epsilon L_\infty \times D^*_{\tfrac12} \R^{\nn-n} \subset D^*\R^\nn.
  \end{align*}
\end{lemma}
We will need the following lemma to prove this. Recall that $L$ agrees with the zero section over a neighborhood of $\partial D^{n}$.

\begin{lemma} \label{cor:Theorem1:1}
  There exists a diffeomorphism
  \begin{align*}
    \psi : L \times D^{\nn-n} \cong D^n\times D^{\nn-n}
  \end{align*}
  as manifolds with corners, which is the identity near the boundary part $\partial L\times D^{\nn-n}$ and of the form
  \begin{align} \label{eq:Theorem1:1}
    \psi(l,x) = (\psi_1(l,\hat x),\norm{x}\psi_2(l,\hat x))
  \end{align}
  near the other boundary part $L\times S^{\nnm-n} \cong D^n\times S^{\nnm-n}$.
\end{lemma}

\begin{proof}
  Consider a homotopy equivalence relative to the boundary $\varphi:L \to D^n$ which is the identity near $\partial L = S^{n-1}$. Since $2\dim(L)<9$, we may lift this map to a smooth embedding $L \to D^n \times D^{\nn-n}$, which agrees with the standard embedding of $D^{n}$ in a neighborhood of $S^{n-1}\times D^{\nn-n}$. The normal bundle of this embedding is trivial over the boundary $S^{n-1}$ by identifying its disk with $D^{\nn-n}$ and we claim that this trivialization extends over all of $L$. Indeed, in the cases $n=1$ and $n=2$, this follows as $L$ is $D^n$ and the homotopy equivalence is homotopic relative to the boundary to the identity, and in the case $n=3$, it follows from $\pi_3(BO(6))=0$. In the case $n=4$, it follows since the Hirzebruch signature theorem (used after attaching a disk to get a 4-sphere) shows that the tangent bundle of $L$ relative to the boundary is stably trivial. 

  It follows that we have an embedding $L\times D^{\nn-n}_r \to D^n\times D^{\nn-n}$ for small $r>0$ which is the standard inclusion over a neighborhood of the boundary of $L$. The closure of the complement is thus (after possibly making $r$ smaller) an $h$-cobordism relative to the boundary from $L\times S_r^{\nnm-n}$ to $D^n\times S^{\nnm-n}$, which by the $h$-cobordism theorem can be trivialized relative to the boundary. That is, we identify it with the closure of $L \times (D^{\nn-n}-D^{\nn-n}_r)$. This can be pasted together with the inclusion to get the desired diffeomorphism.

  The last statement follows, since we may assume that the trivialization of the $h$-cobordism respects the radial collar close to $D^n\times S^{\nnm-n}$.
\end{proof}

\begin{proof}[Proof of Lemma~\ref{lem:Action:1b}]

Let $\psi$ be a diffeomorphism as in Lemma~\ref{cor:Theorem1:1} above. As it satisfies Equation~(\ref{eq:Theorem1:1}) close to $L\times S^{\nnm-n}$ we may radially extend it to a diffeomorphism
\begin{align*}
  \overline{\psi} : L \times \R^{\nn-n} \cong D^n \times \R^{\nn-n}
\end{align*}
which is the identity in a neighborhood of $S^{n-1}\times \R^{\nn-n}$. This has globally bounded derivatives, as it is radial in the second factor. We may further extend this by the identity to a diffeomorphism $\Psi : L_\infty \times \R^{\nn-n} \cong \R^{\nn}$. Changing this diffeomorphism inside $D_2^n \times D^{\nn-n}$, we may assume that in a neighborhood of $x_0$ it maps $x \mapsto x-x_0$.

The existence of $\epsilon>0$ and an embedding $D_\epsilon^*L_\infty \to T^*\R^n$ is proved exactly as in Lemma~\ref{lem:Action:1} with the following small exception. The diffeomorphism $\Psi$ is actually the identity outside $L$, so the induced Riemannian structure on $L_\infty$ is in this case the trivial one. The global bound on the derivative of $\Psi$ makes it possible to pick $\epsilon'>0$ as prescribed in the lemma.
\end{proof}

We will again employ Chaperon's construction of a generating function over $\R^\nn$. However, as we will then restrict this to $D^n \times \{0\}$, it is important to identify precisely what such restrictions generate. This is given by a symplectic reduction. Indeed, for a generating function $F:\R^\nn \times \R^{2k} \to \R$ generating a Lagrangian $K \subset T^*\R^\nn$, the restriction generates the image of
\begin{align*}
  K \cap (T^*\R^\nn)_{\mid \R^n \times \{0\}} \to T^*\R^n.
\end{align*}
Here we are restricting the entire $\nn$-dimensional cotangent bundle over $\R^n$ and then orthogonally projecting to the subbundle. Note that, even if $K$ is embedded in $T^*\R^\nn$, the image might not be embedded in $T^*\R^n$. The condition that the restriction cuts out its singular set transversely is equivalent to the intersection above being transverse.

To finish the argument, we thus need the following slight modification of Lemma~\ref{lem:Action:5}. We are now using $(q,p)$ coordinates on $T^*L_\infty'$ induced by a diffeomorphism to $\R^\nn$ as in Lemma~\ref{lem:Action:1b} above.

\begin{lemma} \label{lem:Action:5b}
  Let $S_c$ be the result of the Hamiltonian flow of the Hamiltonian from Equation~\eqref{eq:3} applied to the fiber $T^*_{x_0}L_\infty' \subset T^*L_\infty'$ and let $p_c:S_c \to \R$ be the unique primitive of the Liouville form $\lambda_{L_\infty'}$ in $T^*L_\infty'$ whose value is $0$ at $(q,p)=(0,0)$. For any $C>0$ there exists a $c>0$ large enough such that
  \begin{align*}
    S_c\cap (T^*L')_{\mid L} = L \sqcup K_1
  \end{align*}
  with primitive $p_c=0$ on $L$ and $\max_{K_1}p_c < -C$. Furthermore, the intersection is transverse along $L$.
\end{lemma}

\begin{proof}
  The proof is exactly the same, except that we apply the generalized rotation within $D_{\epsilon'}^*L_\infty'$ from Lemma~\ref{lem:Action:1b}. Furthermore, even though $L'$ is non-compact, we only need $L$ to be compact to conclude all statements in the lemma except the last. $L$ is cut out transversely since the flow of the fiber gives the zero section in a neighborhood of $L$.  
\end{proof}

The version of Proposition~\ref{prop:Action:1} that we will need is the following.

\begin{proposition} \label{prop:Action:1b}
  There exists a compactly supported Hamiltonian isotopy of the zero section $\R^\nn \subset \C^\nn$ to a Lagrangian $K$ such that
  \begin{align*}
    K\cap (T^*\R^\nn)_{\mid D^n\times \{0\}} = L \sqcup K_1 \qquad \textrm{and} \qquad K\cap (T^*_{x_0}\R^{\nn}) = \{(x_0,0)\}
  \end{align*}
  and such that for any primitive $p_K: K \to \R$ of the Liouville form there is a $C\in \R$ such that
  \begin{align*}
    \min_{z \in L} p_K(z) > C > \max_{z \in K_1} p_K(z).
  \end{align*}
\end{proposition}

\begin{proof}
  The proof is the same as that of Proposition~\ref{prop:Action:1}, except that we use the lemmas above, and we only care about ensuring that $K$ has the desired form and obtaining the bounds over the compact sets $L\times \{0\}$ and $D^n\times \{0\}$.
\end{proof}

\begin{proof}[Proof of Lemma~\ref{lem:Theorem1:2} in the case $n\leq 4$]
  The proof is essentially the same, except that we use the proposition above instead of Proposition~\ref{prop:Action:1} and restrict the generating function to $D^n\times \{0\}$. Note that the projection of $K_1$ to $T^*D^n$ need not be embedded, but the constructed generating function only generates $L$, which is embedded.
\end{proof}


%% file: Gauss.tex
\section{The Space $\M_\infty$ and the Lagrangian Gauss map} \label{sec:lagrangian-gauss-map}

In this section we define the spaces $\M_k$ whose colimit we denote $\M_\infty$. The space $\M_k$ consists of the possible fibers of generating functions quadratic at infinity that generate the zero section with primitive value 0. This means that, given such a generating function $F:N \times \R^{2k} \to \R$, the adjoint defines a map
\begin{align*}
  N \to \M_k,
\end{align*}
which we will make precise in this section. This is of interest since the generating function from Theorem~\ref{thm:1} generates the zero section over each point of the boundary $S^{n-1}$, and therefore the adjoint over $S^{n-1}$ defines a map
\begin{align*}
  S^{n-1} \to \M_k.
\end{align*}
We also consider the canonical map
\begin{align*}
  N_\infty : \M_\infty \to BO \subset \Z \times BO
\end{align*}
given by taking the negative eigenspace of the Hessian at the critical point. The composition $S^{n-1} \to \M_k \subset \M_\infty \to \Z\times BO$ thus defines an element in $\pi_{n-1}(\Z\times BO)$. The stable Lagrangian Gauss map $L \to U/O$ maps $S^{n-1}$ to the basepoint, hence using $(L,S^{n-1}) \simeq (D^n,S^{n-1})$ we get an element in $\pi_n(U/O)$. In this section we prove that these two elements agree under the standard Bott-periodicity map $\pi_n(U/O) \cong \pi_{n-1}(\Z\times BO)$.

Let $Q_k^l : \R^k \times \R^l \to \R$ be the generalization of $Q_k=Q_k^k$ given by
\begin{align} \label{eq:Gauss:3}
  Q_k^l(x,y) = -\norm{x}^2 + \norm{y}^2
\end{align}
for $x\in \R^k$ and $y \in \R^l$. To make things as explicit as possible, we denote the coordinates in $\R^k \times \R^l$ by
\begin{align*}
  z = (x_k,x_{k-1},\dots,x_1,y_1,y_2,\dots,y_l).
\end{align*}

Let $\cC_k^l$ be the set of smooth maps $f:\R^{k+l}\to \R$ such that
\begin{align*}
  \norm{f-Q_k^l}_{C^1} = \max(\norm{f-Q_k^l}_{L^\infty},\norm{\nabla f-\nabla Q_k^l}_{L^\infty}) < \infty.
\end{align*}
However, as this $C^1$ norm is not continuous in the $C^\infty$ topology, we topologize $\cC_k^l$ by identifying it with the subspace of pairs
\begin{align*}
  (f,\norm{f-Q_k^l}_{C^1}) \in C^{\infty}(\R^{k+l},\R) \times \R.
\end{align*}
We need this modified topology to properly control the behavior at infinity. We will also often need this bound explicitly, so we define
\begin{align} \label{eq:Gauss:1}
  c_f = \norm{f-Q_k^l}_{C^1} + 1.
\end{align}
To understand the homotopy types of the mapping spaces, the following lemma is convenient.

\begin{lemma} \label{lem:Gauss:2}
  Given any compact set $K\subset \R^{k+l}$, the subspace of $\cC_k^l$ consisting of functions $f$ for which the support of $f-Q_k^l$ is contained in $K$ has the $C^\infty$ topology. 
\end{lemma}

\begin{proof}
  The map $C^\infty(\R^{k+l},\R) \to \R$ given by taking the $C^1$ norm defined above is continuous on the subspace of functions with support in $K$.
\end{proof}

Note, however, that the subspace of all functions with arbitrary compact support does not have the $C^\infty$ topology. For that reason, we will quite often use the homotopy from the identity on $\cC_k^l$ that bumps off the function outside a compact set to make it equal to $Q_k^l$ at infinity. Indeed, pick a function $\varphi : \R \to [0,1]$ such that $\varphi(t) = 0$ for $t\leq 0$, $\varphi(t)=1$ for $t\geq 2$ and $\varphi'-\varphi \leq \sqrt2 -1$. We use this to define the homotopy $B_t : \cC_k^l \to \cC_k^l$ from the identity to functions where $f-Q_k^l$ has compact support. We let
\begin{align} \label{eq:Gauss:2}
  B_t(f)(z) = (1-t\varphi(\norm{z}-c_f))f(z) + t\varphi(\norm{z}-c_f)Q_k^l(z).
\end{align}

\begin{lemma} \label{lem:Gauss:1}
  The homotopy $B_t : \cC_k^l \to \cC_k^l$ is well-defined, and for each $t$ the functions $B_t(f)$ and $f$ are equal near their critical points. Furthermore, $c_{B_t(f)} \leq \sqrt{2}c_f$ for each $t$.
\end{lemma}

\begin{proof}
  As the topology we impose on $\cC_k^l$ makes taking the norm $\norm{f-Q_k^l}_{C^1}$ continuous, Equation~(\ref{eq:Gauss:2}) defines a continuous map $\cC_k^l \to C^\infty(\R^{k+l},\R)$. Since $\norm{B_t(f)-Q_k^l}_{C^1}$ is continuous in $t$ and $f$, it follows that $B_t$ is continuous.

  We have $|B_t(f)-Q_k^l| \leq |f-Q_k^l|$. The gradient of $B_t(f)$ satisfies
  \begin{align*}
    \norm{\nabla B_t(f) - \nabla Q_k^l} \leq (1-t\varphi)\norm{\nabla f - \nabla Q_k^l} + t\varphi'|f-Q_k^l| \leq \sqrt 2 \norm{f-Q_k^l}_{C^1}.
  \end{align*}
  As $\norm{\nabla Q_k^l} = 2\norm{z}$ the gradient bound also shows that all critical points of $B_t(f)$ must be inside $D_{c_f}^{k+l}$ where it equals $f$.
\end{proof}

In this section, we will only need the following subspace.

\begin{defn} \label{def:Gauss:1}
  Let $\M_k^l \subset \cC_k^l$ be the subspace of functions $f$ such that
  \begin{itemize}
  \item $f$ has only one critical point.
  \item This critical point is non-degenerate and has critical value $0$.
  \end{itemize}
  We denote $\M_k=\M_k^k$.
\end{defn}

Note that the bound on the $C^1$-supremum norm means that the level sets at infinity are diffeomorphic to those of $Q_k^l$. This is even true for any $f\in \cC_k^l$. By standard Morse theory, the index of the unique critical point of $f\in \M_k^l$ must be $k$. We will only need $\M_k$ for the rest of this section.

The \emph{adjoint} of a map $F : X \times \R^{2k} \to \R$ is the map $F^{\Adj} : X \to \Map(\R^{2k},\R)$, which maps $x\in X$ to $F_{\mid \{x\}\times \R^{2k}}$. Any generating function $F:D^n \times \R^{2k} \to \R$ from Theorem~\ref{thm:1} has a single critical point over each fiber $x\in S^{n-1}$. The critical value may be nonzero, but it is constant as a function of $x$ on $S^{n-1}$. We claim that we may assume that this constant is $0$. Indeed, any $f\in \cC_k^l$ has $f+c \in \cC_k^l$ for any $c\in \R$. So, by adding a global constant to $F$ and bumping off fiberwise using Lemma~\ref{lem:Gauss:1} to again make the function equal to $Q_k$ at infinity, we obtain the claim. Note in particular that according to the lemma, the bumping off does not affect fiberwise critical points---hence does not change what Lagrangian the function generates. Using compactness of $S^{n-1}$ and the fact that the generating function equals $Q_k$ outside a compact set we thus get a map
\begin{align}\label{eq:4}
  F^{\Adj} : S^{n-1} \to \M_k.
\end{align}
It is well-known that admitting a generating function for a fixed $k$ is not invariant under Hamiltonian isotopy (this can be proved explicitly by locally rotating any Lagrangian sufficiently to obtain an arbitrarily large range of Maslov indices when intersected with a fiber). So, we replace these with their stable versions by defining stabilization maps $s_k:\M_k \to \M_{k+1}$ given by adding the quadratic form (corresponding to the standard one) in the new variables. That is,
\begin{align*}
  s_k(f)(x_{k+1},x,y,y_{k+1}) = -x_{k+1}^2 + f(x,y) + y_{k+1}^2.  
\end{align*}
This is often not equal to $Q_{k+1}$ at infinity, even if $f$ is. However, it does satisfy the needed bound to land in $\M_{k+1}$. Alternatively, one could have used Lemma~\ref{lem:Gauss:1} to bump off, thereby achieving equality at infinity.

We define
\begin{align*}
  \M_\infty = \colim_{k\to \infty} \M_k
\end{align*}
using these $s_k$. There are maps
\begin{align*}
  N_k : \M_k \to \Gr_k(\R^{2k})
\end{align*}
given by taking the negative eigenspace of the Hessian at the critical points. This map strictly commutes with stabilization maps if we use the convention that the stabilization map $\Gr_k(\R^{2k}) \to \Gr_{k+1}(\R^{2k+2})$ takes direct sum with $\R \subset \R^2$ by adding coordinates $x_{k+1}$ and $y_{k+1}$ such that $x_{k+1}$ is in the subspace. In the colimit $k\to \infty$ we get a map
\begin{align*}
  N_\infty : \M_\infty \to BO = \{0\} \times BO \subset \Z\times BO.
\end{align*}
For the rest of this section, we allow $L \to T^*D^n$ to be any \emph{immersed} Lagrangian disk that equals the zero section in a neighborhood of $S^{n-1} \subset T^*D^n$. This has a Lagrangian Gauss map $L \to U(n)/O(n)$. Indeed, the tangent space at each point is a Lagrangian in $\C^n$ and hence defines a point in $U(n)/O(n)$. On the boundary $S^{n-1}$ this is the standard $\R^n \in U(n)/O(n)$, which we use as the basepoint. Since $L/S^{n-1} \simeq S^n$ we thus get an element in $\pi_n(U(n)/O(n))$. Stabilizing using
\begin{align*}
  U/O = \colim_{n\to \infty} U(n)/O(n)
\end{align*}
we get an element in $\pi_n(U/O)$. By Bott-periodicity (see e.g. \cite{MR0163331}, which we partly recall in the proof below) we have that $\Omega(U/O) \simeq \Z\times BO$ and thus
\begin{align*}
  \pi_{n-1}(\Z\times BO) \cong \pi_n(U/O).
\end{align*}

\begin{proposition} \label{prop:Gauss:1}
  Let $S^{n-1} \to \M_k$ be the adjoint of the restriction to the boundary of a generating function $F:D^n\times \R^{2k} \to \R$ quadratic at infinity. Assume that it generates an immersed Lagrangian disk $L \to T^*D^n$ which agrees with the zero section in a neighborhood of $S^{n-1} \subset T^*D^n$. The composition
  \begin{align*}
    S^{n-1} \to \M_k \to \M_\infty \xrightarrow{N_\infty} \Z \times BO
  \end{align*}
  represents the class defined above in $\pi_n(U/O)$ under the isomorphism induced by Bott-periodicity.
\end{proposition}

\begin{proof}
  First, we note that the map $S^{n-1} \to \M_\infty$ has not been constructed as a based map. We could have done this, but the construction would have been more intricate, and at this point it makes no difference since the action of $\pi_1(BO)$ on higher homotopy groups is trivial because $BO$ is an infinite loop space. The unbased nature of the map could also make it slightly ambiguous as to which component of $\Z\times BO$ we land in. However, by definition $N_\infty$ lands in $\{0\}\times BO$ (only relevant for $n=1$). It is thus enough to prove that the two maps from $S^{n-1}$ to $BO$ are freely homotopic.

  In \cite{MR948769} Giroux proves that if a Lagrangian has a generating function then the Gauss map is null-homotopic, which is unsurprising here, since $L$ is contractible. However, using his explicit null-homotopy we will see the needed identification. Indeed, consider $L \simeq D^n$ mapping into $U/O$ and sending the boundary $S^{n-1}$ to the basepoint. Giroux's null-homotopy is a free null-homotopy of the map from this disk. Tracing what happens at the boundary we get a map:
  \begin{align*}
    S^{n-1} \times I \to U/O
  \end{align*}
  which sends $S^{n-1} \times \{0,1\}$ to the basepoint. The adjoint therefore represents the free homotopy class of the adjoint $S^{n-1} \to \Omega(U/O)$ associated to the original stable Gauss map $L/S^{n-1} \simeq S^n \to U/O$.

  We now recall Giroux's proof and elaborate on some parts to also identify this with the negative eigen-bundle map of the generating function under Bott periodicity. His argument goes as follows: Let $\Lambda(n,2k)$ be the linear Lagrangians in $\C^{n+2k}$ that transversely intersect $H=\C^n\oplus \R^{2k}$. Such Lagrangians symplectically reduce to $\C^n$ by transversely intersecting with $H$ and projecting to $\C^n$. The reduction map
  \begin{align*}
    r : \Lambda(n,2k) \to \Lambda(n)=U(n)/O(n)
  \end{align*}
  is a fiber bundle with contractible fibers. We define the standard stabilization of the Gauss map as adding Lagrangian factors given by the differential of the function $\eh \norm{x}^2+\eh \norm{y}^2$ with $(x,y)\in \R^{2k}$ in $T^*\R^{2k}=\C^{2k}$; this adds the Lagrangian factor $(1+i)\R^{2k}$. When the tangent space at a point $l\in L$ is also a graph (in $T^*\R^n$) the stabilized linear Lagrangian at that point is the graph of a block matrix:
  \begin{align*}
    \begin{pmatrix}
      A & 0 \\
      0 & I
    \end{pmatrix},
  \end{align*}
  where the two diagonal blocks are of size $n\times n$ and $2k\times 2k$ respectively. In particular, for $l$ close to $\partial L=S^{n-1}$ the first block $A$ is $0$.

  Sometimes it is more natural to define these stabilizations by adding the factor $\R^{2k}$ (corresponding to the second block in the matrix also being 0) or $i\R^{2k}$ (which is not a section). However, since we want the result both to be graphical like this over $\R^{2k}$ and to lie in $\Lambda(n,2k)$, we need it to be transverse to both of these.

  The fact that $r$ has contractible fibers means that any two sections in $\Lambda(n,2k)$ are homotopic. In particular, the stabilization is a section, and the differential of the generating function $F:D^n \times \R^{2k} \to \R$ also defines a section; hence there is a homotopy between these two. Giroux's final observation is that since all the Lagrangians in $dF$ are transverse to $i\R^{n+2k}$ ($dF$ is graphical), its Gauss map is null-homotopic.

  Unwinding the construction of this homotopy locally for a point in $S^{n-1} \subset L$ where the tangent space initially agreed with the zero section we get the following. The stabilization from the generating function is equivalent to adding the differential of another non-degenerate quadratic form. In fact, it adds the Hessian of $F$ which varies with the point in $S^{n-1}$ and which we assume, without loss of generality, to have only eigenvalues $\pm 1$. At these points, the Hessian of $F$ is in block form:
  \begin{align*}
    \begin{pmatrix}
      0 & 0 \\
      0 & B
    \end{pmatrix},
  \end{align*}
  where the diagonal blocks again have size $n\times n$ and $2k\times 2k$ respectively and $B$ is a non-degenerate symmetric matrix with eigenvalues $\pm 1$. As both this $B$ and $I$, defining the two sections, are non-degenerate, we may view the lower $2k\times 2k$ part as defining graphs over $i\R^{2k}$ (taking values in $\R^{2k}$). Hence the homotopy between the two sections in $\Lambda(n,2k)$ can simply be described by a convex interpolation of these graphs (keeping the 0 constant in the first $n$ factors). This interpolation may pass through $i\R^{2k}$ and hence at times it will not be graphical over the base $\R^{2k}$. In fact, its intersection with $i\R^{2k}$ happens exactly halfway through the interpolation and is precisely $i$ times the negative eigen-bundle of the Hessian $B$. This is exactly how the space of minimal geodesics in the corresponding component of $\Omega (U(2k)/O(2k))$ mapping to $\Gr_k(\R^{2k})$ is identified in the proof of Bott-periodicity in \cite{MR0163331}.

  One might object that the choice of initial stabilization seems important. After all, we could have used the differential of another quadratic form (e.g. $-\eh \norm{x}^2-\eh \norm{y}^2$) and obtained a different result. This will indeed produce other null-homotopies of the map $D^n \to U/O$. However, as $D^n$ is contractible, such null-homotopies only differ (up to homotopy) by the action of $\pi_1$, and this is precisely the ambiguity of which component of $\Z\times BO$ we land in, which we addressed at the beginning of the proof.
\end{proof}


%% file: Tubes.tex
\section{The Fibration Sequence} \label{sec:flabby-map}

In this section we define and prove the unstable version
\begin{align*}
  \M_k^l \to \U_k^l \to \H_k^l  
\end{align*}
of the fibration sequence from Theorem~\ref{thm:3}. We define stabilization maps between these sequences that increase $k$ and $l$. We finally prove that the fibration structures are compatible with these stabilization maps, resulting in the colimit version of the sequence:
\begin{align*}
  \M_\infty \to \U_\infty \to \H_\infty
\end{align*}
In the final sections, we identify the space $\U_\infty$ as $F/O$, the space $\H_\infty$ as the stable $h$-cobordism space of a point, and the map $\U_\infty \to \H_\infty$ as the Hatcher--Waldhausen map.

With $\cC_k^l$ as in Section~\ref{sec:lagrangian-gauss-map}, let $(\cC_k^l)^1 \subset \cC_k^l$ be the subspace of functions that have 1 as a regular value. We define
\begin{itemize}
\item $\H_k^l \subset (\cC_k^l)^1$ as the connected component containing $Q_k^l$.
\item $\U_k^l \subset \H_k^l$ as the subspace of functions where there is only one critical point with value below $1$. We further require that it has critical value $0$ and is non-degenerate.
\end{itemize}
Again, by considering level sets, the unique critical point of any $f \in \U_k^l$ must have Morse index $k$.

For $k+l\geq 6$, the space $\H_k^l$ is homotopy equivalent to the $h$-cobordism space denoted $\H(S^{l-1}\times D^k)$ (see Section~\ref{sec:Hatcher-Wald-map}).

We define the \emph{shrinking} (\emph{expansion}) $b\se f$ of a function $f\in \cC_k^l$ by a constant $b\in (0,1]$ ($b\geq 1$) by the formula
\begin{align} \label{eq:Tubes:1}
  (b \se f)(z) = b^2f(b^{-1}z).
\end{align}
Note that if $f=Q_k^l$ on some set $A$ then $b\se f=Q_k^l$ on $bA$. Furthermore, we have
\begin{align}
  \label{eq:Tubes:2}
  \norm{b\se f - Q_k^l}_{C^1} \leq \max(b,b^2)\norm{f-Q_k^l}_{C^1}
\end{align}
and the critical values of $b\se f$ are those of $f$ scaled by $b^2$.

\begin{lemma} \label{lem:Tubes:1}
  The space $\M_k^l$ from Definition~\ref{def:Gauss:1} is a subspace of $\U_k^l$.
\end{lemma}

\begin{proof}
  The only condition we need to check is that it actually lies in the connected component of $(\cC_k^l)^1$ containing $Q_k^l$. We prove this by taking any $f\in \M_k^l$ and continuously shrinking $f$. Indeed, $b\se f$ for $b\in [\epsilon,1]$ defines a path from $\epsilon \se f$ to $f$. For small $\epsilon$ we have $\norm{\epsilon\se f-Q_k^l}_{C^1}$ as small as needed. We may make it small enough so that the convex interpolation $(1-t)(\epsilon\se  f) + tQ_k^l$ has 1 as a regular value.
\end{proof}

We define the subspace $K\cC_k^l \subset \cC_k^l$ as those functions $f\in \cC_k^l$ where $f-Q_k^l$ has compact support. We define
\begin{align}
  \label{eq:Tubes:1k}
  K\M_k^l \subset \M_k^l, \quad    K\U_k^l \subset \U_k^l, \quad   K\H_k^l \subset \H_k^l.
\end{align}
using the same condition. We similarly define the subspace $D_b\cC_k^l \subset \cC_k^l$ as those functions $f\in \cC_k^l$ where the support of $f-Q_k^l$ is contained in $D_b^{k+l}$. We define
\begin{align}
  \label{eq:Tubes:1b}
  D_b\M_k^l \subset \M_k^l, \quad    D_b\U_k^l \subset \U_k^l, \quad   D_b\H_k^l \subset \H_k^l.
\end{align}
using the same condition. Note that $K\cC_k^l = \cup_b D_b\cC_k^l$.

\begin{lemma} \label{lem:Tubes:2}
  The inclusions in Equation~(\ref{eq:Tubes:1k}) are homotopy equivalences.
\end{lemma}

\begin{proof}
  This follows from Lemma~\ref{lem:Gauss:1} and the fact that all the spaces are defined by conditions on critical points (and critical values).
\end{proof}

Consider the subspace $(\H_k^l)^{\geq 1} \subset \H_k^l$ defined by those $f$ where all values above $1$ are also regular. We define the homotopy
\begin{align} \label{eq:Tubes:3}
  H_t(f) =
  \begin{cases}
    \pare*{1-t + t\pare*{\tfrac{1}{c_f^2+9}}} \se f & t \in [0,1] \\
    (2-t)\pare*{\pare*{\tfrac{1}{c_f^2+9}}\se f} + (t-1)Q_k^l & t\in [1,2]
  \end{cases}
\end{align}
where again $c_f=\norm{f-Q_k^l}_{C^1}+1$.

\begin{lemma}
  The homotopy $H_t$ defines a homotopy from the identity on $(\H_k^l)^{\geq 1}$ to the constant map at $Q_k^l$. Hence $(\H_k^l)^{\geq 1}$ is contractible.
\end{lemma}

\begin{proof}
  The shrinking scales the critical values by a factor less than 1, so $H_t(f)$ has all values in $[1,\infty)$ as regular values for $t\in [0,1]$. The bound (by Equation~(\ref{eq:Tubes:2}))
  \begin{align*}
    \norm*{\pare*{\tfrac{1}{c_f^2+9}}\se f - Q_k^l}_{C^1} \leq \frac{c_f}{c_f^2+9} \leq \tfrac19
  \end{align*}
  shows that all values in $[1,\infty)$ are regular for $H_t(f)$ with $t\in [1,2]$. Indeed, any point $z$ where $H_t(f)(z)=1$ has $Q_k^l\geq \tfrac89$ and the gradient norm of $Q_k^l$ at these points is much larger than $\tfrac19$.
\end{proof}

The composition $\M_k^l \to \U_k^l \to \H_k^l$ lands in $(\H_k^l)^{\geq 1}$ and is therefore null-homotopic by the above lemma.

\begin{lemma} \label{lem:oldprop:Tubes:1}
  The sequence
  \begin{align*}
    \M_k^l \subset \U_k^l \subset \H_k^l
  \end{align*}
  with the null-homotopy above is a homotopy fibration sequence.
\end{lemma}

\begin{proof}
  Fix a $b>1$ and let $\wH$ be the space of smooth codimension 1 submanifolds in $\R^{k+l}$ which agree with $\{Q_k^l=1\}$ outside $D_b^{k+l}$. We give $\wH$ the $C^\infty$ topology (using spaces of local sections in normal bundles as charts). We claim that the map $D_b\H_k^l \to \wH$ given by $f \mapsto \{f=1\}$ and the composition $D_b\U_k^l \to \wH$ through $\H_k^l$ are both fiber bundles. It follows from this that $D_b\H_k^l \to \wH$ is a homotopy equivalence, since its fibers are contractible.

  To prove the claim we let $M\in \wH$ be given. Note that $M=\{Q_k^l=1\}$ outside $D_b^{k+l}$ and is thus transverse to $S_b^{k+l-1}$. So we may pick a smooth tubular neighborhood $M\times [-2,2] \subset \R^{k+l}$ which over any point $x\in M \cap S_b^{k+l-1}$ is mapped into $S_b^{k+l-1}$. Using this we can identify an open neighborhood $U\subset \wH$ of $M$ with smooth maps $s: M \to (-1,1)$ with support in $M\cap D_b^{k+l}$ (sections in the normal bundle). Pick a map $\psi: [-2,2] \to [0,1]$ which is $1$ on $[-1,1]$ and $0$ close to $\{-2,2\}$. For each $s:M \to (-1,1)$ let $X_s$ be the vector field on $M\times [-2,2]$ defined by $\psi(t) s(x)\pd{}{t}$ at points $(x,t)$. This is equal to $s(x)\pd{}{t}$ in the interval between $s(x)$ and $0$---hence its time-1 flow, which we denote by $\varphi^s$, on $M\times [-2,2]$ takes the zero section $M$ to $s(M)$. As $\varphi^s$ is the identity outside $D_b^{k+l}$ it acts continuously on $D_b\U_k^l$ and $D_b\H_k^l$ (by Lemma~\ref{lem:Gauss:2}). We thus get trivializations of both bundles in the claim by composing functions over $s\in U$ from the right with $\varphi^s$.
  
  The fiber over the basepoint of the fiber bundle $D_b\U_k^l \to \wH$ consists of functions $f\in \U_k^l$ such that $\{f= 1\}=\{Q_k^l=1\}$. We claim that this deformation retracts onto the subspace $W \subset D_b\M_k^l$ defined by those $f\in D_b\M_k^l$ which agree with $Q_k^l$ on $\{Q_k^l\geq 1\}$. To prove this claim, we may first use an isotopy to make $f=Q_k^l$ at $\{Q_k^l=1\}$ to all orders (this uses that $\{f=1\}=\{Q_k^l=1\}$ and that $f$ has $1$ as a regular value). After this we may convexly interpolate on the set $\{Q_k^l\geq 1\}$ between $f$ and $Q_k^l$ (leaving $f$ as it is on $\{Q_k^l\leq 1\}$ where, by the definition of $\U_k^l$, it has only the one critical point).

  We also claim that $W \subset D_b\M_k^l$ is a homotopy equivalence. Indeed, shrinking with factor $b^{-1}$ defines a map back $D_b\M_k^l \to W$. The compositions in both directions are homotopic to the identity using the shrinking homotopies considered above.

  The final claim is that this is compatible with the null-homotopy $H_t$ given in Equation~(\ref{eq:Tubes:3}) above. Indeed, we first note that the homotopy preserves $D_b(\H_k^l)^{\geq 1} = D_b\H_k^l \cap (\H_k^l)^{\geq 1}$. Moreover, it even preserves the property $f=Q_k^l$ on $\{Q_k^l\geq 1\}$. It follows that the composition $W \to D_b\H_k^l \to \wH$ constantly maps to the basepoint in $\wH$ during this homotopy. This shows that the map from $W$ and hence $D_b\M_k^l$ to the homotopy fiber of $D_b\U_k^l \to \wH$ is a homotopy equivalence. Hence the induced map to the homotopy fiber of the $D_b\U_k^l \to D_b\H_k^l$ is also a homotopy equivalence. This last map to the larger homotopy fiber is compatible with unions for larger and larger $b$ and shows that
  \begin{align*}
    K\M_k^l \to K\U_k^l \to K\H_k^l 
  \end{align*}
  is a homotopy fibration sequence with this particular null-homotopy. As the inclusions of these spaces into the original sequence are homotopy equivalences (Lemma~\ref{lem:Tubes:2}) and the map to the homotopy fiber is again compatible with these inclusions, the result follows.
\end{proof}

We define stabilization maps
\begin{align*}
  s_- : \H_k^l \to \H_{k+1}^l \qquad \textrm{and} \qquad s_+ : \H_k^l \to \H_{k}^{l+1}
\end{align*}
by
\begin{align} \label{eq:9}
  s_-(f)(x_{k+1},x,y) = & - x_{k+1}^2 + f(x,y) \\
  s_+(f)(x,y,y_{l+1}) = & f(x,y) + y_{l+1}^2.
\end{align}
Note that $s_- \circ s_+ = s_+ \circ s_-$.

The function $s_-(f)$ only has critical points of the type $(0,x,y)$ where $(x,y)$ is critical for $f$. Similarly for $s_+(f)$. This shows that $s_\pm$ preserves the three types of subspaces $\U_k^l$, $\M_k^l$ and $\H_k^l$ and we get a commutative diagram
\begin{align*}
  \xymatrix{
    \M_k^l \ar[r] \ar[d]^{s_\pm} & \U_k^l \ar[r] \ar[d]^{s_\pm} &  \H_k^l \ar[d]^{s_\pm} &  \\
    \M_{k'}^{l'} \ar[r] & \U_{k'}^{l'} \ar[r] &  \H_{k'}^{l'}
  }
\end{align*}
We define
\begin{align*}
  \M_\infty = \colim_{k\to \infty,l\to \infty} \M_k^l, \qquad \U_\infty = \colim_{k\to \infty,l\to\infty} \U_k^l, \qquad \H_\infty = \colim_{k\to \infty,l\to\infty} \H_k^l
\end{align*}
all under these stabilization maps. The stabilizations also preserve the subspaces $(\H_k^l)^{\geq 1}$, and we have a contractible subspace $\H_\infty^{\geq 1} =\colim_{k \to \infty,l\to \infty} (\H_k^l)^{\geq 1} \subset \H_\infty$. The homotopy in Equation~(\ref{eq:Tubes:3}) is compatible with stabilizations and defines a null-homotopy of the identity on $\H_\infty^{\geq 1}$, which defines a null-homotopy of the composition
\begin{align*}
  \M_\infty \to \U_\infty \to \H_\infty.
\end{align*}

\begin{proposition} \label{prop:oldcor:Tubes:1}
  The sequence 
  \begin{align*}
    \M_\infty \to \U_\infty \to \H_\infty.    
  \end{align*}
  with the above null-homotopy is a homotopy fibration sequence.
\end{proposition}

\begin{proof}
  This follows since the map from each $\M_k^l$ to the homotopy fiber of $\U_k^l \to \H_k^l$ is a homotopy equivalence by the above lemma, and this is compatible with taking the colimit.
\end{proof}


%% file: Parhandles.tex
\section{Parameterized handle attachments} \label{sec:param-handle-attachm}

In this section we give explicit constructions of functions $f\in \cC_k^l$ depending on certain maps defined on vector spaces (similar to $(V,i)$ from the introduction used to describe the Hatcher--Waldhausen map). In the case considered in the introduction these functions will have level sets $\{f=1\}$ that are smoothings of the boundary of $\{Q_k^l \leq -1\} \cup T$ where $T$ is a tubular neighborhood of the disk image (pink in Figure~\ref{Fig:hwpic}). The idea is to implant a standard Morse function in a certain chart of the tubular neighborhood; we refer to the resulting function as a \emph{mountain pass} function. As the resulting function can be thought of as having the original map as a parameterization of its unstable manifold at the unique critical point with value less than 1, we will refer to these maps as \emph{unstable disk maps}.

There will be two very important cases for later, which we handle simultaneously by making the construction more general: one in which the unstable disks have dimension $k$, as discussed in the introduction, and one in which $l=0$ and the disk dimension is anywhere between $0$ and $k$. This produces a level set which is a so-called tube (see Section~\ref{sec:Hatcher-Wald-map}).

For a vector space $V \subset \R^{k+l}$ we will use the coordinate notation $(v,w)\in V\times V^\perp \subset \R^{k+l}$. We will also consider the generalization
\begin{align*}
  Q_V : \R^{k+l} \to \R
\end{align*}
defined by $Q_V(v,w) = -\norm{v}^2+\norm{w}^2$ (which we do not assume to have signature $l-k$). We will let $\vec v$ denote the (identity) vector field on $V$, that is $\vec v(v) = v \in V=T_vV$. Similarly, we denote the identity vector field on $V^\perp$ by $\vec w$. By abuse we also use this notation for the associated parallel vector fields on $\R^{k+l}$. For example, we have $\nabla Q_V = -2\vec v+2\vec w$.

\begin{defn} \label{defn:un}
  Let $\UN_k^l$ be the space of triples $\tilde e=(V,e,c)$ with $V \subset \R^{k+l}$ a vector space, $e=(e_x,e_y) : V \to \R^k \times \inte D_c^l$ a smooth map, called an \emph{unstable disk map}, and $c\geq 1$ such that
  \begin{itemize}
  \item $e$ composed with the projection to $\R^k$ is proper.
  \item $e$ is the linear inclusion $V\subset \R^{k+l}$ close to $0$.
  \item $\vec v(\norm{e_x}) > 0$ for $\norm{e_x} \geq c$.
  \end{itemize}
  We topologize this using local trivializations of the canonical bundle over $\sqcup_d \Gr_d(\R^{k+l})$ and the product topology on $(e,c)$, where we use the $C^\infty$ topology on $e$.
\end{defn}

The mountain pass functions themselves will be functions in the following space.
\begin{defn} \label{defn:Fkl}
  Let $\F_k^l \subset (\cC_k^l)^1$ be the subspace of those $f$ satisfying:
  \begin{itemize}
  \item There is a $V\in \sqcup_d \Gr_d(\R^{k+l})$ so that $f =Q_V$ close to $0$.
  \item The only critical point with value $\leq 1$ is $0$.
  \end{itemize}
\end{defn}
The parameter $c$ will determine how wide the mountain is, and we will introduce a second parameter $\delta$ that will control how narrow the mountain pass is. To make this precise we will need a few explicit constructions.

Pick a smooth even function $\varphi: \R \to \R$ such that
\begin{itemize}
\item $\varphi(t)=0$ for $|t|\leq 1$.
\item $\varphi'(t) > 0$ for $t>1$.
\item $\varphi(t) = t^2$ for $t\geq 2$.
\end{itemize}
Define
\begin{align} \label{eq:FmodO:1}
  \tilde Q_k^l(x,y) = -\varphi(\norm{x}) + \varphi(\norm{y}). 
\end{align}
This has $\tilde Q_k^l \in \H_k^l$ and is $0$ on $D^k\times D^l$; it has $\nabla Q_k^l$ as a weak pseudo-gradient. By a \emph{weak} pseudo-gradient $X$ for a function $f$ we mean a vector field such that $X(f)\geq0$ with equality only at critical points. A weak pseudo-gradient where the vector field itself vanishes at all critical points will be called a \emph{strict} pseudo-gradient.

We will need the following technical constructions to be able to implant the mountain pass.

\begin{lemma} \label{lemcor:Parhandles:1}
  There exist $C^\infty$ tubular embeddings $\tau_{\tilde e} : V \times D V^\perp \to \R^k \times \inte D^l_c$, continuous in $\tilde e=(V,e,c) \in \UN_k^l$, such that
  \begin{itemize}
  \item The restriction of $\tau_{\tilde e}$ to $V\times \{0\}$ equals $e$.
  \item $\tau_{\tilde e}$ is the standard identification $V\oplus V^\perp = \R^{k+l}$ in a neighborhood of $0$.
  \item $((c \se \tQ_k^l) \circ \tau_{\tilde e})(v,w)$ does not depend on $w \in V^\perp$.
  \end{itemize}
\end{lemma}

\begin{proof}
  As $e$ lands in $\R^k \times D_{c}^l$ we note that
  \begin{align*}
    (c \se \tQ_k^l)(x) = -(c \se \varphi)(\norm{x})  
  \end{align*}
  for $x\in \R^k \times D_c^l$. The last bullet point in Definition~\ref{defn:un} implies that $e$ is strictly transverse to the foliation $S_a^{k-1} \times D_c^l, a\geq c$. These are the level sets of $c \se \tQ_k^l$ in $\R^k \times D_{c}^l$ except that $D_{c}^k \times D_{c}^l$ is one large (degenerate) level set.

   Let $\nu_e \to \im e$ denote a choice of normal bundle to $e$ which, close to zero, is $V^\perp$ and outside $D_c^k \times \inte D_c^l$ is contained in the distribution defined by the foliation $S_a^{k-1} \times D_c^l, a \geq c$. This is possible since the last condition in Definition~\ref{defn:un} makes $\im e$ transverse to this foliation. Give $\nu_e$ the induced metric. Pick a tubular neighborhood $T : D\nu_e \to \R^k \times \inte D_c^l$, continuous in $e$, such that $T(z,w) = z+w$ for $z\in \im e$ close to $0$ and such that $T(z,-)$ has image in a single leaf of the foliation for $\norm{z} \geq c$. Pick a family of isometric isomorphisms $\Phi_e : \nu_e \cong \im e \times V^\perp$, that equal the identity close to $0$, and define
  \begin{align*}
    \tau_{\tilde e} : V \times DV^\perp \xrightarrow{e\times \id} \im e \times D V^\perp \xrightarrow{\Phi_e^{-1}} D \nu_e \xrightarrow{T} \R^k \times \inte D^l_c.
  \end{align*}
\end{proof}

We fix such a choice of $\tau_{\tilde e}$ in the following, and use this to construct the mountain pass functions.

\begin{lemma} \label{lem:FmodO:7}
  There is a map $\MP : \UN_k^l\times (0,1] \to \F_k^l$ such that with
  \begin{align*}
    W=\tau_{V,e,c}(V\times D_{\delta}V^\perp) \cap \{c\se \tQ_k^l \geq -3\}
  \end{align*}
  we have
  \begin{itemize}
  \item $\MP(V,e,c,\delta)$ equals $c\se\tilde Q_k^l+2$ outside $W$.
  \item $\MP(V,e,c,\delta)$ equals $Q_V$ in a neighborhood of $0$.
  \item $\MP(V,e,c,\delta) \circ \tau_{V,e,c}$ has $\nabla Q_V$ as a strict pseudo-gradient on $V \times (\inte D_\delta V^\perp)$.
  \end{itemize}
\end{lemma}

This construction moves up the values of $c\se\tilde Q_k^l$ by $2$ and implants a narrow mountain pass in $W$ along $e$ using the tubular neighborhood $\tau_{\tilde e}$. How narrow the pass is depends on $\delta$.

\begin{proof}
  The properness in Definition~\ref{defn:un} implies that $W$ is compact. Define $T=V\times D_\delta V^\perp$ and $W' = \tau_{\tilde e}^{-1}(W)\subset T$. Note that $W'$ is also compact. Consider the function $P : T \to \R$ given by
  \begin{align*}
    P = (c\se\tQ_k^l) \circ \tau_{\tilde e\mid W'}.
  \end{align*}
  The last property of the tube from Lemma~\ref{lemcor:Parhandles:1} is then that $P(v,w)=P(v,0)=(c \se \tQ_k^l)(e(v))$. This together with the definition of $\UN_k^l$ implies that
  \begin{align*}
    \vec v(P) \leq 0 \qquad \textrm{and} \qquad (\vec v(P) < 0 \quad \textrm{for} \quad P(v,0)<0).
  \end{align*}  
  Fix a smooth family of smooth functions $\psi_\delta : [0,\infty) \to [-2,0]$ for $\delta\in(0,1]$ such that
  \begin{itemize}
  \item $\psi_\delta(t)=-2+t^2$ for small $t$.
  \item $\psi_\delta'(t)>0$ for $t\in (0,\delta)$.
  \item $\psi_\delta(t)=0$ for $t\geq \delta$.
  \end{itemize}
  Fix a smooth function $\opsi : \R \to [0,1]$ so that
  \begin{itemize}
  \item $0 \leq \opsi'(t) < \tfrac12$ for all $t$.
  \item $\opsi(t)=0$ for $t\leq -3$.
  \item $\opsi(t)=1$ for small $t\geq 0$.
  \end{itemize}
  Define the alternative function
  \begin{align*}
    \tilde P(v,w) =
    \begin{cases}  
      P(v,0) + 2 + \opsi(P(v,0))\psi_\delta(\norm{w}) & (v,w) \in W' \\
      P(v,0) + 2 & (v,w) \in T\setminus W'
    \end{cases}
  \end{align*}
  On the boundary of $W'$ this equals $P(v,0)+2$ to all orders. Indeed, either $P(v,0)=-3$ or $\norm{w}=\delta$ on this boundary. It also equals $\norm{w}^2$ in a neighborhood of $0$, and it has
  \begin{align*}
    \vec v(\tilde P) = \vec v(P) (1 + \opsi'(P)\psi_\delta(\norm{w})) \leq 0
  \end{align*}
  with equality on $T$ only at points where $\vec v(P)=0$. We also have
  \begin{align*}
    \vec w(\tilde P(v,w)) = \opsi(P)\norm{w} \psi_\delta'(\norm{w}) \geq 0
  \end{align*}
  with equality on $T$ only when $\opsi(P)=0$ or $\norm{w}\in\{0,\delta\}$. It follows that
  \begin{align*}
    \nabla Q_V = -2\vec v + 2\vec w
  \end{align*}
  is a weak pseudo-gradient for $\tilde P$ on $T$. It is strict outside the union of $\partial T$ and $Z:=\{\norm{w}=0\}\cap \{\tilde P = 0\}$, which is in the interior of $W'$. All points in $Z$ are critical points for $\tilde P$. We can remove these, except the one at $0$, by adding a $C^1$-small function $\psi$ with support in $\inte W'$ which equals $-\norm{v}^2$ in a neighborhood of $0$ and such that $\vec v(\psi)<0$ on $Z$. Now $\nabla Q_V$ is a strict pseudo-gradient on the interior of $T$ and the function equals $Q_V$ in a neighborhood of $0$. Implanting the resulting function on $W \subset \R^{k+l}$ using $\tau_{\tilde e}$ and extending by $c\se \tilde Q_k^l+2$ proves the lemma.
\end{proof}

We will need to adjust the parameters $c$ and $\delta$ depending on an already given function $f \in \F_k^l$ for which we have a given candidate for an unstable disk map. The following definition and corollary make what we need precise.

\begin{defn} \label{defn:unstabledisk}
  An element $(V,e,c) \in \UN_k^l$ is called an unstable disk map \emph{for} $f\in \F_k^l$, if $f=Q_V$ in a neighborhood of $0$ and $f\circ e$ has $-\vec v$ as a strict pseudo-gradient.

  Let $\UNF_k^l$ be the space of pairs $(\tilde e,f)=(V,e,c,f) \in \UN_k^l \times \F_k^l$ such that $(V,e,c)$ is an unstable disk map for $f$.
\end{defn}

\begin{corollary} \label{cor:Parhandles:1}
  There exists a continuous function $\delta_1 : \UNF_k^l \to (0,1]$ such that with $\delta \leq \delta_1(\tilde e,f)$ and $c\geq c_f$ we have that the convex interpolation
  \begin{align*}
    t \mapsto (1-t)f + t\MP(V,e,c,\delta)
  \end{align*}
  defines a path in $\F_k^l$.
\end{corollary}

\begin{proof}
  The function $g=f\circ \tau_{\tilde e}$ has, by the assumptions, $-\vec v$ as a strict pseudo-gradient on $V\times \{0\}$. Close to $0$ we have $g=Q_V$. So the function $dg(\nabla Q_V)$ is positive on $(V \setminus \{0\}) \times \{0\}$ and equals $4\norm{\vec v}^2+4\norm{\vec w}^2$ near $0$. It follows that we can use bounds on the first three derivatives to find a $\delta_1$ so that this is positive away from $0$ on the compact set $W'$ defined as in the lemma above for a $\delta\leq \delta_1$. It follows that $(d\tau_{\tilde e})_*(\nabla Q_V)$ is a strict pseudo-gradient for both $\MP(V,e,c,\delta)$ and $f$ on the interior of $W$ (defined as in the lemma above). It then follows that the convex interpolation cannot have any other critical points than $0$ in the interior of $W$.

  For $c \geq c_f$ we see that $\nabla Q_k^l$ is a weak pseudo-gradient for both $f$ and $\MP(V, e, c, \delta)$ on the closure of the complement of $W$. It follows that the convex interpolation cannot have critical points other than convex interpolations of critical points; since both functions have all critical values above 1 on this set, so will the convex interpolations.
\end{proof}


%% file: FmodO.tex
\section{Identification of $\U_\infty \simeq F/O$} \label{sec:Identify-FO}

The goal of this section is to identify $\U_\infty$ from the fibration sequence in Proposition~\ref{prop:oldcor:Tubes:1} as $F/O$. We also introduce replacement spaces $\tU_k^l\to \U_k^l$ with compatible stabilizations, which give an equivalence in the colimit and which will be convenient in the next section.

We first generalize the spaces $\cC_k^l$ to spaces $\oC_k^l$, which consist of pairs $(f,X)$ with $f\in \cC_k^l$ and $X$ a smooth vector field on $\R^{k+l}$ such that:
\begin{itemize}
\item $\norm{X-\nabla Q_k^l}_{L^\infty} < \infty$.
\item $X(f)\geq 0$ with equality only at critical points (a weak pseudo-gradient).
\end{itemize}

Similarly, to $c_f$ we define
\begin{align}   \label{eq:FmodO:2}
  c_{f,X}-1 = \norm{(f,X)-(Q_k^l,\nabla Q_k^l)}_{C^1} = \max(\norm{f-Q_k^l}_{C^1},\norm{X-\nabla Q_k^l}_{L^\infty})
\end{align}
and topologize $\oC_k^l$ as a subspace of the product of two $C^\infty$ spaces and $\R$, identifying it with triples $(f,X,c_{f,X})$ so that $c_{f,X}$ defines a continuous function. We again define the subspace
\begin{align*}
  D_b\oC_k^l \subset \oC_k^l
\end{align*}
where both the support of $f-Q_k^l$ and $X-\nabla Q_k^l$ are contained in $D_b^{k+l}$, and again this inherits the $C^\infty$ topology.

As before, we define a homotopy $B_t=(B_t^f,B_t^X) : \oC_k^l \to \oC_k^l$ by
\begin{align*}
  B_t^f(f,X)(z) &= (1-t\varphi(\norm{z}-c_{f,X}))f(z) + t\varphi(\norm{z}-c_{f,X})Q_k^l(z), \\
  B_t^X(f,X)(z) &= (1-t\varphi(\norm{z}-c_{f,X}))X(z) + t\varphi(\norm{z}-c_{f,X})\nabla Q_k^l(z)
\end{align*}
with $\varphi$ as used in Equation~(\ref{eq:Gauss:2}).

\begin{lemma} \label{lem:Gauss:1b}
  The homotopy $B_t : \oC_k^l \to \oC_k^l$ is well-defined and for each $t$ the functions $B^f_t(f,X)$ and $f$ are equal near their critical points.
\end{lemma}

\begin{proof}
  The proof is almost the same as the proof of Lemma~\ref{lem:Gauss:1} (noting that $c_{f,X} \geq c_f$). The difference is that here we also need to prove that $B_t^X(f,X)$ is actually a weak pseudo-gradient for $B_t^f(f,X)$.

  We only need to consider points where $\norm{z} > c_{f,X}$ as otherwise $B_t^f(f,X)=f$ and $B_t^X(f,X)=X$. On this set we have $\norm{\nabla Q_k^l} > 2c_{f,X}$. In Lemma~\ref{lem:Gauss:1} we saw that $\norm{\nabla B_t^f(f,X)-\nabla Q_k^l} < \sqrt2 c_f \leq \sqrt2 c_{f,X} < \tfrac{1}{\sqrt{2}}\norm{\nabla Q_k^l}$. From the definition above we have $\norm{X-\nabla Q_k^l} \leq c_{f,X}<\tfrac12 \norm{\nabla Q_k^l}$. For three vectors $a_1,a_2,b$ with $2\norm{a_i-b} < \sqrt 2\norm{b}$ we have that $\inner{a_i,b}>0$ and $\inner{a_1,a_2}>0$. Combining this we get that both $\nabla Q_k^l$ and $X$ are strict pseudo-gradients for $B_t^f(f,X)$ for $\norm{z} > c_{f,X}$.
\end{proof}

We then consider the subspace $\oU_k^l \subset \oC_k^l$ defined by pairs $(f,X)$ such that
\begin{itemize}
\item $f \in \F_k^l$ from Definition~\ref{defn:Fkl}.
\item $X=\nabla f$ in a neighborhood of $0$.
\end{itemize}
The first condition implies that the unique critical point with value $0$ is actually at $0$. As we saw earlier, this critical point has Morse index $k$; it follows that the signature of the quadratic form is $l-k$. Hence the $d$ in Definition~\ref{defn:Fkl} is in this case equal to $k$.

\begin{lemma} \label{lem:FmodO:4}
  The canonical forgetful map $\oU_k^l \to \U_k^l$ is a homotopy equivalence.  
\end{lemma}
As the proof is standard and would be distracting at this point, we have moved it to Appendix~\ref{sec:proof-stand-results}. The advantages of $\oU_k^l$ are that the unstable manifold emanating from $0$ is flat close to $0$, and the weak pseudo-gradient adds some needed flexibility to the space.

We define the stabilization maps $s_- : \oU_k^l \to \oU_{k+1}^l$ and $s_+: \oU_k^l \to \oU_{k}^{l+1}$ as before on the function part and on the pseudo-gradient part by
\begin{align*}
  s_-^X(f,X) = - 2x_{k+1}\pd{}{x_{k+1}} + X  \quad \textrm{and} \quad   s_+^X(f,X) = X + 2y_{l+1}\pd{}{y_{l+1}}
\end{align*}
where we, by abuse of notation, also denote the relevant pullback of $X$ to $\R^{k+l+1}$ by $X$. With this, the canonical map $\oU_k^l \to \U_k^l$ commutes with stabilizations.

The maps $N_k : \M_k \to \Gr_k(\R^{2k})$ considered in Section~\ref{sec:lagrangian-gauss-map} generalize to maps
\begin{align} \label{eq:FmodO:8}
  N_k^l : \U_k^l \to \Gr_k(\R^{k+l}) \quad \textrm{and} \quad  N_k^l : \oU_k^l \to \Gr_k(\R^{k+l})
\end{align}
which, by abuse of notation, are denoted the same. These are again defined by taking the negative eigenspace of the Hessian at the unique critical point with value $0$ (or equivalently the tangent space of the unstable manifold). We define $s_- : \Gr_k(\R^{k+l}) \to \Gr_{k+1}(\R^{1+k+l})$ by
\begin{align*}
  s_-(V) = \R \oplus V \subset \R \times \R^{k+l} = \R^{1+k+l}
\end{align*}
where the new coordinate is $x_{k+1}$ and $s_+ : \Gr_k(\R^{k+l}) \to \Gr_{k}(\R^{k+l+1})$ by
\begin{align*}
  s_+(V) = V \subset \R^{k+l}\times \R = \R^{k+l+1}
\end{align*}
where the new coordinate is $y_{l+1}$. With this, stabilizations commute, and both versions of $N_k^l$ commute with stabilizations.

Recall that we denote the identity vector field on a vector space by $\vec v$. Using $X$ we may construct a canonical smooth parameterization of the unstable manifold
\begin{align} \label{eq:FmodO:4}
  \un(X) : N_k^l(f) \hookrightarrow \{f \leq 1\}
\end{align}
defined by the conditions
\begin{itemize}
\item $\un(X)$ is the linear inclusion of $N_k^l(f)$ in a neighborhood of $0$.
\item $(d\un(X))(\vec v)=-\tfrac12 X$.
\end{itemize}
Note that these two conditions are not contradictory, since the inclusion locally satisfies this equation. Indeed, $X=\nabla f$ close to $0$ and $f$ is a quadratic form with only $\pm1$ eigenvalues. The conditions uniquely define $\un(X)$ because each radial half-line pointing away from $0$ in $\R^k$ solves a specific ODE (not defined at $0$), and the local condition around zero fixes initial conditions. Importantly, the bound $\norm{X-\nabla Q_k^l} \leq c_{f,X}$ also ensures that the solutions do not run off to infinity in finite time.

We define the map $\oU_k^l \to \UN_k^l$ (from Definition~\ref{defn:un}) by
\begin{align} \label{eq:FmodO:7}
  (f,X) \mapsto (N_k^l(f),\un(X),c_{f,X}).
\end{align}
Indeed, the bound $c_{f,X}$ on $\norm{\nabla Q_k^l -X}_{L^\infty}$ proves both of the following.
\begin{itemize}
\item $-X$ is strictly inward-pointing on the boundary of $\R^k\times D_{c_{f,X}}^l$, which means that $\un(X)$ lands in the interior of this set.
\item The flow of $X$ increases $\norm{x}$ to first order for $\norm{x}\geq c_{f,X}$.
\end{itemize}
Furthermore, $\un(X)$ is proper as there are no critical points with value less than $0$ to which the gradient lines can flow.

By uniqueness we get that
\begin{align*}
  \un(s^X_-(f,X)) : \R \oplus N_k^l(f) \to \{s_-(f) \leq 1\}
\end{align*}
sends $(x_{k+1},v)$ to $(x_{k+1},\un(X)(v)) \in \R^{1+k+l}$. The positive stabilization is easier as it is the same map composed with the inclusion $\R^{k+l} \subset \R^{k+l+1}$.

\begin{lemma}\label{lem:FmodO:3}
  The one-point compactified map $\un(X)^+ : N_k^l(f)^+ \to  \{f\leq 1\}^+$ is a based homotopy equivalence.
\end{lemma}

\begin{proof}
  Standard Morse theory and the fact that we only have the one critical point with value in $[-1,1]$ show that the inclusion $\{f\leq -1\}^+ \cup \im \un(X)^+ \subset \{f\leq 1\}^+$ is a homotopy equivalence. Hence the map
  \begin{align*}
    \im \un(X)^+ / (\im  \un(X)^+ \cap \{f \leq -1\}^+) \to \{f \leq 1\}^+/\{f \leq -1\}^+
  \end{align*}
  is a homotopy equivalence.  The flow of $-X$ shows that both $\{f \leq -1\}^+$ and $\im \un(X)^+ \cap \{f \leq -1\}^+$ are contractible. Hence also $\un(X)^+$ is a homotopy equivalence.
\end{proof}

\begin{defn}
  We let $\tU_k^l \subset \oU_k^l$ be the subspace defined by those $(f,X)$ such that the weak pseudo-gradient $X$ restricted to $\im \un(X)$ points strictly toward $0\in \R^{k+l}$ (except at $0$).
\end{defn}
Pointing strictly toward zero at some point $z$ means that the flow strictly decreases the norm $\norm{z}$. The condition is equivalent to $d(\un(X))(\vec v)$ pointing away from zero. We see that the stabilizations $s_\pm$ preserve these subspaces and thus their restrictions define compatible stabilizations
\begin{align*}
  s_- : \tU_k^l \to \tU_{k+1}^{l} \quad \textrm{and} \quad   s_+ : \tU_k^l \to \tU_{k}^{l+1}.
\end{align*}

In the following we identify $S^k=(\R^k)^+$. Let
\begin{align*}
  F(k) = \Map_*(S^k,S^k)_{\pm 1}
\end{align*}
denote the mapping space of based maps of degree $\pm 1$ with monoid structure given by composition. The basepoint is the identity. The $J$-homomorphism can be represented by the homomorphisms $J_k : O(k) \to F(k)$ sending $A\in O(k)$ to the induced map on $(\R^k)^+$. A delooping of this map is
\begin{align*}
  BJ_k : BO(k) \to BF(k)
\end{align*}
which is the map of classifying spaces that sends a metric vector bundle to its fiberwise one-point compactification.

The next goal is to show that $N_k^l$ lifts to $F/O$. To make this lift explicit we will use the model $\colim_{k,l} (F/O)_k^l$ for $F/O$ where $(F/O)_k^l$ consists of pairs $(V,\theta)$ such that $V \in \Gr_k(\R^{k+l})$ and $\theta : V^+ \to S^k$ is a based homotopy equivalence. Explicitly this is the quotient of $F(k)$ times the highly connected Stiefel manifold of $k$ frames in $\R^{k+l}$ by $O(k)$. It thus fits in the fibration sequence
\begin{align*}
  F(k) \to (F/O)_k^l \to \Gr_k(\R^{k+l}).
\end{align*}
Let
\begin{align*}
  \pi_k : \R^{k+l} \to \R^k
\end{align*}
be the projection onto the $x$ coordinates.

\begin{lemma} \label{lem:FmodO:2}
  The map $(\pi_k \circ \un(X))^+ : (N_k^l(f))^+ \to (\R^k)^+$ is well-defined for each $(f,X) \in \oU_k^l$ and is a homotopy equivalence. This thus defines a canonical lift
  \begin{align*}
    \xymatrix{
      & (F/O)_k^l \ar[d] \\
      \oU_k^l \ar[r]^-{N_k^l} \ar[ur]^{\oN_k^l} & \Gr_k(\R^{k+l}).
    }
  \end{align*}
  given by setting $\theta=(\pi_k \circ \un(X))^+$.
\end{lemma}

\begin{proof}
  By the bound $\norm{f-Q_k^l} \leq c_f$ we have an inclusion $\{f\leq 1\}^+ \subset \{Q_k^l\leq c_f\}^+$. We claim that this is a homotopy equivalence. Indeed, this inclusion is defined for all $f\in \H_k^l$ (not just $\U_k^l$) and the homotopy types are locally constant in $f$. So, the claim follows as the inclusion at $f=Q_k^l$ is a homotopy equivalence and $\H_k^l$ is connected.

  Using Lemma~\ref{lem:FmodO:3} we now see that all maps in the sequence
  \begin{align*}
    N_k^l(f)^+ \xrightarrow{\un(X)^+} \{f\leq 1\}^+ \subset \{Q_k^l \leq c_f\}^+ \xrightarrow{\pi_k^+} (\R^k)^+.
  \end{align*}
  are homotopy equivalences.
\end{proof}

The stabilizations $s_- : (F/O)_k^l \to (F/O)_{k+1}^l$ and $s_+ :(F/O)_k^l \to (F/O)_k^{l+1}$ are defined by $s_-(V,\theta)=(\R\oplus V,\id_{\R} \wedge \theta)$ and $s_+(V,\theta) = (V,\theta)$ (adding only an ambient $y_{l+1}$ coordinate). So, by construction the lifts $\oN_k^l$ commute with stabilizations.

\begin{proposition} \label{prop:Hatcher:1}
  The colimits of the maps
  \begin{align*}
    \oU_\infty \xrightarrow{\oN_\infty} \colim_{k,l\to  \infty} (F/O)_k^l \simeq F/O  \quad \textrm{and} \quad   \tU_\infty \xrightarrow{\oN_\infty} \colim_{k,l\to  \infty} (F/O)_k^l \simeq F/O
  \end{align*}
  are homotopy equivalences. In particular, $\tU_\infty \to \oU_\infty$ is a homotopy equivalence and hence the map $\tU_\infty \to \U_\infty$ is a homotopy equivalence.
\end{proposition}

Before proving this we need a few constructions and lemmas. Let $\oW_k^l \subset \oU_k^l$ denote the subspace of those $(f,X)$ where $f=Q_k^l$ (and thus $X=\nabla Q_k^l$) in a neighborhood of $0$. We have the map of fibrations
\begin{align} \label{eq:FmodO:5}
  \xymatrix{
    \oW_k^l \ar[r] \ar[d]^{i_k^l} & \oU_k^l \ar[r]^-{N_k^l} \ar[d]^{\oN_k^l} & \Gr_k(\R^{k+l}) \ar[d]^= \\
    F(k) \ar[r] & (F/O)_k^l \ar[r] & \Gr_k(\R^{k+l})
  }
\end{align}
where $i_k^l$ is the restriction of $\oN_k^l$.

For any smooth map $\varphi : S^{k-1} \to \R^k-\{0\}$ we call a map
\begin{align*}
  \Phi : \R^k \to \R^k 
\end{align*}
a $b$-\emph{suspension} of $\varphi$ if it satisfies
\begin{align} \label{eq:FmodO:3}
  \Phi(x) = \norm{x}\varphi(\hat x)
\end{align}
when $\norm{x}\norm{\varphi(\hat x)} \geq b$. Let $D_bF(k) \subset F(k)$ be the subspace of one-point compactifications of $b$-suspensions.

\begin{lemma} \label{lem:FmodO:5}
  The map $i_k^l$ takes $D_b\oW_k^l$ to $D_bF(k)$.
\end{lemma}

\begin{proof}
  Outside of $D_b^{k+l}$ we have $d(\un(X))(\vec v) = -\tfrac 12 \nabla Q_k^l$, which means the projection to $\R^k$ satisfies
  \begin{align*}
    d(\pi_k \circ \un(X))(\vec v) = \vec v
  \end{align*}
  (where now $\vec v$ is the identity vector field on $\R^k$) when the image is outside $D_b^{k+l}$, which implies the lemma.
\end{proof}

\begin{lemma} \label{lem:FmodO:6}
  The subspace inclusion $D_bF(k) \to F(k)$ is $(k-2)$-connected.
\end{lemma}

\begin{proof}
  Given a map $\varphi: S^{k-1} \to S^{k-1}$ we can define its suspension as the one-point compactification of $x \mapsto \norm{x}\varphi(\hat x)$. We claim that $D_bF(k)$ deformation retracts onto the subspace of suspensions of smooth maps. Indeed, we construct such a deformation as the concatenation of two deformation retractions: the first is induced by convexly interpolating $\Phi$ as in Equation~(\ref{eq:FmodO:3}) to the map defined globally on $\R^k$ by the same formula. The second is again induced by the same formula while continuously rescaling $\varphi$ to actually land in $S^{k-1}$.

  This means the inclusion is equivalent to the inclusion $F^u(k-1) \to F(k)$ where $F^u(k-1)$ denotes unbased homotopy equivalences $S^{k-1} \to S^{k-1}$ and the map is the suspension. This map is $(k-2)$-connected, as there is a fibration sequence $F(k-1) \to F^u(k-1) \to S^{k-1}$, where the stabilization factors (up to homotopy) as the composition $F(k-1) \to F^u(k-1) \to F(k)$. Furthermore, the stabilization $F(k-1) \to F(k)$ is $(k-1)$-connected by Freudenthal's theorem.
\end{proof}

\begin{proof}[Proof of Proposition~\ref{prop:Hatcher:1}]
  Consider the map of fibration sequences in Equation~(\ref{eq:FmodO:5}) \emph{and} the version where $\oW_k^l \to \oU_k^l$ is replaced by the subspaces $\tW_k^l \to \tU_k^l$. These show that we only need to prove that the corresponding fiber maps $\oW_\infty \to F = \colim_{k \to \infty} F(k)$ and $\tW_\infty \to F$ in the colimit are homotopy equivalences. We prove both cases simultaneously.
  
  \textbf{Surjectivity on homotopy groups}: Let $\theta_- :S^n \to F(k)$ be a based map representing a homotopy class in $\pi_n(F(k))$. That is, $\theta_w \in F(k)$ for each $w\in S^n$, and $\theta_{w_0}=\id_{\R^k}^+$ for the basepoint $w_0 \in S^n$. We may increase this $k$ by stabilizations. The goal is to construct a based family $(f_w,X_w),w\in S^n$ in $\tW_k^l$ (respectively $\oW_k^l$), lifting this homotopy class.

  \textbf{Step 1:} Constructing the unstable disk maps. We will denote these $e_w : \R^k \to \R^{k+l}$ for $w\in S^n$. These will correspond to a family of unstable disk maps, which we then turn into mountain pass functions in Step 2.

  By picking $k$ to be larger than $n+3$ we may assume by Lemma~\ref{lem:FmodO:6} that $\theta_w$ is the one-point compactification of a smooth family of $c$-suspension maps
  \begin{align*}
    \Phi_w : \R^k \to \R^k,\quad w\in S^n
  \end{align*}
  for some chosen $c>1$. We may assume that $\Phi_{w_0}$ is the identity. We may also assume that each $\Phi_w, w\in S^n$ is the identity close to $0$.

  For $l>k+2n+1$ (achieved by positive stabilizations) we may pick a lift of this family to a family of smooth embeddings
  \begin{align*}
    e_w = (\Phi_w,\Phi^y_w): \R^k \to \R^k \times \inte D^l_c
  \end{align*}
  such that for each $w\in S^n$ we have
  \begin{align}
    \label{eq:FmodO:6}
    d e_w (\vec v) = -\tfrac12 \nabla Q_k^l    
  \end{align}
  outside $D_c^k\times D^l_c$. We may assume that each $e_w$ is the standard inclusion close to $0$. We may also assume that $e_{w_0}$ is the standard inclusion globally.

  In the case of $\tW_k^l$ we require this lift to further satisfy the condition that $d(e_w)(\vec v)$ is outward-pointing except at $0$. This may be achieved by starting with a lift $e_w$ as above, and then considering the smooth functions $g_{\hat x,w}(t)=\norm{e_w(t\hat x)}, t\geq 0$ for each $\hat x \in S^{k-1}$ and $w\in S^n$. Each of these is equal to $t$ in a neighborhood of $0$ and strictly positive elsewhere. As $e_w$ follows the flow of $-\tfrac12 \nabla Q_k^l$ outside the set $D_c^k\times D_c^l$ it follows that there is a $t_0>0$ so that $g_{\hat x,w}$ is strictly increasing on $t\geq t_0$ for all $\hat x$ and $w$. Now pick another such family of smooth functions $\tilde g_{\hat x,w} \leq g_{\hat x,w}$ which equals $g_{\hat x,w}$ close to $0$ and on $t \geq t_0$, but whose derivative is positive away from $0$. We may then replace $e_w$ by a small embedded perturbation of the smooth map
  \begin{align*}
    x \mapsto  \frac{\tilde g_{\hat x,w}(\norm{x})}{g_{\hat x,w}(\norm{x})} e_w(x)
  \end{align*}
  whose norm is now strictly increasing along radial lines. Convex interpolation shows that this represents the same homotopy class in $\pi_n(F(k))$.

  Consistent with the notation from the previous section we put $\tilde e_w=(\R^k,e_w,c)$. We have thus defined a lift $\tilde e_{-} : S^n \to \UN_k^l$. Indeed, we made sure that $e_w$ lands in $\R^k \times \inte D_c^l$ and has $(de_w)_*(\vec v)$ equal to $-\tfrac12 \nabla Q_k^l$ outside $D_c^k \times D_c^l$ so that its flow strictly increases $\norm{x}$.

  \textbf{Step 2:} Constructing an unbased lift to $\F_k^l$. Applying Lemma~\ref{lem:FmodO:7} we may use $\tilde e_w$ to define the family of functions $f_w=\MP(\tilde e_w,\delta)$ for some $\delta\in (0,1]$ for each $w\in S^n$. We will postpone arguing that this lands in $\oW_k^l$ (or $\tW_k^l$).
  
  We need to also define a pseudo-gradient $X_w$ for $f_w$ so that $\un(X_w)=e_w$. We define it to be equal to $(d\tau_{e_w})_*(\nabla Q_k^l)$ on a neighborhood of $\im e_w \cap W$ (with $W$ from Lemma~\ref{lem:FmodO:7}). We extend this (by convex interpolation) to $D_c^k\times D_c^l$ using the gradient of $f_w$. We further extend this to $\R^{k+l}$ using the gradient $\nabla Q_k^l$, which by construction is equal to $-2(de_w)_*(\vec v)$ on the image of $e_w$ outside of $W$.
  
  \textbf{Step 3:} Basepoint. We will prove that our lift $(f_{w_0},X_{w_0})$ over the basepoint ${w_0} \in S^n$ is homotopic to $(Q_k^l,\nabla Q_k^l)$ through maps over $e_{w_0}^+=\id_{S^k}\in F(k)$. This implies (in both cases) that the lift can be assumed to be a based lift. This even implies, for $n\geq 1$, that $(f_w,X_w)$ actually lies in $\oW_k^l\subset \oU_k^l$ (or $\tW_k^l$). As this statement about the lift follows from Step 3 below where we consider a more general relative lift we omit it here.

  For the special case when $n=0$, we note that the unstable disk map illustrated in Figure~\ref{Fig:hwpic} (in the introduction) has $V=\R^1$ and the induced proper map $\R^1 \to \R^1$ is homotopic through proper maps to $-\id_\R$. In the introduction we only described the maps from $D^1$ and not all of $\R$, so the black line should be extended out to infinity in both directions within the red area. Applying Step 2 to this unstable disk map produces a mountain pass function lifting the only nontrivial element in $\pi_0(F(1))=\pi_0(F)$.

  \textbf{Injectivity on homotopy groups.} We assume that $(g_{-},Y_{-}) : S^{n-1} \to \oW_k^l$ (or $\tW_k^l$) is mapped by $i_k^l$ to a map $S^{n-1} \to F(k)$, which is the restriction of a map $\theta_-: D^n \to F(k)$. We again need to lift, but this time relative to the lift already given at the boundary. Again we divide the argument into three steps.

  \textbf{Step 1:} Extending the family of unstable disk embeddings $\un(Y_w) : \R^k \to \R^{k+l}$ to $e_w$ for each $w\in D^n$. We first stabilize to get the bounds $k>n+3$ and $l>k+2n+1$ used in Step 1 above. Using Lemma~\ref{lem:Gauss:1b} we may replace $(g_-,Y_-)$ by a family that lands in some $D_{b}\oW_k^l$ (or $D_{b}\tW_k^l$) for some $b>0$. Using a $c\geq b$, every part of the construction in Step 1 above carries over to this relative case, where we extend $e_w=\un(Y_w)$ from $S^{n-1}$ over $D^n$.

  \textbf{Step 2:} Constructing an unbased lift $(f_w,X_w)$ possibly not equal to the original $(g_w,Y_w)$ over $S^{n-1}$. We may construct $(f_w,X_w)$ with $f_w\in \F_k^l$ for each $w \in D^n$ precisely as in Step 2 above using any $\delta$.

  \textbf{Step 3:} We show that the families $(g_w,Y_w)$ and $(f_w,X_w)$ for $w\in S^{n-1}$ are homotopic through such lifts. That is, there is a homotopy between them preserving each unstable map $\un(X_w)=\un(Y_w)$. This implies that we can insert this homotopy in a little collar around $S^{n-1} \subset D^n$ to get the relative lift. It also implies the basepoint statement in Step 3 above so that each $(f_w,X_w)$ actually lies in the right component $\oW_k^l$ (or $\tW_k^l$).

  We note that as the unstable disk maps agree, each ray must solve the same ODE and we must have $X_w=Y_w$ on the image $\un(X_w)=\un(Y_w)=e_w$ for $w\in S^{n-1}$. It follows by Corollary~\ref{cor:Parhandles:1} that if we pick $c\geq c_{g_-,Y_-}$ and $\delta$ small enough, then the convex interpolation defines such a homotopy from $g_w$ to $f_w$ (in either of the two cases). Again, since the choice of pseudo-gradient (fixed to be $-2e_w(\vec v)$ on the unstable disks) is a contractible choice we may make such a choice for the convex interpolation.
\end{proof}


%% file: Hatcher.tex
\section{Identification of the Hatcher--Waldhausen map}\label{sec:Hatcher-Wald-map}

In this section we identify the map $\U_\infty \to \H_\infty$ in the colimit fibration sequence in Proposition~\ref{prop:oldcor:Tubes:1} as the Hatcher--Waldhausen map, thereby finishing the proof of Theorem~\ref{thm:3}.

In \cite{MR686115} Waldhausen defined the so-called tube space, which he denoted $\T_\infty$. Since we will prove that our model using function spaces is equivalent (and is more closely related to the model by Rognes in \cite{MR1282230}), we will also denote our colimit $\T_\infty$. There is a classifying map for the associated spherical fibration $c:\T_\infty \to BF$ (recalled below). In \cite{MR686115} Waldhausen defined the so-called rigid tube map $\rt : BO \to \T_\infty$, which factors the delooping of the $J$-homomorphism as $BJ:BO \xrightarrow{\rt} \T_\infty \xrightarrow{c} BF$. He also identified the homotopy fiber of the map $c:\T_\infty \to BF$ as the stable $h$-cobordism space of a point. We will model this homotopy fiber as a space $\hF_\infty$, and prove that it is also equivalent to our $\H_\infty$. He then considered the map, which is now known as the Hatcher--Waldhausen map, induced on homotopy fibers:
\begin{align*}
  \xymatrix{
    F/O \ar[d]^{\hw} \ar[r] & BO \ar[d]^{\rt} \ar[r]^{J} & BF \ar[d]^= \\
    \hF_\infty \ar[r] &  \T_\infty \ar[r]^c & BF 
  }
\end{align*}
We will first construct a homotopy-commutative diagram
\begin{align} \label{eq:Hatcher:1}
  \xymatrix{
    \hU_\infty \ar[r]^{w_\infty} \ar[d]^{N_\infty} &  \H_\infty \ar[d]^{h_\infty} \\
    BO \ar[r]^{\rt_\infty} &  \T_\infty
  }
\end{align}
where the top map is the colimit of the previously defined (but not named until now) maps $w_k^l : \hU_k^l \to \U_k^l \subset \H_k^l$ and the left vertical map is the colimit of the maps $N_k^l$ from Equation~(\ref{eq:FmodO:8}) extending $N_k=N_k^k$ from Section~\ref{sec:lagrangian-gauss-map}. Part of this construction is to properly recognize the lower map as Waldhausen's rigid tube map. We then lift the vertical maps, in a coherent way, to the homotopy fibers over $BF$ and construct a homotopy-commutative diagram
\begin{align} \label{eq:Hatcher:1b}
  \xymatrix{
    \hU_\infty \ar[r]^{w} \ar[d]^{\tilde N_\infty}_{\simeq} &  \H_\infty \ar[d]_\simeq^{\tilde h_\infty} \\
    F/O \ar[r]^{\hw} &  \hF_\infty
  }
\end{align}
where the left vertical map is now the colimit of the canonical lifts $\tilde N_k^l$ of $N_k^l$ from Lemma~\ref{lem:FmodO:2}.

\subsection{Level-wise construction and a model for tube space}

Let $T_k^l$ be the boundary of a smooth compact tubular neighborhood of $\{0\}\times S^{l-1} \subset \R^k \times \R^l$. Recall that $(\cC_{k+l}^0)^1$ is the space of smooth functions ``close'' to $-\norm{z}^2$ at infinity and with 1 a regular value. We define the tube space $\T_k^l \subset (\cC_{k+l}^0)^1$ as the component containing those $f$ where $\{f=1\}$ is isotopic to $T_k^l$.

We need to make some specific choices to define the missing maps in the diagram in Equation~(\ref{eq:Hatcher:1}). We pick
\begin{figure}[ht]
  \begin{tikzpicture}
    \draw[->] (-1,0) -- (10,0) node[below]{$x$};
    \draw[->] (0,-0.3) -- (0,2.1) node[right]{$y$};
    \draw (3,0.06) -- node[below]{3} (3,-0.06);
    \draw (6,0.06) -- node[below]{6} (6,-0.06);
    \draw (9,0.06) -- node[below]{9} (9,-0.06);
    \draw (-0.06,1) -- node[left]{9} (0.06,1);
    \draw (-0.06,15/8) -- node[left]{15} (0.06,15/8);
    \draw[domain=-1:4, smooth, variable=\x, brown, very thick] plot ({\x}, {\x*\x/8});
    \draw[brown] (3.85,1.9) node[left] {$t^2$};
    \draw[domain=-1:3, smooth, variable=\x] plot ({\x}, {\x*\x/8});
    \draw (3,9/8) to[out=36.87,in=180] (6,15/8) to[out=0,in=165.96] (6.36,14.7/8) -- (8.75,10.2/8) to[out=-14.04,in=86.82] (9,9/8);
    \draw[domain=9:9.6, smooth, variable=\x] plot ({\x}, {90/8-(\x*\x/8)});
    \fill (3,9/8) circle (1.5pt);
    \fill (6,15/8) circle (1.5pt);
    \fill (9,9/8) circle (1.5pt);
    \draw (2,0.5) node[above] {$t^2$};
    \draw (8,1.7) node {slope $<-1$};
    \draw (9.14,0.7) node[right] {$90 - t^2$};
  \end{tikzpicture}
  \caption{The function $\oq$.} \label{Fig:oq}
\end{figure}
a smooth \emph{even} function $\oq:\R \to \R$ such that
\begin{itemize}
\item $\oq(t) = t^2$ for $0 \leq t \leq 3$.
\item $\oq(t) = 90 - t^2$ for $t\geq 9$.
\item $\oq(t) \leq t^2$ for all $t$.
\item $\oq(6) = 15$ is a non-degenerate maximum.
\item $t=0$ and $|t|=6$ are the only critical points for $\oq$.
\item $\oq'(t) < -1$ for $t\geq 9/\sqrt 2 > 6$.
\end{itemize}
This is illustrated in Figure~\ref{Fig:oq}. The last bullet point can be obtained since $\oq$ decreases by $6$ from $t=6$ to $t=9$. We define
\begin{align*}
  \oQ_k^l(x,y) = Q_k^l(x,y) - \norm{y}^2 + \oq(\norm{y}) = -\norm{x}^2 + \oq(\norm{y})
\end{align*}
which equals $Q_k^l$ on $D_{3}^{k+l}$ and satisfies that $\oQ_k^l \in \cC_{k+l}^0$. Using again the shorthand $c_f=\norm{f-Q_k^l}_{C^1}+1 \geq 1$ we define a map
\begin{align*}
  h : \cC_k^l \to \cC_{k+l}^0
\end{align*}
by
\begin{align*}
  h(f) = f - \norm{y}^2  + (c_f\se \oq)(\norm{y}) = f-Q_k^l + (c_f\se\oQ_k^l)
\end{align*}
which leaves $f$ unchanged for $\norm{y} \leq 3c_f$ but outside this set it bends the positive $y$ directions down into negative directions. In particular $h(Q_k^l)=\oQ_k^l$ and by the third bullet point above we have $h(f) \leq f$.

We will need the following results involving $\oq$. Define
\begin{align*}
  F_s(t_1,t_2) = (1-s)(\oq(t_1)+\oq(t_2)) + s\oq(\norm{(t_1,t_2)})
\end{align*}
for each $s\in [0,1]$.

\begin{lemma} \label{lem:Hatcher:6}
  Each function $F_s : \R^2 \to \R$ satisfies that
  \begin{align*}
    \norm{(\nabla F_s)_{(t_1,t_2)}} \leq 1 \quad \Ra \quad F_s(t_1,t_2) \geq 4 \quad \textrm{or} \quad \norm{(t_1,t_2)} \leq 2.
  \end{align*}
  In particular $\oq(t)\geq 4$ if $|\oq'(t)| \leq 1$ and $|t|\geq 2$.
\end{lemma}

\begin{proof}
  For $2 \leq \norm{(t_1,t_2)} \leq 9$ we have that
  \begin{align*}
    \oq(t_1) + \oq(t_2) \geq \min(t_1^2,9) + \min(t_2^2,9) \geq 4 
  \end{align*}
  and $\oq(\norm{(t_1,t_2)}) \geq \min(t_1^2+t_2^2,9) \geq 4$. On the set $\norm{(t_1,t_2)} \geq 9$ and $t_1\geq 9/\sqrt2$ we have
  \begin{align*}
    \pd{F_s}{t_1} \leq (1-s)(-1) + s(-2t_1) \leq -1
  \end{align*}
  and one of these two inequalities is always strict; hence $\norm{\nabla F_s} > 1$ on this set. Similarly, for the symmetric case of $t_2 \geq 9/\sqrt2$.

  The last statement follows from $\oq(t)=F_0(0,t)$ and $|\oq'(t)|=\norm{(\nabla F_0)_{(0,t)}}$.
\end{proof}

We use the lemma to prove a similar statement for the functions $\tilde Q_s \in \cC_{k+l+1}^0$ defined by
\begin{align*}
  \tilde Q_s(x,y,y_{l+1}) = (1-s)(\oQ_k^l(x,y) + \oq(y_{l+1})) + s\oQ_k^{l+1}(x,y,y_{l+1}).
\end{align*}
The following corollary will be needed to handle stabilization maps.

\begin{corollary} \label{cor:Hatcher:3}
  If $\norm{(\nabla \tilde Q_s)_z} \leq 1$ then either $\norm{z} \leq 3$ or $\tilde Q_s(z) \geq 3$.
\end{corollary}

\begin{proof}
  As $\nabla \tilde Q_s(x,y,y_{l+1}) = \pare*{-2x , \tfrac{y}{\norm{y}}\tdd{F_s}{t_1}(\norm{y}),\tdd{F_s}{t_2}(y_{l+1})}$ (for $y\neq 0$) the assumption implies that
  \begin{align*}
    4\norm{x}^2 \leq 1 \qquad \textrm{and} \qquad \norm{(\nabla F_s)_{(\norm{y},y_{l+1})}} \leq 1,
  \end{align*}
  which combined with $\norm{(x,y,y_{l+1})}>3$ gives $\norm{(y,y_{l+1})} >2$. Combining all this with the lemma above yields
  \begin{align*}
    \tilde Q_s(x,y,y_{l+1}) = -\norm{x}^2 + F_s(\norm{y},y_{l+1}) \geq -\tfrac14 + 4 > 3. 
  \end{align*}
\end{proof}

All we will need in this subsection is the following special case.

\begin{corollary} \label{cor:Hatcher:2}
  If $\norm{(\nabla \oQ_k^l)_z} \leq 1$ then either $\norm{z} \leq 3$ or $\oQ_k^l(z) \geq 3$.
\end{corollary}

We now observe that $h(f)$ equals $f$ on $\{\norm{y} \leq 3c_f\}$, where all critical points of $f$ lie. Outside this ball, $h(f)$ will have other critical points. Shrinking by $c_f^{-1}$ we see that
\begin{align*}
  c_f^{-1}\se (h(f)) = c_f^{-1}\se(f-Q_k^l) + \oQ_k^l.
\end{align*}
The first term is bounded in $C^1$ norm by $1$ (by Equation~(\ref{eq:Tubes:2})). So at a critical point we must have $\norm{\nabla \oQ_k^l} \leq 1$. If in addition this point is outside $D_3^{k+l}$ then Corollary~\ref{cor:Hatcher:2} shows that
\begin{align*}
  (c_f^{-1}\se(f-Q_k^l) + \oQ_k^l )(z) \geq -1 + 3 > 1.
\end{align*}
Expanding back, it follows that the critical values of $h(f)$ from the part where it does not equal $f$ are strictly above $c_f^2 \geq 1$. In particular if 1 is regular for $f$ then 1 is regular for $h(f)$.

As $\oq(t)$ has a unique non-degenerate maximum at $t=6$, we have that $\oQ_k^l=h(Q_k^l)$ has the submanifold $\{0\} \times S^{l-1}_6 \subset \R^k \times \R^l$ as a Morse-Bott critical manifold of maxima, and $\{\oQ_k^l=1\}$ is thus the boundary of its tubular neighborhood $\{\oQ_k^l \geq 1\}$ and hence isotopic to $T_k^l$. It follows that $h$ maps the component $\H_k^l$ in $(\cC_k^l)^1$ to the component $\T_k^l$ in $(\cC_{k+l}^0)^1$ and we define the map
\begin{align} \label{eq:Hatcher:5}
  h_k^l : \H_k^l \to \T_k^l
\end{align}
by restricting $h$. We record the following result proved above.

\begin{lemma} \label{lem:Hatcher:7}
  The critical points of $h_k^l(f)$ are those of $f$ on $D_{3c_f}^{k+l}$ where $h_k^l(f)=f$, and the critical points outside this ball all have value strictly larger than $1$.
\end{lemma}

We define the rigid tube map at level $(k,l)$
\begin{align*}
  \rt_k^l : \Gr_k(\R^{k+l}) \to \T_k^l
\end{align*}
by $\rt_k^l(V)=\MP(V,i_V,1,1)$ where $i_V : V \subset \R^{k+l}$ is the linear inclusion and $\MP(-,1):\UN_{k+l}^0 \to \cC_{k+l}^0$ is the mountain pass construction from Lemma~\ref{lem:FmodO:7}. The sublevel set $\{\rt_k^l(V) \leq 1\}$ is (the smoothing of) a tubular neighborhood around $V$ together with a neighborhood of infinity, which is how Rognes described the rigid tube map. We will return to the specifics of this in a later subsection.

We have now defined the maps in the unstable version of the diagram from Equation~(\ref{eq:Hatcher:1}):
\begin{align} \label{eq:Hatcher:3}
  \xymatrix{
    \hU_k^l \ar[d]^{N_k^l} \ar[r]^{w_k^l} &  \H_k^l \ar[d]^{h_k^l} \\
    \Gr_k(\R^{k+l}) \ar[r]^{\rt_k^l}  \ar@{=>}[ur]|{H_t} &  \T_k^l
  }
\end{align}
The two composites around this diagram are now somewhat different. However, as indicated, the diagram commutes up to homotopy, and the rest of this subsection is devoted to constructing this homotopy $H_t$.

For $(f,X) \in \tU_k^l$, consider the following ``shrinking'' homotopy
\begin{align*}
  \un(X)_t : N_k^l(f) \to \R^{k+l}
\end{align*}
from the inclusion of the tangent space $N_k^l(f) \subset \R^{k+l}$ to $\un(X)$ defined by
\begin{align*}
  \un(X)_t(v) =
  \begin{cases}
     v & t=0 \\
     t^{-1}\un(X)(tv) & t \in (0,1]
  \end{cases}
\end{align*}
This is smooth since $\un(X)$ equals the linear inclusion in a neighborhood of $0$. The image of $\un(X)_t$ is the image of $\un(X)$ scaled by $t^{-1}$. These maps are transversely outward-pointing on \emph{all} spheres (this was part of the definition of $\tU_k^l$). This means that we can define a homotopy
\begin{align*}
  \tilde e : \tU_k^l \times I \to \UN_{k+l}^0 \qquad \qquad \textrm{(from Definition~\ref{defn:un})}
\end{align*}
given by
\begin{align*}
  \tilde e_t(f,X) = \pare*{N_k^l(f),\un(X)_t,(1-t) + tc_{h_k^l(f)}}.
\end{align*}
We also define
\begin{align*}
  \delta(t) = (1-t) + t\delta_1(N_k^l(f),\un(X),c_{h_k^l(f)},h_k^l(f)),
\end{align*}
where $\delta_1$ is the function constructed in Corollary~\ref{cor:Parhandles:1}.

We then define the homotopy $H_t$ by
\begin{align*}
  H_t(f,X) =
  \begin{cases}
    \MP(\tilde e_{2t}(f,X),\delta(2t)) & 0\leq t\leq\tfrac12 \\
    (2-2t)\MP(\tilde e_1(f,X),\delta(1)) + (2t-1)h_k^l(f) & \tfrac12 \leq t \leq 1
  \end{cases}
\end{align*}
The first part of $H_t$ expands the plateau and narrows the mountain pass while moving the unstable disk from the linear inclusion of $N_k^l(f)$ to the actual unstable disk of $f$. Note that as this disk lands in $\R^k \times D_{c_f}^l$ it is also an unstable disk for $h_k^l(f)$ as this equals $f$ on this set.

By Corollary~\ref{cor:Parhandles:1} the convex interpolation for $t\in [\tfrac12,1]$ provides well-defined functions in $\F_{k+l}^0 \subset (\cC_{k+l}^0)^1$. Hence these are also in $\T_k^l$.

\subsection{Stabilizations}

In this subsection we define stabilizations $s_\pm : \T_k^l \to \T_{k'}^{l'}$. However, the maps constructed in the previous subsection will not be strictly compatible with these. We resolve this by defining explicit homotopies at each stage that uniquely define a homotopy class of colimit maps. We even lift this structure to the homotopy $H_t$ so that the resulting diagram of colimit maps commutes up to homotopy.

Assume that $f_k : A_k \to B_k$ is a sequence of maps that commute up to homotopy with stabilization maps $s : A_k \to A_{k+1}$ and $s : B_k \to B_{k+1}$. If we pick a homotopy $h_k : A_k\times I \to B_{k+1}$ at each stage from $s \circ f_k$ to $f_{k+1} \circ s$, then we get an induced map
\begin{align*}
  A_\infty \simeq TA_\infty \to B_\infty
\end{align*}
where $TA_\infty$ denotes the infinite mapping telescope. Of course the homotopy class of this map can depend on the choice of homotopies, which is why we need to be a bit careful.

For our tube spaces $\T_k^l$ we define stabilizations
\begin{align*}
  s_-(f) = - x_{k+1}^2 + f 
\end{align*}
and
\begin{align} \label{eq:Hatcher:2}
  s_+(f) = f + (c_f\se\oq)(y_{l+1}).
\end{align}
Note that we have $s_- \circ s_+ = s_+ \circ s_-$ as $c_{s_-(f)} = c_f$. These stabilizations are well-defined because of the following lemmas, and one may check that $  \{s_\pm(\oQ_k^l)=1\} $ is isotopic to $T_{k+1}^l$ or $T_k^{l+1}$ depending on the sign. In fact this follows from the homotopy that we will construct in Lemma~\ref{lem:Hatcher:4}.

\begin{lemma} \label{lem:Hatcher:1}
  For any $f \in \cC_k^l$ we have that the critical points lie in $D_{c_f}^{k+l}$ and their critical values are bounded from below by $1-c_f-c_f^2$.
\end{lemma}

\begin{proof}
  As $\norm{f-Q_k^l} \leq c_f$ and $(\nabla Q_k^l)_{(x,y)} = -2x+2y$ has norm $2\norm{(x,y)}$, the gradient of $f$ is nonzero outside $D_{c_f}^{k+l}$. On $D_{c_f}^{k+l}$ we have
  \begin{align*}
    f(x,y) \geq Q_k^l(x,y) - c_f + 1 \geq -\norm{x}^2 - c_f + 1 \geq 1-c_f-c_f^2.
  \end{align*}
\end{proof}

\begin{lemma} \label{lem:Hatcher:2}
  Any critical point of $s_+(f)$ with critical value less than or equal to 1 has $y_{l+1}=0$. In particular $1$ is a regular value for $s_+(f)$.
\end{lemma}

\begin{proof}
  Any critical point $(z,y_{l+1})$ for $s_+(f)$ consists of a critical point $z$ for $f$ and either $y_{l+1}=0$ or $|y_{l+1}|=6 c_f$. So we need to check that those with $|y_{l+1}|=6c_f$ have value above 1. Lemma~\ref{lem:Hatcher:1} and $\oq(6) = 15$ now yield for these points that
  \begin{align} \label{eq:Hatcher:6}
    f(z) + (c_f\se\oq)(y_{l+1}) > 1 - c_f - c_f^2 + 15c_f^2  > 1.
  \end{align}
\end{proof}

The map $h_k^l$ from Equation~(\ref{eq:Hatcher:5}) satisfies $h_k^l \circ s_- = s_- \circ h_k^l$. However, we need a homotopy for the positive stabilizations.

\begin{lemma} \label{lem:Hatcher:4}
  We have a homotopy $\alpha_t^+$ from $s_+ \circ h_k^l$ to $h_k^{l+1} \circ s_+$ such that
  \begin{align*}
    \{f\leq 1\} \times \{0\} \subset \{\alpha_t^+(f) \leq 1\} 
  \end{align*}
  for each $t\in I$.
\end{lemma}

\begin{proof}
  We define the homotopy between these by
  \begin{align*}
    \alpha_t^+(f)(x,y,y_{l+1})  =h_k^l(f)(x,y) + (((1-2t)c_{h_k^l(f)}+2tc_f)\se\oq)(y_{l+1})
  \end{align*}
  for $0\leq t \leq \tfrac12$, and then the convex interpolation from $\alpha_{1/2}^+(f)$ to $s_+(h_k^l(f))$, which is given by 
  \begin{align*}
    \alpha_t^+(f) = f - Q_k^l + c_f \se \tilde Q_{2t-1}
  \end{align*}
  for $\tfrac12 \leq t \leq 1$, where $\tilde Q_s$ is the function considered in Corollary~\ref{cor:Hatcher:3}. Note that $\alpha_1^+(f)=h_k^l(s_+(f))$ uses that $c_{s_+(f)}=c_f$ for $f\in \H_k^l$.

  The first half of this homotopy is well-defined as Lemmas~\ref{lem:Hatcher:1} and~\ref{lem:Hatcher:7} tell us that $h_k^l(f)$ has all its critical values above $1-c_f-c_f^2$, which means that any expansion factor $\geq c_f$ is enough to make sure that all critical points away from $y_{l+1}=0$ have value strictly larger than 1. This follows as in the proof of Lemma~\ref{lem:Hatcher:2}.

  The second half of the homotopy is also well-defined. Indeed, this follows from Corollary~\ref{cor:Hatcher:3} and the fact that $\norm{c_f^{-1} \se( f-Q_k^l)}_{C^1} < 1$, which implies that at a critical point of $c_f^{-1} \se( f-Q_k^l) + \tilde Q_s$ we have $\norm{\nabla \tilde Q_s} \leq 1$, which then by the corollary implies that
  \begin{align*}
    c_f^{-1} \se(f-Q_k^l) + \tilde Q_s > -1 + 3 >2
  \end{align*}
  on the set where it does not equal $f + y_{l+1}^2$. Hence the critical values of $\alpha_t^+(f)$, for $t\geq \tfrac12$, that are greater than or equal to 1 are all above $2c_f^2>1$. Hence 1 is regular throughout.

  The last statement in the lemma follows as all the maps in the homotopy are equal to $h_k^l(f) \leq f$ on the set $y_{l+1}=0$.
\end{proof}

We define the colimit map $h_\infty : \H_\infty \to \T_\infty$ using the maps $h_k=h_k^k$ and the homotopies $\alpha^k_t=s_- \circ \alpha_t^+$.

Recall that for a vector space $V\subset \R^n$, we let $i_V : V \to \R^n$ denote the inclusion.

\begin{lemma} \label{lem:Hatcher:5}
  We have homotopies $\beta_t^-$ from $s_- \circ \rt_k^l$ to $\rt_{k+1}^l \circ s_-$ and $\beta^+_t : s_+ \circ \rt_k^l \simeq \rt_k^{l+1} \circ s_+$ such that for $V \in \Gr_k(\R^{k+l})$ the inclusions
  \begin{align*}
    \R\oplus V \subset \{\beta_t^-(V) \leq 1\} \quad \textrm{respectively} \quad V \subset \{\beta_t^+(V) \leq 1\}
  \end{align*}
  are unstable disk maps for $\beta_t^\pm(V)$ (cf. Definition~\ref{defn:unstabledisk}) for each $t\in I$.
\end{lemma}

\begin{proof}
    We need to provide a path in $\T_{k+1}^l$ between the two functions $-x_{k+1}^2 + \MP(V,i_V,1,1)$ and $\MP(\R \oplus V,i_{\R\oplus V},1,1)$. These both have unstable disk given by $i_{\R \oplus V}$. It follows by Corollary~\ref{cor:Parhandles:1} that we can follow the path $\MP(\R \oplus V, i_{\R\oplus V},t,t^{-1})$ for $t\geq 1$ until the plateau is large enough and the pass is narrow enough to convexly interpolate to the first function. By construction the map $i_{\R\oplus V}$ is an unstable disk map throughout.

  The argument in the other case is the same, except we use the map $i_V$.
\end{proof}

We use these to define the colimit map $\rt_\infty$ using the maps $\rt_k=\rt_k^k$ and the homotopies $\beta^k$, where $\beta^k$ is the concatenation of the two homotopies $s_- \circ \beta^+_t$ and $\beta_t^- \circ s_+$. 

\begin{lemma} \label{lem:Hatcher:11}
  With the above defined maps, the diagram in Equation~(\ref{eq:Hatcher:1}) commutes up to homotopy.
\end{lemma}

\begin{proof}
  Let $H_u^k=H_u$ be the homotopy defined in the previous subsection at level $k=l$. Consider the map $C : \tU_k\times (\partial I^2) \to \T_{k+1}$ given by
  \begin{itemize}
  \item $C(f,X,t,0) = \alpha_t^k \circ w_k$.
  \item $C(f,X,t,1) = \beta_t^k \circ N_k$.
  \item $C(f,X,0,u) = s \circ H_u^k$.
  \item $C(f,X,1,u) = H_u^{k+1} \circ s$.
  \end{itemize}
  We get a homotopy between the induced colimit maps if we fill out the square. For each point $(t,u)$ in the square we have the map
  \begin{align} \label{eq:Hatcher:10}
    \id_\R \times \un(X)_{\min(2u,1)} : \R \times N_k(f)  \to \R^{k+l+1}.
  \end{align}
  By construction and the inclusion conditions in Lemmas~\ref{lem:Hatcher:4} and~\ref{lem:Hatcher:5}, this is, for each $(t,u) \in \partial I^2$, an unstable disk map for the function defined at that point. This means that using Corollary~\ref{cor:Parhandles:1} we can, as when creating the homotopies, use mountain pass functions to fill in the square.
\end{proof}

\subsection{Comparison to Waldhausen's construction}

In this subsection we lift the map $h_\infty$ from Equation~(\ref{eq:Hatcher:1}) to the map $\th_\infty$ in Equation~(\ref{eq:Hatcher:1b}) and prove that this lift is a homotopy equivalence. The latter is done by relating our spaces to Waldhausen's $h$-cobordism spaces and then using known results about these. We first reduce the problem a bit by replacing the sequence with a sequence of subspaces that behaves better.

Let $M$ be a compact smooth connected manifold, possibly with boundary. We now recall Waldhausen's definition from \cite{MR686115} of a simplicial set $H(M)$ called the $h$-cobordism space. 

By an $h$-cobordism we will mean a smooth compact submanifold $P \subset M \times (0,1)$ of codimension 1 such that
\begin{itemize}
\item $P$ agrees with $M\times \{a\}$ in a neighborhood of $\partial M \times (0,1)$ for some $a\in (0,1)$,
\item $\partial P = \partial M \times \{a\}$, and
\item the set between $M\times \{0\}$ and $P$ is well-defined and is an $h$-cobordism, which we call the associated cobordism.
\end{itemize}
A $p$-simplex in $H(M)$ is a smooth family of $h$-cobordisms parameterized by $\Delta^p$. That is, a smooth submanifold $\tilde P \subset M \times (0,1) \times \Delta^p$ which is transverse to the fibers of the projection to $\Delta^p$ and such that each of these intersections yields an $h$-cobordism. Even though we allow the number $a\in (0,1)$ to vary smoothly over the simplex, we do require that there is a neighborhood around $(\partial M) \times (0,1) \times \Delta^p$ in which $\tilde P$ is equal to $\partial M$ times the graph of this $a : \Delta^p \to (0,1)$. This is a technical condition and does not change the homotopy type of the simplicial set. The face and degeneracy maps are the obvious restrictions and pullbacks.

We let $H_{\textrm{sm}}(M)$ denote the $C^\infty$ space of $h$-cobordisms. The following lemma replaces Waldhausen's simplicial sets with these spaces.

\begin{lemma} \label{lem:GenFun:3}
  The map $|H(M)| \to H_{\textrm{sm}}(M)$ is a homotopy equivalence.
\end{lemma}

Here $|-|$ denotes geometric realization.

\begin{proof}
  This follows from standard smooth approximation theory together with a careful treatment of the conditions near the boundary.
\end{proof}

As the smooth versions are more directly related to our spaces we will work with these from now on. Let $H_{\textrm{sm}}'(M) \subset H_{\textrm{sm}}(M)$ be the subspace where $\partial P = \partial M \times \{\tfrac12\}$. 

\begin{lemma} \label{lem:Hatcher:9}
  The inclusion $H_{\textrm{sm}}'(M) \subset H_{\textrm{sm}}(M)$ is a homotopy equivalence.
\end{lemma}

\begin{proof}
  This is standard, so we omit the proof.
\end{proof}

Assume that $N$ is another smooth compact manifold and we are given a codimension 0 embedding $\inte M \subset N$. This induces an inclusion $H(M) \subset H(N)$. Indeed, this is given by extending any $P$ which equals $M \times \{a\}$ near $\partial M \times (0,1)$ with $(N\setminus \inte M)\times \{a\}$ outside $\inte M$. The following lemma is standard, and we have moved its proof to Appendix~\ref{sec:proof-stand-results}.

\begin{lemma} \label{lem:GenFun:2}
  If the inclusion $\inte M \to N$ is $n$-connected and $n \leq \dim M$ then the inclusion $H_{\textrm{sm}}(M) \subset H_{\textrm{sm}}(N)$ is at least $(\tfrac{n}{3}-5)$-connected.
\end{lemma}

We also need to consider a variant of this $h$-cobordism space. Indeed, assume $\partial M$ and $M$ are connected and let $H_{\partial}(M)$ be the $C^\infty$ space of connected codimension 1 submanifolds $P \subset \inte {M}$ (closed with no boundary) such that there is an $h$-cobordism $W \subset M$ with $\partial W = \partial M \sqcup P$. We have an inclusion
\begin{align*}
  H_{\textrm{sm}}(\partial M) \to H_\partial(M)
\end{align*}
induced by choosing a collar $c : \partial M\times [0,1) \to M$, which is a contractible choice.

\begin{lemma} \label{lem:GenFun:4}
  If the dimension of $M$ is at least 6 then the inclusion above is a homotopy equivalence.
\end{lemma}

Again, this is standard and we have moved the proof to Appendix~\ref{sec:proof-stand-results}.

Returning to our goal, the maps $c_k^l: \T_k^l \to BF(k)$ (discussed stably in the introduction to this section) are defined to be the classifying maps for the canonical based spherical fibration that has fiber $\{f \leq 1\}^+$ over $f\in\T_k^l$. As we did for $(F/O)_k^l$, we model the homotopy fiber $\hF_k^l$ of $c_k^l : \T_k^l \to BF(k)$ as the space $\hF_k^l$ of pairs $(f,\eta)$ where $f\in \T_k^l$ and $\eta : \{f \leq 1\}^+ \to S^k$ is a based homotopy equivalence. We again focus on the diagonal colimit where $k=l$ and lift the stabilization map $s : \hF_k \to \hF_{k+1}$ by noting that the inclusion
\begin{align*}
  (\R \times \{f \leq 1\} \times \{0\})^+ \subset \{s(f) \leq 1\}^+
\end{align*}
is a homotopy equivalence, and there is thus a contractible space of choices extending the homotopy equivalence
\begin{align*}
  (\id_\R)^+ \wedge \eta : (\R \times \{f \leq 1\} \times \{0\})^+ \to S^{k+1}
\end{align*}
to $\{s(f) \leq 1\}^+$. So we inductively define these stabilizations by making such choices.

\begin{remark}
  We will repeatedly use the following fact. Let $A \subset B$ be a cofibrant inclusion and a homotopy equivalence. Given a map defined on $A$, the space of its extensions to $B$ is contractible.
\end{remark}

To define the lift  $\th_k^l : \H_k^l \to \hF_k^l$ of the map $h_k^l$ we need the following lemma. Recall that $h_k^l(f) \leq f$.

\begin{lemma} \label{lem:Hatcher:3}
  The inclusion $\{f\leq 1\}^+ \subset \{h_k^l(f) \leq 1\}^+$ is a homotopy equivalence.
\end{lemma}

\begin{proof}
  Consider the global map of Serre fibrations:
  \begin{align*}
    \xymatrix{
      \bigcup_{f \in \H_k^l} \{f\leq 1\}^+ \ar[r]^{\subset} \ar[d] & \bigcup_{f \in \T_k^l} \{f\leq 1\}^+ \ar[d] \\
      \H_k^l \ar[r]^{h_k^l} & \T_k^l
    }
  \end{align*}
  We claim that this is an equivalence of spherical fibrations over the map $\H_k^l \to \T_k^l$ (the diagram is a homotopy pull back). Indeed, $\H_k^l$ and $\T_k^l$ are both connected and a direct inspection shows that the inclusion of the fibers over the base points $\{Q_k^l\leq 1\}^+ \subset \{\oQ_k^l \leq 1\}^+\simeq S^k$ is a homotopy equivalence.
\end{proof}

We define the lift $\th_k^l$ by
\begin{align} \label{eq:Hatcher:4}
 \th_k^l(f)=(h_k^l(f),\eta_f),
\end{align}
where $\eta_f: \{h_k^l(f) \leq 1\}^+ \to S^k$ is a choice of extension of $\pi_k^+ : \{f \leq 1\}^+ \to S^k$.

The homotopies $\alpha^k$ used in the previous subsection to define $h_\infty$ also admit lifts up to contractible choices. Indeed, they were defined using Lemma~\ref{lem:Hatcher:4}, which shows that the inclusion
\begin{align*}
  (\R \times \{f\leq 1\} \times \{0\})^+ \subset \{\alpha_t^+(f) \leq 1\}^+
\end{align*}
is defined for each $t$. At both $t=0$ and $t=1$, the above definitions provide extensions of $\pi_{k+1}^+$ to the larger space. It is again a contractible choice to extend these choices to all $t\in I$. This defines the colimit map
\begin{align}\label{eq:Hatcher:7}
  \th_\infty : \H_\infty \to \hF_\infty
\end{align}
from Equation~(\ref{eq:Hatcher:1b}).

\begin{proposition} \label{prop:Hatcher:10}
  The map $\th_\infty$ is a homotopy equivalence.
\end{proposition}

Before we can prove this we need a few lemmas. Let $\nH_k^l \subset \H_k^l$ be the subspace of functions for which $\{f\leq 1\}$ contains $\R^k\times\{0\}$ in its interior. Since $\pi_k : \{f\leq 1\}^+ \to S^k$ is a homotopy equivalence and restricts to the identity on $(\R^k)^+$, the inclusion $(\R^k)^+ \subset\{f \leq 1\}^+$ is a homotopy equivalence.

Let $\nHF_k^l \subset \hF_k^l$ be the subspace where $\{f\leq 1\}$ contains $\R^k\times\{0\} \subset \R^k\times \R^l$ in its interior and the homotopy equivalence $\eta : \{f\leq 1\}^+ \to S^k=(\R^k)^+$ restricts to the identity on $(\R^k)^+$.

We note that these subspaces are preserved by stabilizations, and we have commutative diagrams:
\begin{align} \label{eq:Hatcher:8}
  \xymatrix{
    \nH_k^l \ar[d]^{\th_{k\mid}^l} \ar[r]^{\subset} & \H_k^l \ar[d]^{\th_k^l} \\
    \nHF_k^l \ar[r]^{\subset} & \hF_k^l \\
  }
\end{align}

\begin{lemma}
  Both horizontal inclusions in the diagram above are at least $(l-k-1)$-connected.
\end{lemma}

\begin{proof}
  In the case of $\H_k^l$ we observe that $\{Q_k^l\leq c_f\}$ deformation retracts onto $\{f \leq 1\}$. Pick such a family of deformation retractions and compose the end of them with the standard inclusion $\R^k \to \{Q_k^l \leq c_f\}$. We thus get a family of maps
  \begin{align*}
    \oeta_f : \R^k \to \{f \leq 1\}
  \end{align*}
  such that each equals the standard inclusion outside a compact set and it equals the standard inclusion globally when $\R^k \subset \{f \leq 1\}$.

  For any map $(A,B) \to (\H_k^l,\nH_k^l)$ where $(A,B)$ is a pair of finite CW complexes of dimension at most $l-k-1$ we can perturb the associated family $\oeta_a, a\in A$ and assume that each $\oeta_a$ is a smooth embedding $\R^k \to \R^{k+l}$ that is standard outside a compact set. We may of course do so without changing that $\oeta_b$ is the standard inclusion for $b\in B$ and so that each $\oeta_a$ still equals the standard inclusion outside a compact set.
  
  We may even pick for each $a\in A$ a smooth isotopy through such embeddings $\oeta_a^t$ such that $\oeta_a^0=\oeta_a$ and $\oeta_a^1$ is the standard inclusion. We may again assume that this is constant on $B$ and outside a compact set. These isotopies can be extended to compactly supported isotopies of all of $\R^{k+l}$ (isotopy extension theorem). Applying these to the family $f_a$ defines a homotopy, relative to $B$, to a map sending all of $A$ to $\nH_k^l$, which proves the connectivity statement in the first case.

  In the second case, let again $(A,B) \to (\hF_k^l,\nHF_k^l)$ be a finite CW pair of dimension at most $l-k-1$. For each $a\in A$ we denote the image $(f_a,\eta_a) \in \hF_k^l$. It is a contractible choice to pick a family of based homotopy inverses $\oeta_a : S^k \to \{f_a \leq 1\}^+$ \emph{together with} a based homotopy $h_t^a : S^k \to S^k$ from $\eta_a \circ \oeta_a$ to the identity. So fix such a choice where $\oeta_b, b\in B$ is the standard inclusion and $h_t^b, b\in B$ is the constant homotopy.
  
  We may up to homotopy change the family $\oeta_a$ and assume that they are smooth \emph{based} embeddings $\oeta_a : S^k \to \{f_a \leq 1\}^+ \subset S^{k+l}$. Note that such a change to $\oeta_a$ requires also appending the homotopies $h_t^a$ with $\eta_a$ composed with the homotopy changing $\oeta_a$. If not, we would no longer have that the starting point of the homotopy $h_t^a$ were $\eta_a \circ \oeta_a$. We can then modify $\oeta_a$ near infinity (base point) so that each $\oeta_a$ is the standard inclusion near infinity. Indeed, the connectivity of the associated Stiefel manifold is above $l-1$. Both of these changes to $\oeta_a$ can be done without changing that $\oeta_b$ is already on this form for $b\in B$ (and that $h_t^b$ is the constant homotopy).

  We are now in a situation similar to the first case. Indeed, we can forget infinity and consider $\oeta_a' : \R^k \to \R^{k+l}$ the restriction of $\oeta_a$. Then we pick a compactly supported isotopy to the standard embedding (constant over $B$) and use the isotopy extension theorem to produce a homotopy of the function part $f_a, a\in A$. In this case we need to also change $\eta_a$ by pre-composing with the isotopy and we need to change $\oeta_a$ by post-composition with the inverse of the isotopy, which means that this time we need not make any changes to the homotopy $h_t^a$.

  We have thus reduced to the case where $S^k \subset \{f_a \leq 1\}^+$ and $\oeta_a$ is the standard inclusion, but $\eta_a$ is not exactly the identity on this image. Instead, we only have a homotopy $h_t^a :S^k \to  S^k$ from the composition $\eta_a \circ \oeta_a$ to the identity, and this homotopy is constant over $B$. However, using this homotopy and the homotopy extension property for $S^k \subset S^{k+l}$ we may construct a homotopy relative to $B$ so that $\eta_a$ becomes the identity on $S^k$.
\end{proof}

Consider the diagram
\begin{align*}
  \xymatrix{
    \H_k \ar[r] \ar[d]^{\th_k} & \H_k^l \ar[d]^{\th_k^l} \ar[r] & \H_l \ar[d]^{\th_l} \\
    \hF_k \ar[r] & \hF_k^l \ar[r] & \hF_l
  }
\end{align*}
where the left horizontal maps are positive stabilizations and the right horizontal maps are negative stabilizations. As this diagram commutes up to homotopy, it follows that, irrespective of the homotopies chosen to define $\th_\infty$, it is enough to show that the restriction to the sequence of subspaces (in Equation~(\ref{eq:Hatcher:8})) is an equivalence in the colimit. However, we note here that the homotopies constructed in Lemma~\ref{lem:Hatcher:4} and the lifts in Equation~(\ref{eq:Hatcher:4}) restrict to homotopies and lift on the left subspaces in Equation~(\ref{eq:Hatcher:8}). This means that there are no subtleties in defining the colimit map on these subspaces.

Define $\nT_k \subset \T_k$ as the subspace of maps where $\R^k\times\{0\} \subset \{f \leq 1\}$ and the inclusion $(\R^k)^+= (\R^k\times \{0\})^+ \subset \{f\leq 1\}^+$ is a homotopy equivalence. We then note that the forgetful map $\nHF_k \to \nT_k$ that forgets $\eta$ is a homotopy equivalence since we required that $\eta$ restricted to $(\R^k)^+$ was the identity. We therefore prove the following lemma.

\begin{lemma} \label{lem:Hatcher:8}
  The connectivity of the composition $c : \nH_k \xrightarrow{\th_k} \nHF_k \to \nT_k$ is at least $(\tfrac{k}{3}-5)$-connected.
\end{lemma}

\begin{proof}
  Let $\H \subset \nH_k \times [1,\infty)$ be the subspace of pairs $(f,b)$ where $b \geq c_f$ and $f=Q_k$ outside $D_b^{2k}$. Similarly, we define $\T \subset \nT_k \times [1,\infty)$ by $f=Q_{2k}^0$ outside $D_b^{2k}$. Using Lemma~\ref{lem:Gauss:1} it follows that the maps $\H \to \nH_k$ and $\T \to \nT_k$ are homotopy equivalences.

  For any $(f,b)\in \H$ we have that $f-Q_k + (b \se \oQ_k)$ equals $f$ on $D_b^{2k}$ and $b \se \oQ_k$ outside $D_b^{2k}$. It follows that
  \begin{align} \label{eq:Hatcher:9}
    c_{f-Q_k + (b \se \oQ_k)} = \max(c_f,c_{b \se \oQ_k}) \geq \max(c_f,64b) = 64b
  \end{align}
  since there is a point where $|\oQ_k-Q_k|=64$ and $b\geq c_f$. Using the bound in Equation~(\ref{eq:Tubes:2}) we also see that
  \begin{align*}
    c_{f-Q_k + (b \se \oQ_k)} \leq c_{b \se \oQ_k} \leq  b^2c_0
  \end{align*}
  where $c_0 = c_{\oQ_k}$.

  It follows that, up to homotopy we may lift $c$ to the map $C: \H \to \T$ given by
  \begin{align*}
    C(f,b) = (B_1(f-Q_k + (b \se \oQ_k)), \sqrt2 b^2c_0) 
  \end{align*}
  where $B_1$ and the factor of $\sqrt2$ come from Lemma~\ref{lem:Gauss:1}. Note that the lower bound in Equation~(\ref{eq:Hatcher:9}) shows that the bumping off happens outside of $D_{64b}^{2k}$ where the function equals $b \se \oQ_k$. As $\oQ_k$ is negative outside this ball, this region is far from the level set $\{f-Q_k + (b \se \oQ_k)=1\}$, which is helpful in the following explicit construction.
  
  Pick a diffeomorphism $\varphi : \R \times \R_+ \cong (-1,1) \times (0,1)$ which satisfies
  \begin{itemize}
  \item $\varphi(\{Q_1=1\} \cap \{y > 0\}) = (-1,1) \times \{\tfrac12\}$.
  \item $\varphi(-x,y)=(-\varphi_1(x,y), \varphi_2(x,y))$.  
  \end{itemize}
  Using this we define the diffeomorphism $\Phi : \R^k\times (\R^k \setminus \{0\}) \to (\inte D^k) \times S^{k-1} \times (0,1)$ given by
  \begin{align*}
    \Phi(x,y) =
    \begin{cases}
      \pare*{\tfrac{x}{\norm{x}}\varphi_1(\norm{x},\norm{y}),\tfrac{y}{\norm{y}},\varphi_2(\norm{x},\norm{y})} & x\neq 0 \\
      \pare*{0,\tfrac{y}{\norm{y}},\varphi_2(0,\norm{y})} & x= 0 \\
    \end{cases}
  \end{align*}
  This maps $\{Q_k=1\}$ to $(\inte D^k) \times S^{k-1} \times \{\tfrac12\}$. We thus use this to define $\iota_\H : \H \to  H_{\textrm{sm}}(D^k \times S^{k-1})$ by
  \begin{align*}
    \iota_\H(f,b) = \Phi(\{f=1\}) \cup ((\partial D^k) \times S^{k-1} \times \{\tfrac12\}).
  \end{align*}
  The image of $\iota_\H$ is the subspace $H_{\textrm{sm}}'(D^k \times S^{k-1})$, and over this subspace it is a fibration with contractible fibers. It follows by Lemma~\ref{lem:Hatcher:9} that $\iota_\H$ is a homotopy equivalence.

  To construct a similar map for $\T$ we consider the smooth embedding $\tau : [-1,1] \to S^1$ with image the lower half of $S^1$ given by
  \begin{align*}
    \tau(t) = (\cos(\tfrac12 \pi(t+3)),\sin(\tfrac12 \pi(t+3)))
  \end{align*}
  Notice that $\tau_1(-t)=-\tau_1(t)$. We use this to define the smooth rotationally symmetric embedding $\psi : D^k \times [0,1) \to D^{k+1}$ by
  \begin{align*}
    \psi(x,t) =
    \begin{cases}
      (1-t) \pare*{ \frac{x}{\norm{x}}\tau_1(\norm{x}), \tau_2(\norm{x}) } & x\neq 0 \\
      (0,t-1) & x=0
    \end{cases}
  \end{align*}
  whose image is the lower half-disk with respect to the last coordinate. We further use this to define $\Psi : D^k \times S^{k-1} \times (0,1) \to D^{k+1} \times S^{k-1}$ by
  \begin{align*}
    \Psi(x,y,t) = (\psi(x,t),y)
  \end{align*}
  which we in turn use to define the map $\iota_\T : \T \to H_{\partial}(D^{k+1} \times S^{k-1})$ given by
  \begin{align*}
    \iota_\T(f,b) = \Psi(\Phi(\{f=1\})). 
  \end{align*}
  The map is well-defined for $k\geq 3$. Indeed, first we note that the assumption that $(\R^k)^+ \subset \{f\leq 1\}^+$ is a homotopy equivalence implies that the inclusion
  \begin{align*}
    \Psi(\Phi(\{f \geq 1\})) \subset D^{k+1} \times S^{k-1} \simeq S^{k-1}
  \end{align*}
  is a homotopy equivalence. This, in turn, implies that the cobordism between the boundary and the image $\iota_\T(f,b)$ is an $h$-cobordism.

  The map $\iota_\T$ is also a homotopy equivalence. Indeed, the embedding $\Psi$ into the interior of $D^{k+1}\times S^{k-1}$ is isotopic through embeddings to a diffeomorphism. This implies that after modifying $\iota_\T$ by such an isotopy it is a surjective fibration onto $H_{\partial}(D^{k+1} \times S^{k-1})$ and, as for $\iota_\H$, the fibers are contractible.
  
  We now consider the diagram
  \begin{align*}
    \xymatrix{
      \H \ar[rr]^{C}\ar[d]^{\iota_\H}_\simeq && \T \ar[d]^{\iota_\T}_\simeq \\
      H_{\textrm{sm}}(D^k \times S^{k-1}) \ar[r] & H_{\textrm{sm}}(S^k \times S^{k-1}) \ar[r] & H_\partial(D^{k+1} \times S^{k-1})
    }
  \end{align*}
  Here the lower left map is induced by the inclusion $\psi_{\mid D^k \times \{0\}} \times \id_{S^{k-1}}$. The lower right map is induced by the collar given by $(x,y,t) \mapsto ((1-t)x,y)$. The composition of these lower maps is $(\tfrac{k}{3}-5)$-connected by Lemma~\ref{lem:GenFun:2} and Lemma~\ref{lem:GenFun:4}.

  We finish the proof by proving that the diagram commutes up to homotopy. In Figure~\ref{Fig:waysaround} we illustrated the image of the maps obtained by going around the diagram in the case $k=1$, restricted to the upper half-plane $\R \times \R_+$.
  \begin{figure}[ht]
    \begin{tikzpicture}
      \draw[dotted] (-1,0) -- (1,0);
      \draw[green,thick] (-30:0.5) arc (-30:210:0.5);
      \draw[red] (-30:0.5) to[out=60,in=0] (0.3,-0.1) to[out=180,in=0] (0,-0.13) to[out=180,in=0] (-0.3,-0.1) to[out=180,in=120] (210:0.5);
      \fill[yellow] (0.8,-0.6) to[out=135,in=-20] (-30:0.5) to[out=160,in=20] (210:0.5) to[out=200,in=45] (-0.8,-0.6) arc (-143:-37:1);
      \draw (0,0) circle (1);
      \draw (0,-0.5) node {?};
    \end{tikzpicture}
    \caption{The two ways around the diagram} \label{Fig:waysaround}
  \end{figure}
  Here the yellow part is the image of $D_b^2$ and the codimension 1 manifold inside this area depends on the function. The part of the image manifold inside the yellow region does not depend on which way we go around the diagram. The green continuation is what we get going counterclockwise around the diagram while the red is the result of going clockwise around. The size of the yellow area and the placement of the red part depend on $b$. However, increasing $b$ only brings the red part close to the dotted line. We thus pick a family of isotopies from the red lines to the green line always staying on the side not meeting the yellow areas. We may choose this to be symmetric with respect to the sign action on the $x$-coordinate. The above maps used in the constructions are $O(k) \times O(k)$ equivariant (acting on $(x,y)$) in such a way that we can take these chosen isotopies in the case $k=1$ and $y \geq 0$ and upgrade them to the general case of $k$.
\end{proof}

\begin{proof}[Proof of Proposition~\ref{prop:Hatcher:10}]
  The two lemmas above, together with the intervening discussion, prove this.
\end{proof}

\subsection{Comparison to Rognes' construction}
In this subsection, we prove that our rigid tube map
\begin{align*}
  \rt_\infty : BO \to \T_\infty
\end{align*}
is equivalent to that of Rognes and Waldhausen. We also lift the rigid tube map to define the Hatcher--Waldhausen map
\begin{align*}
  \hw : F/O \to \hF_\infty,
\end{align*}
from the diagram in Equation~(\ref{eq:Hatcher:1}). We then check that the lifted diagram commutes up to homotopy.

Waldhausen's rigid tube map was reformulated by Rognes in \cite{MR1282230}. We start by replacing our function space $\T_k^l$ with a $C^\infty$ space of submanifolds in $\R^{k+l}$. We then describe how this space relates to Waldhausen's and Rognes' models, and prove the homotopy coherence needed to identify our map as equivalent to their map.

Again, both Waldhausen and Rognes use simplicial sets to define their spaces, but as in Lemma~\ref{lem:GenFun:3}, these can easily be replaced by $C^\infty$ spaces. They also allow their manifolds to have singularities, but in a way that makes them smoothable by contractible choices (see the appendix of \cite{MR686115} for more details). We will mostly consider the equivalent subspaces where the manifolds are actually smooth. However, as their stabilizations use this feature we will need to consider this smoothing in a few cases.

Let $\cR_k^l$ denote the $C^\infty$ space of smooth compact submanifolds $M \subset \R^{k+l}$ isotopic to $T_k^l$. We let $M_+$ denote the bounded smooth codimension 0 manifold with boundary $M$, and we let $M_-$ denote the unbounded manifold with boundary $M$. We have a Serre fibration that is also a homotopy equivalence
\begin{align*}
  \T_k^l \to \cR_k^l \qquad \textrm{given by} \qquad f \mapsto \{f=1\}
\end{align*}
such that $M_- = \{f\leq 1\}$ and $M_+=\{f\geq 1\}$. Indeed, for each $M \in \cR_k^l$, the corresponding fiber is convex.

Waldhausen describes two stabilizations, denoted $\osig$ and $\usig$, corresponding to our $s_-$ and $s_+$ respectively. Rognes reinterprets these in his model, which is closer to our spaces $\cR_k^l$. In our case, following Rognes, the stabilization $\osig : \cR_k^l \to \cR_{k+1}^l$ is defined by crossing $M_+$ with a narrow interval $[-\epsilon,\epsilon]$ and taking a smoothing of its boundary; this new codimension $0$ manifold in $\R^{k+l+1}$ with boundary is $\osig(M)_+$. The map $\usig : \cR_k^l \to \cR_k^{l+1}$ is similarly given by crossing $M_-$ with $[-\epsilon,\epsilon]$ but also adding a neighborhood of $\infty$ given by $\norm{z}\geq R$ for $R$ large enough (depending on $M$) and then taking a smoothing of the boundary to get $\usig(M)_-$.

\begin{lemma}
  There are choices of Waldhausen's stabilization maps $\usig$ and $\osig$, and of sections $\cR_k^l \to \T_k^l$, such that the sections commute with stabilizations.
\end{lemma}

\begin{proof}
  We denote the image of the sections $g_M\in \T_k^l$ for $M\in \cR_k^l$ and we shorten $c_M=c_{g_M}$. Using such sections we may define maps on $\cR_k^l$ by
  \begin{align*}
    \sigma_-(M) = \{- x_{k+1}^2 + g_M =1\} \quad \textrm{and} \quad \sigma_+(M) = \{g_M + (c_M\se\oq)(y_{l+1})=1\}.
  \end{align*}
  We first claim that we can make the choices of all these $g_M$ so that $\sigma_- \circ \sigma_+ = \sigma_+ \circ \sigma_-$.
  
  Observe that for any choice of $g_M$ on all of $\cR_k^l$ the maps $\sigma_\pm$ are cofibrations. Indeed, we recover $M$ by intersecting the new manifold $\sigma_{\pm}(M)$ with $\R^{k+l}$, so it is injective. An open neighborhood of the image is still transverse to $\R^{k+l}$ so mapping to the intersection with $\R^{k+l}$ and stabilizing again provides a retraction of such a neighborhood onto the image.

  We first define $g_M$ on $\cR^l_0$ inductively in $l$. Indeed, we do this by extending the definition $g_{\sigma_+(M)}=g_M + (c_M\se\oq)(y_{l+1})$ inductively in $l$. Note that $c_{\sigma_+(M)} > c_M$. We then proceed to define $g_M$ on $\cR_1^l$ again inductively in $l$ by making sure that $g_{\sigma_-(M)} = - x_1^2 + g_M$ for $M \in \cR_0^{l}$ \emph{and} that $g_{\sigma_+(M)}=g_M + (c_M\se\oq)(y_{l+1})$ for $M\in \cR_1^{l-1}$. Note that these are not contradictory since $c_{\sigma_-(M)}=c_M$ and hence we may similarly continue inductively in $k$, by using induction in $l$ for each $k$ to define such extension.
  
  We then prove that there are choices of Waldhausen's stabilization maps $\osig$ and $\usig$ so that they are in fact equal to $\sigma_-$ and $\sigma_+$ respectively.
 
  To get $\usig=\sigma_+$ we consider that in the definition of $\usig(M)$ we smooth the boundary of
  \begin{align*}
    W = M_-\times [-\epsilon,\epsilon] \cup \{\norm{z} \geq R\}
  \end{align*}
  \begin{figure}[ht]
    \begin{center}
      \begin{tikzpicture}[scale=0.5]
        \fill[black!14!white] (-5,-3.7) -- (-5,3.7) -- (5,3.7) -- (5,-3.7) -- cycle;
        \fill[white] (0,0) circle (3.5cm);
        \draw[line width=0.6cm,black!14!white] (-5,0) -- (-1,0);
        \draw[line width=0.6cm,black!14!white] (-0.5,0) -- (0.5,0);
        \draw[line width=0.6cm,black!14!white] (1,0) -- (5,0);
        \draw[very thick] (-5,0) -- (-1,0);
        \draw[very thick] (-0.5,0) -- (0.5,0);
        \draw[very thick] (1,0) -- (5,0);
        \def\mypath{ (-1,0) to[out=90,in=0] (-1.6,0.5) to[out=180,in=270] (-2.3,1.4) to[out=90,in=90] (2.3,1.4) to[out=270,in=0] (1.6,0.5) to[out=180,in=90] (1,0) };
        \draw[green, very thick] \mypath;
        \begin{scope}[rotate=180]
          \draw[green, thick] \mypath;         
        \end{scope}
        \draw[green, thick] (-0.5,0) to[out=90,in=90] (0.5,0) to[out=270,in=270] cycle;
        \draw[red, very thick] (-5,1.4) -- (5,1.4);
        \draw[red, very thick] (-5,-1.4) -- (5,-1.4);
        \draw[very thick] (6,2.5) -- (6.8,2.5) node[right] {$M_-$};
        \draw[green, very thick] (6,1.2) -- (6.8,1.2);
        \draw (6.8,1.2) node[right] {$\sigma_+(M)=\{g^+=1\}$};
        \draw[red, very thick] (6,-0.1) -- (6.8,-0.1);
        \draw (6.8,-0.1) node[right] {$|y_{l+1}|=6c_M$};
        \fill[black!14!white] (6,-1.2) -- (6.8,-1.2) -- (6.8,-1.9) -- (6,-1.9) -- cycle;
        \draw (6.8,-1.7) node[right] {$W$};
      \end{tikzpicture}
      \caption{$W$ and level set.} \label{Fig:Wlevel}
    \end{center}
  \end{figure}
  for some $\epsilon>0$ and some large $R \gg \epsilon$. To smooth the compact boundary with corners of a smooth manifold $W\subset \R^{k+l+1}$, we first make the contractible choice of a smooth vector field $Y$ transverse to $\partial W$. If we let $U$ denote the open neighborhood defined by applying the flow of $Y$ for all time to the topological manifold $\partial W$, we get an identification of this manifold with the smooth leaves $U/Y$. Hence $\partial W$ gets a smooth structure. We then make the contractible choice of a smooth section $\partial W \to U$ and its image is then the smoothing of $\partial W$.

  Let $g^+=g_M+(c_M\se\oq)(\norm{y_{l+1}})$. For $R$ very large and $\epsilon < 1$ the gradient $\nabla g^+$ is transverse to $\partial W$ and points out of $W$ (into the bounded region $W^c$, which is white in Figure~\ref{Fig:Wlevel}). Let $U$ denote the image of the flow of $\nabla g^+$ for all time on $\partial W$. We claim that for possibly larger $R$ and smaller $\epsilon$ there is a unique section $s:\partial W \to U$ such that the image is exactly $\sigma_+(M)$, making $\sigma_+(M)$ this particular smoothing of $\partial W$.

  To prove this claim we note that for large enough $R$ we have that $g^+(z)<0$ for $\norm{z}\geq R$, and it follows that at all points in $\partial W$ we have $g^+(z)<1+\epsilon^2$. By the lemma above we have that all critical points of $g^+$ with value below 1 lie in the interior of $W$. We may even assume that $\epsilon$ is so small that $[1,1+\epsilon^2]$ is regular for $g_M$ and $g^+$. It follows that for each $z\in \partial W$ there is a unique time for which the gradient flow takes $z$ to a point where $g^+=1$. Indeed, for $g^+(z) \in [1,1+\epsilon^2]$ we can flow backwards with no critical points in the way, and for $g^+(z)<1$ the positive gradient flow takes us into $W^c$ where we also do not meet any critical points until we reach $\{g^+=1\}$. Similarly, to prove that the corresponding section surjects onto $\sigma_+(M)$ we must argue that any point in $\{g^+=1\}$ can be taken to a point in $\partial W$ by the flow. We also divide this argument into two cases.

  Case 1: $z\in W$ implies that $z \in M_- \times [-\epsilon,\epsilon]$. Indeed, if $\norm{z} \geq R$ we would have $g^+(z)<1$. If $y_{l+1}=0$ it means that $g_M=g^+=1$ and thus $z\in \partial W$. If $y_{l+1}\neq 0$ the gradient flow will increase $\absv{y_{l+1}}$ and it will eventually reach the boundary of $M_-\times [-\epsilon,\epsilon]$.

  Case 2: $z\notin W$ means that we can use the negative gradient flow to get to $\partial W$. Indeed, this again follows from the fact that all critical points with value below 1 are inside $W$.

  The argument for $\osig=\sigma_-$ using the gradient of $g^-=-x_{k+1}^2+g_M$ and a unique section to smooth the boundary of $W=M_+\times [-\epsilon,\epsilon]$ is similar yet easier.
\end{proof}

\begin{corollary}
  The colimit $\T_\infty = \colim_{k,l\to\infty} \T_k^l$ is equivalent to Waldhausen's tube space, and the map $\rt_\infty : BO \to \T_\infty$ is equivalent to Waldhausen's and Rognes' rigid tube map.
\end{corollary}

\begin{proof}
  We have already related the spaces and stabilization maps, so we proceed to identify the rigid tube map itself. Rognes' rigid tube map is defined to be a smoothing of the boundary of $V \times D_\epsilon V^\perp \cup \{\norm{z} \geq R\}$ for some small $\epsilon>0$ and large $R>0$. The following more general construction takes us closer to our definition of the map. Pick for each $V \in \Gr_k(\R^{k+l})$ a tubular neighborhood embedding $\tau : V\times DV^\perp \to \R^{k+l}$ which is tangent to the foliation by spheres $S_r^{k+l}$ for $r\geq R>0$. Then for small $\epsilon>0$ the boundary of $\tau(V\times D_\epsilon V^\perp)$ is transverse to $\{\norm{z} = R\}$ and can be smoothed.

  We precisely picked such tubes $\tau=\tau_{V,i_V,1}$ using Lemma~\ref{lemcor:Parhandles:1} when defining our rigid tube map as $\MP(V,i_V,1,1)$. However, in the map to $\cR_k^l$, where Rognes' tubes live, this gives $\{\MP(V,i_V,1,1)=1\}$ and not the image of a small $\epsilon$-tube boundary. We relate these two exactly as in the proof above. Indeed, in Lemma~\ref{lem:FmodO:7} we saw that $D\tau_*(\nabla Q_V)$ is a strict pseudo-gradient for $\MP(V,i_V,1,1)$ and this is outward-pointing on the boundary of $\{\norm{z} \geq 1\} \cup \tau(V\times D_\epsilon V^\perp)$. It is thus a contractible choice to extend this to a strict pseudo-gradient $X$ which is also inward-pointing on $S_R^{k+l}$. The identification then works exactly as in the proof above using the flow of $X$. 
  
  Rognes defines coherence with stabilization by simply noting that such $\epsilon$-neighborhoods are contractible choices, and the stabilizations are obtained by adding a new direction to such neighborhoods. This statement is equivalent to our coherence homotopies constructed in Lemma~\ref{lem:Hatcher:5}. Indeed, in the proof of that lemma we used that the unstable disk map was given throughout by $i_{\R \oplus V}$, and the mountain pass construction formalizes the contractible choice of such neighborhoods in our slightly more general setup of Rognes' construction. 
\end{proof}

As the inclusion of $V^+ \subset \{\rt_k^l (V)\leq 1\}^+$ is a homotopy equivalence, we have an essentially canonical homotopy between the composition $c_k^l \circ \rt_k^l$ and the delooping of the $J$-homomorphism $BJ_k^l : \Gr_k(\R^{k+l}) \to BF(k)$. This implies that the Hatcher--Waldhausen map $\hw_k^l : (F/O)_k^l \to \hF_k^l$ at level $(k,l)$ can be realized explicitly in our models by
\begin{align*}
  \hw_k^l(V,\theta) = (\rt_k^l(V),\theta_V)
\end{align*}
where $\theta_V : \{\rt_k^l(V)\leq 1\}^+ \to S^k$ is some choice extending $\theta$ to the sublevel set. As these inclusions are homotopy equivalences throughout the homotopies constructed in Lemma~\ref{lem:Hatcher:5} we get a contractible choice lift of the homotopies $\beta_t^\pm$ to similar homotopies for $\hw_k^l$. The colimit map $\hw_\infty$ thus matches Waldhausen's definition of the Hatcher--Waldhausen map.

\begin{lemma} \label{lem:Hatcher:1oldcorr}
  The diagram in Equation~(\ref{eq:Hatcher:1b}) commutes up to homotopy.
\end{lemma}

\begin{proof}
  We have already produced most of the pieces of the needed homotopy in the proof of Lemma~\ref{lem:Hatcher:11}. Indeed, in the language of that proof, all we need is to argue that the spherical trivializations to $S^k$ given over $\UN_k\times (\partial I^2)$ extend to $\UN_k \times I^2$. The argument in that proof used that the maps
  \begin{align*}
    \R \times \un(X)_{\min(2u,1)} : \R \times N_k^l(f) \to \R^{k+l+1}
  \end{align*}
  included into the relevant sublevel sets are unstable disk maps for the relevant functions for each $(t,u) \in I^2$. By Lemma~\ref{lem:FmodO:3} this means that their one-point compactifications are homotopy equivalences. On these subsets we now consider the spherical trivializations given by the one-point compactifications of
  \begin{align*}
    \R \times \im \un(X)_s \xrightarrow{\id_\R \times \un(X)_s^{-1}} \R \times N_k^l(f) \xrightarrow{\pi_{k+1}^+ \circ (\id_\R \times \un(X))} S^{k+1},
  \end{align*}
  where $s=\min(2u,1)$. The spherical trivializations already given over $\UN_k \times (\partial I^2)$ are by definition extensions of these to the relevant sublevel sets. Since the space of such extensions is contractible we may extend over $I^2$.
\end{proof}

\begin{proof}[Proof of Theorem~\ref{thm:3}]
  In Proposition~\ref{prop:oldcor:Tubes:1} we established a fibration sequence involving $\M_\infty$ as the fiber. In Proposition~\ref{prop:Hatcher:1} we saw that the inclusion $\U_\infty \subset \H_\infty$ is equivalent to the map $w_\infty : \tU_\infty \to \H_\infty$. The diagram in Equation~(\ref{eq:Hatcher:1b}), which was constructed in this section and which commutes up to homotopy, then relates this to the Hatcher--Waldhausen map.
\end{proof}
